\documentclass[11pt, oneside]{article}   	
\usepackage[totalwidth=480pt,totalheight=680pt]{geometry}                		
\usepackage[british]{babel}

\usepackage[usenames,dvipsnames]{color}
\usepackage[colorlinks,citecolor=BlueViolet,linkcolor=Brown]{hyperref}
\usepackage{graphicx}				
\usepackage[numbers,sort]{natbib}

\usepackage{amssymb}
\usepackage{amsmath}
\usepackage{amsthm}
\usepackage{mathtools}
\usepackage{enumerate}
\numberwithin{equation}{section}

\renewcommand{\u}{\mathbf{u}}

\newcommand{{\LPC}}{\textbf{LPC}}

\renewcommand{\L}{\mathcal{L}}
\newcommand{\Reals}{\mathbb{R}}

\newcommand{\ball}{\operatorname{\mathbb{B}}}

\newcommand{\eps}{\varepsilon}

\newcommand{\bX}{\operatorname{\bar{X}}}

\newcommand{\bp}{\operatorname{\bar{p}}}

\newcommand{\Om}{\Omega}
\newcommand{\clv}{\mathcal{V}}
\newcommand{\cla}{\mathcal{A}}
\newcommand{\clb}{\mathcal{B}}
\newcommand{\clf}{\mathcal{F}}
\newcommand{\cle}{\mathcal{E}}
\newcommand{\clm}{\mathcal{M}}
\newcommand{\clg}{\mathcal{G}}
\newcommand{\clp}{\mathcal{P}}
\newcommand{\PP}{\mathbb{P}}
\newcommand{\QQ}{\mathbb{Q}}
\newcommand{\RR}{\mathbb{R}}
\newcommand{\EE}{\mathbb{E}}
\newcommand{\NN}{\mathbb{N}}
\newcommand{\lan}{\langle}
\newcommand{\ran}{\rangle}
\newcommand{\bfB}{\mathbf{B}}
\newcommand{\sys}{\mathfrak{S}}
\newcommand{\om}{\omega}

\newtheorem{thm}{Theorem}[section]
\newtheorem{lemma}[thm]{Lemma}

\newtheorem{rem}[thm]{Remark}

\newtheorem{proposition}[thm]{Proposition}

\newtheorem{lem}[thm]{Lemma}
\begin{document}
\title{Large Deviations from the Hydrodynamic Limit for a System with Nearest Neighbor Interactions}
\author{Sayan Banerjee, Amarjit Budhiraja\thanks{%
Research supported in part by the National Science
Foundation (DMS-1305120) and the Army Research Office (W911NF-14-1-0331) }, and Michael Perlmutter}
\date{}

\maketitle
\begin{abstract}
We give a new proof of the large deviation principle from the hydrodynamic limit for the Ginzberg-Landau model studied in  \cite{donvar5} using techniques from the theory of stochastic control and weak convergence methods. The proof is based on characterizing subsequential hydrodynamic limits of controlled diffusions with nearest neighbor interaction that arise from a variational representation of certain Laplace functionals. The approach taken here does not  require superexponential probability estimates, estimation of exponential moments, or an analysis of eigenvalue problems, that are central ingredients in previous proofs. Instead, proof techniques are very similar to those used for the law of large number analysis, namely in the proof of convergence to the hydrodynamic limit (cf. \cite{guopapvar}). Specifically, the key step in the proof is establishing suitable bounds on relative entropies and Dirichlet forms associated with certain controlled laws. This general approach has the promise to be applicable to other interacting particle systems as well and to the case of non-equilibrium starting configurations, and to infinite volume systems.


\end{abstract}

\section{Introduction and Notation} 
We consider the Ginzberg-Landau model in finite volume, namely the following system of interacting diffusions in $\Reals^N$:
\begin{align}\label{align:uncontroleq}
dX^{N}_i(t)&= dZ^{N}_i(t) - dZ^{N}_{i+1}(t),\nonumber\\
dZ^{N}_i(t)&= \frac{N^2}{2}\left[\phi'\left(X^{N}_{i-1}(t)\right) - \phi'\left(X^{N}_{i}(t)\right) \right]dt + NdB_i(t)
\end{align}
on some finite time horizon $0\leq t\leq T$ for $1 \leq i \leq N$.  The random variable $X_i^N(t)$ is thought of as the amount of charge at the site $i/N$ on the periodic lattice $\{1/N,\ldots,(N-1)/N,1\},$ and we identify $Z^N_{N+1}$ and $X^N_{N+1}$ with $Z^N_1$ and $X^N_1$ respectively. 
Here $\{B_i(t)\}_{i=1}^\infty$ are independent standard one-dimensional Brownian motions given on some probability space
$(\clv, \clf, \PP)$ and $\phi:\mathbb{R}\rightarrow\mathbb{R}$ is a twice continuously differentiable function such that 
\begin{align}
\int_{\mathbb{R}}\exp(-\phi(x))dx &= 1, \nonumber\\
\label{genfun} M(\lambda) \doteq \int_\mathbb{R} \exp(\lambda x - \phi(x)) &< \infty \quad \text{for all }\lambda\in\mathbb{R},
\end{align}
and 
\begin{align}\label{sigmaass}
\int_{\mathbb{R}} \exp(\sigma|\phi'(x)|-\phi(x))dx &<\infty \quad \text{for all }\sigma>0 .
\end{align}
The process $X^N = (X^N_i)_{i=1}^N$ is a $\RR^N$-valued Markov process with  generator given by
\begin{equation}\label{eq:defln}
\mathcal{L}^N\doteq 
\frac{N^2}{2} \sum_{i=1}^N V_i^2 - \frac{N^2}{2}\sum_{i=1}^N [\phi'(x_{i})-\phi'(x_{i+1})]V_i,
\end{equation}
where $V_i = \partial_{i}-\partial_{i+1}$ and $\partial_{i}$ denotes the partial derivative with respect to $x_i$. Let $\Phi$ be the probability measure on $\mathbb{R}$ defined by $\Phi(dx) \doteq e^{-\phi(x)}dx$, 
and let $\Phi^N$ be the measure on $\mathbb{R}^N$ defined by $\Phi^N(dx) \doteq \Phi(dx_1)\Phi(dx_2)\ldots\Phi(dx_n).$  One may check via integration by parts that $\mathcal{L}^N$ is a symmetric operator on $L^2(\mathbb{R}^N,\Phi)$ and that therefore, $\Phi^N$ defines an invariant measure for the  diffusion $X^N$. Throughout this work, $X^N$ will be the stationary process obtained by taking $X^N(0)$ distributed according to  $\Phi^N$.

Associated with the collection $(X^{N}_i(t))_{i=1}^N$ for $t\ge 0$, consider  the signed measure on the circle $S$ (namely  the interval $[0,1]$ with its end points identified), defined by 
\begin{equation}
\label{empmeas}\mu^N(t,d\theta) \doteq \frac{1}{N}\sum_{i=1}^N X_i^{N}(t) \delta_{i/N}(d\theta).
\end{equation}
The goal of this work is to establish a large deviation principle for the stochastic process $\{\mu^N(t)\}$ that takes values in the space $\mathcal M_S$ of signed measures   on $S$.  \\

Hydrodynamic limits for the sequence of  signed measure valued stochastic processes given by \eqref{empmeas} were first investigated in the seminal work of \cite{guopapvar} using techniques based on estimates on relative entropies and Dirichlet forms (governing the rate of change of relative entropies). A subsequent paper \cite{donvar5} laid the mathematical foundations of the large deviation theory for such interacting particle systems. The methods developed in \cite{donvar5}  for the large deviation analysis have been used and extended in a 
variety of interacting particle system settings such as the nongradient Ginzberg-Landau model \cite{quastel}, the Ginzberg-Landau $\nabla \phi$-interface model \cite{funaki}, the infinite volume versions of the Ginzberg-Landau model and the zero range processes \cite{landimyau} \cite{benois}, the weakly asymmetric simple exclusion process \cite{kipnisollavar}, the symmetric exclusion process in dimension at least three \cite{QRV}  and interacting spin systems \cite{pra}, to name a few. The analysis in all these works proceeds via a precise control of moments for exponential martingales. The key ingredient  is a superexponential estimate (see, for example, Theorem 2.2 of \cite{landimyau}) that is used to replace the correlation fields appearing in the exponential martingales by suitable functions of the density field. 
One of the challenges in obtaining
such superexponential probability estimates is that they exploit reversibility in a crucial way and applicability of methods based on such estimates is somewhat limited for non-equilibrium initial distributions or non-reversible generators. Furthermore, in the infinite volume case, if the starting distribution does not have finite entropy with respect to the stationary measure, it becomes a considerably challenging problem to derive superexponential estimates by the standard method of comparison with the equilibrium model and solving an eigenvalue problem for the latter model. In  general, superexponential probability estimates are the most technical parts of the  large deviation proofs for such systems.
We note however that such estimates have been established for some infinite volume and non-equilibrium settings (see \cite{landimyau}, \cite{benois}) using the so-called ``one-block" and ``two-block" estimates. 
In other works, large deviation problems for some weakly asymmetric models have been addressed via model-specific computations \cite{enaud,mathews}.

In this work we give a new proof of the large deviation principle originally obtained in \cite{donvar5}.  The proof technique is very different from all the references listed above. The central ingredient in our approach are certain stochastic control representations and weak convergence techniques.  The latter are very similar to those developed for the proof of the law of large numbers in \cite{guopapvar} (see for example the proofs of Lemma \ref{mubartight} and Theorem \ref{HL}). These techniques allow us to prove tightness of certain controlled processes and to characterize the weak subsequential limits. In contrast to the proof in \cite{donvar5} (see e.g. Lemmas 2.1, 2.2, 2.3, 2.7 and Theorem 2.5 therein), no superexponential probability estimates or exponential moment bounds are invoked in this method. 
In particular, analysis of eigenvalue problems of the form in previous works (cf. \cite[Lemma 2.2]{donvar5}) is not needed in this approach.
The starting point of  our proof is the Bryc-Varadhan equivalence between  the Laplace principle and the large deviations principle for random variables taking values in a Polish space (see the discussion leading upto \eqref{LP}). Using a stochastic control representation for exponential functionals of Brownian motions (\cite{boudup}, see also \cite{budfanwu} and Lemma \ref{lem:repnmod} in this work), this equivalence reduces the problem of large deviations to the study of asymptotics of costs associated with certain controlled stochastic processes. Characterization of the limits of the controlled processes and the costs relies on a qualitative understanding of properties such as existence, uniqueness, and continuity (in the control) of solutions of certain controlled analogues of the hydrodynamic limit PDE associated with the system (see Lemmas \ref{existence} and \ref{continu}). We note that hydrodynamic limits of certain `mildly perturbed systems'
are studied in \cite{donvar5}, however the form of the perturbations and the role they play in the analysis is somewhat different. In particular, the perturbations analyzed in \cite{donvar5}  are non-random and they appear only in the proof of the lower bound. In contrast, the controlled systems studied in the current work correspond to random perturbations and are central ingredients in the proofs of both the upper and the lower bound.

The general framework of the proof suggests that for any given system of interacting particles the large deviation analysis hinges on a good understanding of the associated hydrodynamic limit theory (i.e. law of large number behavior). In particular, for the current setting, the weak convergence arguments that allow the characterization of costs and controlled processes in the stochastic control representations rely on  similar estimates, on relative entropies and Dirichlet forms associated with probability densities of the controlled processes, that form the basis of the hydrodynamic limit proof
in \cite{guopapvar} for the uncontrolled system. Obtaining these estimates, which are relatively straightforward for the uncontrolled system, is the most demanding part of the proof. One key technical step  in getting these estimates (Lemma \ref{reglemma}) is establishing suitable regularity of the  density of the controlled process. Although many steps in the proof of Lemma  \ref{reglemma} are classical in PDE literature, we have provided a full proof for keeping the presentation self-contained. 
This lemma is crucial in the proof of Lemma \ref{Ilemma} which relies on an application of It\^{o}'s formula.
%
%
Despite the demanding proofs of Lemmas \ref{reglemma} and \ref{Ilemma}, the general method of proof of the large deviation principle seems
 quite robust and many ingredients in the proof carryover to nonreversible systems and  to the infinite volume case as well. Moreover, the method is applicable for a broad class of initial distributions (which could be very different from the equilibrium measure). The proof framework developed in the current work will be a starting point for the study of these more general settings and will be taken up in future work.
\\

\noindent {\em Notation.}  We will use the following notation. For a Polish space $\cle$, $\clp(\cle)$ will denote the space of probability measures on $\cle$ which will be equipped with the topology of weak convergence; $C(\cle)$ will denote the space of real valued continuous functions on $\cle$; and $C([0,T]:\cle)$ will denote the space of continuous functions from
$[0,T]$ to $\cle$ equipped with the topology of uniform convergence. A collection of $\clp(\cle)$-valued random variables will be called tight if their probability distributions form a relatively compact collection in $\clp(\clp(\cle))$.
We will denote by $L^2([0,T]:\RR^N)$ and $L^2([0,T]\times S)$ the Hilbert spaces of square integrable functions from
$[0,T]$ (resp. $[0,T]\times S$) to $\RR^N$ (resp. $\RR$). 
Given two probability measures $\gamma, \theta$ on some measurable space, the relative entropy of  $\gamma$ with respect to  $\theta$  will be denoted as $R(\gamma \|\theta)$. For any subset $A$ in the sigma-algebra of a measurable space $\mathcal{R}$, $\mathbb{I}_A: \mathcal{R} \rightarrow \mathbb{R}$ will denote the indicator function which takes value one on $A$ and zero on the complement of $A$. We will denote by $\kappa, \kappa_1, \kappa_2, \ldots$ generic finite constants that appear in the course of a proof. The values of these constants may change from one proof to the next.

In order to give a precise statement of the result we begin by discussing the topology on the space $\clm_S$ and on the space of $\clm_S$-valued continuous paths.

\subsection{Topology on the Space of Signed Measures}
\label{sec:topsign}
The space $\clm_S$ equipped with the topology of weak convergence is not metrizable and therefore this topology is not convenient to work with. Instead we proceed as in \cite{guopapvar}. Consider the spaces $\{\clm_S^l\}_{l\in \NN}$, where
$\clm_S^l$ is the space of signed measures on $S$ with total variation bounded by $l$, namely $\clm_S^l$ consists of
$\gamma \in \clm_S$ such that
\begin{equation}
	\|\gamma\|_{TV} \doteq \sup_{f \in B_1(S)} \langle \gamma , f \rangle \le l,
\end{equation}
where $B_1(S)$ is the space of real functions on $S$ with $\|f\|_{\infty}\doteq \sup_{\theta \in S} |f(\theta)| \le 1$
and for a signed measure $\gamma$ and a bounded real function $f$ on $S$, 
$\langle \gamma , f \rangle \doteq \int_S f(\theta) \gamma(d\theta)$. Note that
$\clm_S = \cup_{l\in \NN} \clm_S^l$. The space $\clm_S^l$ equipped with the topology of weak convergence is a Polish space
and one convenient metric (see Lemma \ref{lem:topprops}) on this space is the {\em bounded-Lipschitz distance} defined as
\begin{equation*}
d_{BL}(\gamma_1,\gamma_2) \doteq \sup_{f \in BL_{1}(S)} |\langle \gamma_1-\gamma_2,f\rangle|, \; \gamma_1, \gamma_2 \in 
\clm_S^l,
\end{equation*}
where $BL_{1}(S)$ is the space of Lipschitz functions on $S$ with 
$\|f\|_{BL} \doteq \max\{\|f\|_\infty,\|f\|_L\} \leq 1$, and
 \begin{equation*}
\|f\|_L \doteq \sup_{\theta_1, \theta_2 \in S, \theta_1\neq \theta_2} \left|\frac{f(\theta_1)-f(\theta_2)}{d(\theta_1, \theta_2)}\right|,
\end{equation*}
where $d(\theta_1, \theta_2)$ is the length of the arc $[\theta_1, \theta_2]$ of the circle $S$ viewed as the interval [0,1] with its endpoints identified. Let $\Om^l \doteq C([0,T]: \clm_S^l)$ be the space of $\clm_S^l$-valued continuous paths. This is a Polish space with distance $d_*$ given as 
\begin{equation}\label{suplip}
d_*(\mu_1,\mu_2) \doteq \sup_{0\leq t \leq T} d_{BL}(\mu_1(t,\cdot),\mu_2(t,\cdot)),\;  \mu_1, \mu_2 \in \Om^l .
\end{equation}
Let $\Om = \cup_{l \in \NN} \Om^l$. 
Let $C([0,T]:\clm_S)$ denote the space of all paths in $\clm_S$ that are continuous in the topology of weak convergence.
It is easy to check that for any $\mu\in C([0,T]:\clm_S)$ and any continuous function $f$  on $S$
$
\sup_{0\leq t \leq T} \int_S f(\theta) \mu(t,d\theta) < \infty.
$
Therefore, by the uniform boundedness principle (see for example \cite{rud})
\begin{equation*}
\sup_{0\leq t \leq T} \|\mu(t,\cdot)\|_{TV} = \sup_{0\leq t \leq T} \sup_{f \in B_1(S)}  \int_S f(\theta) \mu(t,d\theta) < \infty,
\end{equation*}
and thus  $\Omega = C([0,T]:\clm_S)$. 
In particular,  for any $\mu\in \Omega$ such that for every $t\in [0,T]$, $\mu(t, \cdot)$ has a density $m(t,\theta)$ (namely,  $\mu(t,d\theta) = m(t,\theta)d\theta$), 
\begin{align}
\label{TVbound}\sup_{0\leq t \leq T} \int_S |m(t,\theta)|d\theta<\infty.
\end{align}
The space $\Om$ will be equipped with the {\em direct limit topology}, namely a set $G\subset \Om$ is open if and only if for every $l \in \NN$, $G^{l}\doteq G \cap \Om^l$ is open in $\Om^l$. Similarly,
$\clm_S = \cup_{l\in \NN} \clm_S^l$ is equipped with the corresponding direct limit topology.

The stochastic process $\{\mu^N(t)\}$ introduced in \eqref{empmeas} has sample paths in $\Om$, i.e. $\{\mu^N\}$ is a sequence
of $\Om$-valued random variables. The goal of this work is to establish a large deviation principle for
$\{\mu_N\}$ on $\Om$ (equipped with the direct limit topology). We record below a few useful facts about the topology used here. Proofs are given in Section \ref{sec:pflemtop}.
For $x \in \Om$ and a set $A \in \Om$, let $d_*(x,A)\doteq \inf\{d_{*}(x,y): y \in A\}$.
\begin{lemma}
	\label{lem:topprops}
	The following hold.
	\begin{enumerate}[(a)]
		\item For each $l\in \NN$, the weak convergence topology on $\mathcal{M}^l_S$  is equivalent to the topology 
		induced by the bounded Lipschitz metric.
		\item Let $\mu_n, \mu \in \Om$ and suppose that $\mu_n\to \mu$. Then the following hold.
		\begin{enumerate}[(i)]
			\item For every $f \in C(S)$, $\sup_{0\le t \le T}|\lan \mu_n(t), f\ran  - \lan \mu(t), f\ran| \to 0$ as $n \to \infty$.
			\item For some $l<\infty$, $\mu_n, \mu \in \Om^l$ for all $n\in \NN$.
			\item $d_*(\mu_n, \mu)\to 0$ as $n\to \infty$. 
		\end{enumerate}
		\item Let $F$ be a closed set in $\Om$. Let for $l \in (0,\infty)$ and $x \in \Om$, $h(x) \doteq d_*(x,F^l)$.
		Then $h$ is a continuous function on $\Om$.
	\end{enumerate}
\end{lemma}

\subsection{Rate Function}
We now introduce the rate function associated with the collection $\{\mu_N\}$. The form of this rate function is different from that given in \cite{donvar5} (see Remark \ref{DVrate}).
Let  for $\lambda \in \RR$, $\rho(\lambda) \doteq \log M(\lambda),$ and let $h(x) \doteq \sup_{\lambda\in\mathbb{R}}\{\lambda x - \rho(\lambda)\}$ be the Legendre transform of $\rho$.
Let $\tilde\Omega$ denote the collection of all $\mu$ in $\Omega$ such that for all $0\leq t\leq T,$ $\mu(t,d\theta)$ has a density $m(t,\theta)$  (namely $\mu(t,d\theta)=m(t,\theta)d\theta$) that is weakly differentiable in $\theta$ and 
satisfies
\begin{align}
\label{hl1}\int_{[0,T] \times S} h(m(t,\theta))dtd\theta <\infty,\\
\label{hl2}\int_{[0,T] \times S} [h'(m(t,\theta))]_\theta^2dtd\theta  <\infty.
\end{align}

Let  $\pi\in\mathcal{P}(\mathbb{R}\times S)$ be such that $\pi$ can be disintegrated as $\pi(dx\; d\theta) = \pi_1(dx\mid\theta) d\theta$, 
$\int_{\RR}|x| \pi_1(dx|\theta) <\infty$ for a.e. $\theta$, and with $m_0(\theta) = \int_{\RR}x \pi_1(dx|\theta)$, $\int_S h(m_0(\theta)) d\theta <\infty$. We denote the collection  of all such $\pi$ as $\mathcal{P}_*(\mathbb{R}\times S)$.
For $(u,\pi) \in L^2([0,T]\times S: \mathbb{R})\times \mathcal{P}_*(\mathbb{R}\times S)$ we define
 $\mathcal{M}_\infty(u,\pi)$ to be the collection of all  $\mu\in\tilde\Omega$, $\mu(t, d\theta) = m(t,\theta) d\theta$,
such that 
$m$ solves 
\begin{align}
\label{ratepde}\partial_t m(t,\theta) &= \frac{1}{2}\left[h'(m(t,\theta))\right]_{\theta\theta} - \partial_\theta u(t,\theta), \; m(0,\theta) = m_0(\theta)
\end{align}
where   the equation is interpreted in the weak sense, namely  for any smooth function $J$ on $S$ and any $t \in [0,T]$,
\begin{align}
&\int_{S} J(\theta) {m}(t,\theta)d\theta - \int_{S} J(\theta) {m}(0,\theta)d\theta\nonumber\\
&\quad \left.= \frac{1}{2}\int_0^t\int_{S}J''(\theta)h'({m}(s,\theta)) d\theta ds + \int_0^t\int_{S}J'(\theta)u(s,\theta) d\theta ds.\right.
 \label{eq:weakpde}
\end{align}
Letting
\begin{equation}\label{pizerodef}
\pi_0(dx\; d\theta) = \Phi(dx)d\theta,
\end{equation}
define $I: \Om \to [0,\infty]$ by
\begin{equation}\label{eq:defnratefn}
I(\mu) = \inf_{\{(u,\pi): \mu\in \mathcal{M}_\infty(u,\pi)\}} \left[\frac{1}{2} \int_0^T\int_S |u(s,\theta)|^2 d\theta ds + R(\pi\|\pi_0)\right]
\end{equation}
for $\mu\in\tilde\Omega,$ where the infimum is over $(u,\pi) \in L^2([0,T]\times S: \mathbb{R})\times \mathcal{P}_*(\mathbb{R}\times S)$, and set $I(\mu)=\infty$ for $\mu\in\Omega\setminus\tilde\Omega.$
By convention, infimum over an empty set is taken to be $\infty$.
\subsection{Statement of the Main Result}
The following is the main result of this work. Proof is given in Section \ref{sec:pfmainth}.

\begin{thm}\label{MainResult}
$I$ is a rate function on $\Om$, namely for every $M<\infty$  $\{\mu \in \Om: I(\mu) \le M\}$ is compact, and   $\{\mu^N\}$ satisfies a large deviations principle on $\Om$ with rate function $I$.
\end{thm}
\begin{rem}\label{DVrate}
In \cite{donvar5}, Donsker and Varadhan proved a large deviations principle for $\mu^N$ with rate function 
\begin{equation*}
\tilde{I}(\mu) = \int_S h(m(0,\theta)) d\theta +\int_0^T |\partial_t m(t,\theta) - [h'(m(t,\theta))]_{\theta\theta}|_{-1} dt,
\end{equation*}
 if $\mu$ has a density such that $\mu(t,d\theta)=m(t,\theta)d\theta,$ and $I(\mu)=\infty$ otherwise. Here, $|\cdot|_{-1}$ denotes the Sobolev $H_{-1}$ semi-norm, i.e. the dual of the Sobolev $H_{1}$ semi-norm (see \cite{adafou}),
By the uniqueness of rate functions (see for example \cite[Theorem 1.3.1]{dupell4}) and Theorem \ref{MainResult}, it follows that $I=\tilde{I}.$ This fact  can also be verified directly by using the identities 
\begin{align*}
 \int_S h(m(0,\theta)) d\theta 
=\inf\left\{R(\pi\|\pi_0): \pi \in \mathcal{P}_*(\mathbb{R}\times S) \text{ and } \int_{\mathbb{R}}x\pi_1(dx|\theta) = m(0,\theta) \text{ for every }\theta\in S\right\} 
\end{align*}
and
\begin{align*}
&\int_0^T |\partial_t m(t,\theta) - [h'(m(t,\theta))]_{\theta\theta}|_{-1} dt\\
&\quad=\inf\left\{\int_0^T\int_S |u(t,\theta)|^2 d\theta dt: u \in L^2([0,T]\times S) \text{ and } \partial_\theta u = \partial_t m - \frac{1}{2} [h'(m(t,\theta)]_{\theta\theta}\right\}.
\end{align*}
\end{rem}

\noindent {\em Organization.} Rest of the paper is organized as follows. In Section \ref{sec:pfmainth} we provide the proof of our main result, namely Theorem \ref{MainResult}. The proof relies on Proposition \ref{union} which says that it suffices to prove certain Laplace asymptotics and compactness properties of level sets of $I$. These properties are established on Theorem \ref{thm:lapasymp}
and Lemma \ref{ALM}. Proofs of Proposition \ref{union}, Theorem \ref{thm:lapasymp}, and Lemma \ref{ALM} are given in Sections \ref{sec:prfpropun},  \ref{sec:lapasymp}, and \ref{sec:prooflemalm} respectively.

Section \ref{sec:lapasymp} that establishes the desired Laplace asymptotics (namely Theorem \ref{thm:lapasymp}) relies on several other results. The first two key lemmas are Lemma \ref{entlemma} and \ref{Ilemma} that give suitable bounds on relative entropies and Dirichlet forms.  The proofs of these two lemmas and of a key lemma on regularity of densities of controlled processes (Lemma \ref{reglemma}) are given in Section \ref{entsection}. Proof of Theorem \ref{thm:lapasymp} also uses a potpourri of tightness and weak convergence results (Lemmas \ref{mubartight}--\ref{convtopistar}) some of which are quite standard. Proof of Lemma \ref{mubartight} is in Section \ref{proofmubartight} while Lemmas \ref{lbartight}--\ref{convtopistar} are proved in Section \ref{prooflemmas}. Another important result needed in the proof of Theorem \ref{thm:lapasymp} is the characterization of weak limits of controls and controlled processes.  This result, formulated in Theorem \ref{HL} is proved in Section \ref{ProofHL}. The final set of results needed for the proof of Theorem \ref{thm:lapasymp} are Lemmas \ref{existence} and \ref{continu} that give existence, uniqueness and continuity properties of controlled hydrodynamic PDE (equation (\ref{ratepde})). These lemmas are proved in Section \ref{weakuniquesec}. Based on the above results, proof of Theorem \ref{thm:lapasymp} is completed in Sections \ref{LUB} and \ref{LLB}. 

Thus the overall organization is as follows. Section \ref{sec:lapasymp}: Proof of Theorem \ref{thm:lapasymp};
Section \ref{sec:prooflemalm}: Proof of Lemma \ref{ALM}; Section \ref{ProofHL}: Proof of Theorem \ref{HL}; Section
\ref{entsection}: Proofs of Lemmas \ref{entlemma} and \ref{Ilemma} and the regularity lemma, Lemma \ref{reglemma};
Section \ref{proofmubartight}: Proof of Lemma \ref{mubartight}; Section \ref{prooflemmas}:
Proofs of Lemmas \ref{lbartight}, \ref{limitssame},  \ref{UnifInt}, and  \ref{convtopistar}; Section \ref{sec:prfpropun}:
Proof of Proposition \ref{union}; Section \ref{weakuniquesec}: Proofs of Lemmas \ref{existence} and \ref{continu}. Finally Section \ref{sec:pflemtop} gives the proof of  Lemma \ref{lem:topprops}.

\section{Proof of Theorem \ref{MainResult}}
\label{sec:pfmainth}
In this section we present the proof of Theorem \ref{MainResult}. The main ingredient in the approach we take
is the following result which says that in order to prove the  large deviation principle it suffices to prove certain Laplace 
asymptotics and certain compactness properties of the function $I$. Recall that a function $I: \cle \to [0,\infty]$ is called a {\em rate function} if it has compact sublevel sets, i.e. for every $M<\infty$, the set $\{x: I(x)\le M\}$ is a compact subset of $\cle$.
For a sequence of random variables $Z^N$ taking values in a Polish space $\cle$, it is well known that (cf. \cite[Theorem 1.2.3]{dupell4}) the large deviations principle is equivalent to the Laplace principle, that is, $Z^N$ satisfies the large deviations principle with rate function $I$ if and only if $I$ has compact sub-level sets and for all bounded and continuous functions $F$,
\begin{equation}
\label{LP}\lim_{N\rightarrow\infty}\frac{1}{N} \log \mathbb{E}\exp(-NF(Z^N)) = -\inf_{x\in \cle} \{I(x) + F(x)\}.
\end{equation}
 Since $\Omega$ is not a Polish space, but instead an infinite union of Polish spaces,  we will need the following generalization of the above result. Proof is given in Section \ref{sec:prfpropun}.

For a set $A \subset \Om$, $A^l$ will denote
$A \cap \Om^l$. For $A \subset \Om$ and $I:\Om \to [0,\infty]$, let
$I(A) \doteq \inf_{x\in A}I(x)$.
\begin{proposition}\label{union}
 Suppose that $\{Z^N\}_{N\in \NN}$ is a sequence of $\Om$-valued random variables such that 
\begin{equation}
\label{expdecay}\lim_{l\rightarrow\infty}\limsup_{N\rightarrow\infty}\frac{1}{N}\log\left(\mathbb{P}(Z^N\in \Om\setminus\Om^l)\right) = -\infty. 
\end{equation}
Let $I:\Om \to [0,\infty]$.
Then the following hold.
\begin{enumerate}[(a)]
	\item If for all continuous and bounded $g: \Om \to \RR$
	\begin{equation}\label{eq:lapopen}
		\liminf_{N\to \infty} \frac{1}{N} \log \EE[\exp(-Ng(Z^N))] \ge -\inf_{\mu\in \Om}\{g(\mu) + I(\mu)\},
	\end{equation}
	then for every open set $G \subset \Om$, 
	\begin{equation}
	\label{lowergen}\liminf_{N\rightarrow\infty} \frac{1}{N}\log \mathbb{P}(Z^N\in G) \ge -I(G).
	\end{equation}
	\item Suppose that for every  $M<\infty$, the set
	$\Gamma_{M}  \doteq \{\mu\in\Omega:I(\mu)\leq M\}$ is a compact subset of $\Om$. 
If for all continuous and bounded $g: \Om \to \RR$
\begin{equation}\label{eq:lapclos}
	\limsup_{N\to \infty} \frac{1}{N} \log \EE[\exp(-Ng(Z^N))] \le -\inf_{\mu\in \Om}\{g(\mu) + I(\mu)\},
\end{equation}
then for every closed set $F\subset \Om$
\begin{equation}
\label{upperl}\limsup_{N\rightarrow\infty} \frac{1}{N}\log \mathbb{P}(Z^N\in F^l) \le -I(F^l), \; \mbox{ for all } l \in \NN,
\end{equation}
and
\begin{equation}
\label{uppergen}\limsup_{N\rightarrow\infty} \frac{1}{N}\log \mathbb{P}(Z^N\in F) \le -I(F).
\end{equation}
%
%

\end{enumerate}

\end{proposition}
The following result shows that for the collection $\{\mu^N\}$ introduced in \eqref{empmeas}, the Laplace asymptotics of the 
form needed in Proposition \ref{union} are satisfied. Proof is given in Section \ref{sec:lapasymp}.

\begin{thm}\label{thm:lapasymp}
For all bounded and continuous $g: \Om \to \RR$,
\begin{equation*}
	\lim_{N\rightarrow\infty}-\frac{1}{N} \log \mathbb{E} \exp(-Ng(\mu^N)) = \inf_{\mu \in \Omega} \left[I(\mu)+g(\mu)\right]
\end{equation*}
where $I$ is defined by \eqref{eq:defnratefn}.
\end{thm}
The following result gives a compactness property of sub-level sets of $I$. The proof is given in Section \ref{sec:prooflemalm}.
\begin{lemma}\label{ALM}
	Let $I$ be as in \eqref{eq:defnratefn}.
For all $l\in \NN$ and $M<\infty$ the set $\Gamma_{l,M}  \doteq \{\mu\in\Omega^l:I(\mu)\leq M\}$ is a compact subset
of $\Om^l$.
\end{lemma}

\subsection{Completing the Proof of Theorem \ref{MainResult}}
\label{Proofofmain}

In order to prove the first statement  of Theorem \ref{MainResult}, namely $I$ is a rate function on $\Om$, it suffices in view of Lemma \ref{ALM} to show that for every $M<\infty$, there exists
a $l \in \NN$ such that $\Gamma_M \doteq \{\mu \in \Om: I(\mu)\le M\} \subset \Gamma_{l,M}$. We argue via contradiction. Suppose  that there exists $M<\infty$ such that for every $l\in \NN$ there exists $\mu^l\notin \Omega^l$ such that $I(\mu^l)\leq M.$ From the lower semi-continuity of total variation it follows that
 $\Omega^l$ is is closed in $\Omega^{l'}$ for all $l' \ge l$. 
Thus, $(\Omega^l)^c$ is open in $\Om$.  \cite[Lemma 6.1]{guopapvar} shows that there exist  $C_1,C_2, l_0 \in (0,\infty)$ such that 
\begin{equation}
\label{GPVest}\mathbb{P}( \mu^N\notin \Omega^l) \leq C_1 e^{-C_2Nl} \mbox{ for all } l\ge l_0 \mbox{ and } N \in \NN.
\end{equation}
In particular \eqref{expdecay} holds with $Z^N$ replaced with $\mu^N$. It now follows from Proposition \ref{union}(a) and Theorem \ref{thm:lapasymp} that
for each $l\in \NN$
\begin{equation}
\liminf_{N\rightarrow\infty}\frac{1}{N}\log\mathbb{P}(\mu^N\in (\Omega^l)^c)\geq -I((\Omega^l)^c) \geq -I(\mu^l) \ge -M. \label{eq:ratefncont}
\end{equation} 
However this contradicts \eqref{GPVest} and therefore the proof that $I$ is a rate function on $\Om$ is complete.  The second part of the theorem
is now immediate from \eqref{GPVest} (which, as noted previously, implies \eqref{expdecay}  with $Z^N$ replaced with $\mu^N$), Proposition \ref{union} and Theorem \ref{thm:lapasymp}.
\hfill \qed

\section{Proof of Theorem \ref{thm:lapasymp}}
\label{sec:lapasymp}
\subsection{Variational Representation}
\label{subsec:varrep}
 The following representation formula for exponential functionals of $F(\mu^N)$ follows from \cite[Proposition 4.1]{budfanwu}. The latter result builds upon ideas in the proof of a similar representation for functionals of a finite dimensional Brownian motion in \cite{boudup}. Let $(\bar \clv, \bar \clf, \bar \PP)$ be a probability space on which we are given  an
$N$-dimensional Brownian motion, which we denote once more as $(B_1, \ldots B_N) = \bfB^N$, and a $\RR^N$-valued random variable $\bar X^N(0)$ independent of $\bfB^N$ and with probability law $\Pi^N$.
Let $\{\bar \clf_t\}$ be any filtration satisfying the usual conditions such that $\bfB^N$ is a $\{\bar \clf_t\}$-Brownian motion and $\bar X^N(0)$ is $\bar \clf_0$ measurable. Let
$\sys_{\Pi^N} \doteq (\bar \clv, \bar \clf, \{\bar \clf_t\}, \bar \PP, \bar X^N(0), \bfB^N)$ and consider the following collection of processes
$$
\cla^N(\sys_{\Pi^N}) \doteq \{\psi: \psi = (\psi_i)_{i=1}^N \mbox{ and each }  \psi_i \mbox{ is a real-valued } \bar \clf_t \mbox{ progressively measurable process}\}.$$
Let $\cla^N_b(\sys_{\Pi^N})$ denote the collection $\psi^N \in \cla^N(\sys_{\Pi^N})$ such that for some $M\in (0,\infty)$,  $\int_0^T |\psi^N(s)|^2 ds \le M$ a.s.
For a $\psi^N \in \cla^N_b(\sys_{\Pi^N})$, let 
$$\bar B^N_i(t) \doteq B_i(t) + \int_0^t \psi^N_i(s) ds, \; t \in [0,T], \; i = 1, \ldots N.$$
Let $\bar X^N(t) \doteq (X^N_i(t))_{i=1}^N$ be the solution to the system of equations defined the same way as (\ref{align:uncontroleq}) but with $\bar{B}_i^{N}(t)$ in place of $B_i(t)$, i.e. 
\begin{align}\label{align:controleq}
d\bar X^{N}_i(t)&\doteq d\bar Z^{N}_i(t) - d\bar Z^{N}_{i+1}(t),\nonumber\\
d\bar Z^{N}_i(t)&\doteq\frac{N^2}{2}\left[\phi'\left(\bar X^{N}_{i-1}(t)\right) - \phi'\left(\bar X^{N}_{i}(t)\right) \right]dt + Nd\bar B_i(t) .
\end{align}
We shall refer to $\bar{X}^{N}$ as the controlled process and to $X^N$ as the uncontrolled process. Let $\bar\mu^N(t)$ denote the signed measure on the unit circle associated with the controlled process $\bar{X}^N_t,$   defined in a manner analogous to (\ref{empmeas}).
Given a probability measure $\Pi^N\in\mathcal{P}(\mathbb{R}^N),$ we consider the  disintegration
\begin{equation*}
\Pi^N(dx) \doteq \Pi_1(dx_1)\Pi_2(dx_2|x_1)\ldots\Pi_N(dx_N|dx_1,\ldots,dx_{N-1}) \doteq \prod_{i=1}^N \bar\Phi_i^N(x,dx_i),
\end{equation*}
and with $\bar X^N(0)$ distributed as $\Pi^N$, we define a family of $\mathcal{P}(\mathbb{R})$-valued random variables by
 \begin{equation}\label{eq:barphini}
\mathbf{\bar\Phi}_i^N(dx) \doteq \bar\Phi_i^N(\bar{X}^N(0),dx). 
\end{equation}
In order to emphasize the initial distribution $\Pi^N$, we will sometimes write the probability measure $\bar \PP$ as $\bar \PP_{\Pi^N}$ and denote the corresponding expectation by
$\bar \EE_{\Pi^N}$.
The following representation is a consequence of \cite[Proposition 4.1]{budfanwu} (see also Lemma 5.1 therein).
\begin{lemma}
	\label{lem:repn}
Let $F: \Om \to \RR$ be a continuous and bounded function.	Then for all $N \in \NN$
\begin{multline}\label{expectvarA}
-\frac{1}{N} \log \mathbb{E} \exp(-NF(\mu^N))\\
= \inf_{\Pi^N, \sys_{\Pi^N}}\inf_{\psi^N \in \cla^N_b(\sys_{\Pi^N})} \bar{\mathbb{E}}_{\Pi^N} \left[\frac{1}{N}\sum_{i=1}^N\left(R(\mathbf{\bar\Phi}_i^N\|\Phi)+\frac{1}{2}\int_0^T|\psi_i^N(s)|^2 ds \right) + F(\bar{\mu}^N)\right],
\end{multline}
where the outer infimum is over all $\Pi^N \in \clp(\RR^N)$ and all systems $\sys_{\Pi^N}$.
\end{lemma}
 An examination of the proof of \cite[Proposition 4.1]{budfanwu} and \cite{boudup} (see \cite[Lemma 3.5]{buddup3}) shows that the class of controls on the right side above can be restricted as follows.
For $N \in \NN$, let $\cla^N_s(\sys_{\Pi^N})$ denote the class of  simple adapted processes $\psi^N \in  \cla^N_b(\sys_{\Pi^N})$, namely 
 for each $i = 1, \ldots , N$, $\psi^N_i$ is of the form 
\begin{equation}
\psi^N_i(t)  \doteq \sum_j U_{ij}\mathbb{I}_{(t_{j},t_{j+1}]}(t)\label{eq:defnsimpcont}
\end{equation}
where $0=t_0\leq t_1\leq\ldots\leq t_K=T$ is a partition of $[0,T]$ and $U_{ij}$ is a family of real random variables such that $U_{ij}$ is measurable with respect to 
$\sigma(\{\bfB^N(t):0\leq t\leq t_j, \bar X^N(0)\})$ and for some $C\in (0,\infty)$
\begin{equation}
\label{controlas} \max_{i,j} |U_{ij}| \leq C
\end{equation} 
almost surely. Note that the partition and the constant $C$ are allowed to depend on $N$ and the control $\psi^N$. The following result says that $\cla^N_b(\sys_{\Pi^N})$ in Lemma \ref{lem:repn}
can be replaced by the smaller class $\cla^N_s(\sys_{\Pi^N})$.
\begin{lemma}
	\label{lem:repnmod}
Let $F: \Om \to \RR$ be a continuous and bounded function.	Then for all $N \in \NN$
\begin{multline}\label{expectvar}
-\frac{1}{N} \log \mathbb{E} \exp(-NF(\mu^N))\\
= \inf_{\Pi^N, \sys_{\Pi^N}}\inf_{\psi^N \in \cla^N_s(\sys_{\Pi^N})} \bar{\mathbb{E}}_{\Pi^N} \left[\frac{1}{N}\sum_{i=1}^N\left(R(\mathbf{\bar\Phi}_i^N\|\Phi)+\frac{1}{2}\int_0^T|\psi_i^N(s)|^2 ds \right) + F(\bar{\mu}^N)\right],
\end{multline}
where the outer infimum is over all $\Pi^N \in \clp(\RR^N)$ and all systems $\sys_{\Pi^N}$.
\end{lemma}

\subsection{The Laplace Upper Bound}
In this section we will prove the inequality 
\begin{equation}\label{eq:lapuppmain}
	\limsup_{N\rightarrow\infty}-\frac{1}{N} \log \mathbb{E} \exp(-NF(\mu^N)) \ge \inf_{\mu \in \Omega} \left[I(\mu)+ F(\mu)\right]
\end{equation}
for all bounded and continuous $F: \Om \to \RR$,
where $I$ is defined by \eqref{eq:defnratefn}. This inequality, together with the complementary inequality given in Section \ref{sec:laplb} will complete the proof of
Theorem \ref{thm:lapasymp}. We begin with some key bounds on certain relative entropies and Dirichlet forms.

\subsubsection{Bounds on relative entropy and an associated Dirichlet form}\label{s2}
In this section we present two technical lemmas. 
Lemma \ref{entlemma} tells us that if the relative entropy $H_N(0)$ of the initial measure $\Pi^N$ with respect to $\Phi^N$ grows linearly with $N$, then so does the relative entropy $H_N(t)$ between the law of the controlled process $\bar X^N(t)$ and that of the uncontrolled (stationary) process $X^N(t)$ at time $t$, for suitable collection of controls. This lemma will be key in proving tightness of the  signed measure valued processes $\bar\mu^N$ as well as in characterizing the subsequential hydrodynamic limits of these controlled processes.
\begin{lemma}\label{entlemma}
	Let $\Pi^N \in \clp(\RR^N)$ for each $N\in \NN$. Consider a sequence of controls $\{\psi^N\}_{N\in \NN}$ such that $\psi^N \in  \cla^N_s(\sys_{\Pi^N})$, for some system $\sys_{\Pi^N}$, for each $N$. Suppose that
	for some $C_0 \in (0,\infty)$
	\begin{equation}
		\label{Hzeroeq}
		\sup_{N\in \NN} \frac{1}{N} \sum_{i=1}^N \int_0^T |\psi^N_i(s)|^2 ds \le C_0, \; \; 	\sup_{N\in \NN} \frac{1}{N}H_N(0) \doteq \frac{1}{N} R(\Pi^N \| \Phi^N) \le C_0.
	\end{equation}
	Denote the controlled process associated with the controls $\psi^N$ and initial distribution $\Pi^N$ as $\bar X^N$ and
let for $t\in [0,T]$, $\bar{\mathbb{Q}}_{\Pi^N}(t)$ denote the law of the controlled random variable $\bar X^N(t)$. Then, there exists  $C_T\in(0,\infty)$ such that for every $t \in [0,T]$,
\begin{equation}
\label{initentest}H_N(t) \doteq R(\bar{\mathbb{Q}}_{\Pi^N}(t) \| \Phi^N) \leq C_TN \text{ for all } N \in \NN.
\end{equation}
\end{lemma}
For any function $f$ on $\mathbb{R}^N$ that is continuously differentiable along the vector fields $V_1,\dots,V_N$, we define the Dirichlet form
$$
D_N(f) \doteq \sum_{i=1}^N\int_{\mathbb{R}^N} (V_if(x))^2\Phi^N(dx)
$$
If, in addition,  $f$ is positive, we define $I_N(f)$, given in terms of the Dirichlet form of the square root of $f$,  as follows:
\begin{equation*}
I_N(f) = 4D_N(\sqrt{f})=\sum_{i=1}^N\int_{\mathbb{R}^N} \frac{(V_if(x))^2}{f(x)}\Phi^N(dx).
\end{equation*}
Lemma \ref{Ilemma} gives an upper bound on $I_N(f)$.
\begin{lemma}\label{Ilemma}
Let $\Pi^N$, $\psi^N$, $\bar X^N$ be as in Lemma \ref{entlemma}. Then for each $t \in [0,T]$ and $N\in \NN$, $\bar X^N(t)$ 
has a density $\{\bar p_N(t,x) : x \in \mathbb{R}^N\}$ with respect to $\Phi^N$   which is continuously differentiable along the vector fields $V_1,\dots,V_N$ and satisfies the following 
bound
for some  $C\in (0,\infty)$:
\begin{equation}
\label{INest} I_N\left(\frac{1}{T}\int_0^T \bar{p}_N(s,\cdot)ds\right) \leq \frac{C}{N} \text{ for all } N \ge 1.
\end{equation}
\end{lemma}
Lemmas \ref{entlemma} and \ref{Ilemma} provide the key technical estimates in proving that subsequential hydrodynamic limits of controlled processes are weak solutions of \eqref{ratepde} via the `block estimate method' of \cite{guopapvar}. More precisely,
these two lemmas will allow us to apply \cite[Theorem 4.1]{guopapvar}, which will be the main ingredient in the proof of part (v) of Theorem \ref{HL} stated in Section \ref{subHL} below (see proof of \eqref{block}). Lemmas \ref{entlemma} and \ref{Ilemma} will be proved in Section \ref{entsection}.
\subsubsection{Tightness results}\label{tightness}
In this section, we collect several lemmas that provide certain tightness properties and characterizations of limit points.
Lemma \ref{mubartight} establishes the tightness of the controlled processes $\{\bar\mu^N \}$ in $\Omega$ for a suitable class of controls.
\begin{lemma}\label{mubartight}
Suppose that $\Pi^N$, $\psi^N$, $\bar X^N$ are as in Lemma \ref{entlemma}. Then the associated sequence of controlled signed measure valued processes $\{\bar \mu^N\}$ is a tight sequence of $\Om$-valued random variables.
\end{lemma}
Lemma \ref{mubartight} will be proved in Section \ref{proofmubartight}.\\\\

For $N\in \NN$, fix $\Pi^N \in \clp(\RR^N)$ and let $\bar X^N(0)$ be a $\RR^N$-valued random variable with distribution $\Pi^N$. Let $\mathbf{\bar\Phi}_i^N$ be $\clp(\RR)$-valued random variables as defined
in \eqref{eq:barphini}. Define
a collection of  $\mathcal{P}(\mathbb{R}\times S)$-valued random variables  by
\begin{equation}
\nu^N_i(dx d\theta) \doteq \mathbf{\bar\Phi}_i^N(dx) \delta_{i/N}(d\theta), \; i = 1, \ldots N, \; N \in \NN \label{nudef}
\end{equation}
and let
 $\nu^N (dx  d\theta) = \frac{1}{N}\sum_{i=1}^N\nu_i^N(dx  d\theta).$
Also consider a related random probability measure on $\mathbb{R}\times S$ given by
\begin{equation}\label{eq:lndxdt}
\bar{L}^N(dxd\theta) \doteq\frac{1}{N} \sum_{i=1}^N\delta_{\bar{X}_i^{N}(0)}(dx)\delta_{i/N}(d\theta).
\end{equation}
The following lemmas establish tightness of $\{\bar{L}^N,\nu^N\}$ and also characterize the subsequential limits.
\begin{lemma}\label{lbartight} 
Let $\mathbf{\bar\Phi}_i^N$, $\bar L^N$ and $\nu^N$	be as above.
	Suppose  
that for some $C\in (0,\infty)$
\begin{equation}
\label{Rbound}\mathbb{E} \left(\frac{1}{N}\sum_{i=1}^N R(\mathbf{\bar\Phi}_i^N\|\Phi)\right) \leq C.
\end{equation}
 Then $\{\bar{L}^N,\nu^N\}$ is a tight collection of $(\clp(\RR\times S))^2$-valued random variables.\end{lemma}
\begin{lemma}\label{limitssame}
	Let $\mathbf{\bar\Phi}_i^N$, $\bar L^N$ and $\nu^N$	be as in Lemma \ref{lbartight}.
	Suppose $(\bar{L}^N,\nu^N)$ converge in distribution to $(\bar{L},\nu)$ along some subsequence. Then $\bar{L}=\nu$ with probability 1. Furthermore, the second marginal of $\bar{L}$ is equal to $\lambda$, the Lebesgue measure on $S$.
\end{lemma}

\begin{lemma}
\label{UnifInt}
Let $\mathbf{\bar\Phi}_i^N$, $\bar L^N$ and $\nu^N$	be as in Lemma \ref{lbartight}.
Suppose that  $\bar{L}^N$ converges in distribution to  $\bar L$ along a subsequence. Then,
\begin{equation}
\label{rightlimit}\int_S\int_{\mathbb{R}} |x|\bar{L}(dxd\theta) < \infty, \mbox{ a.s. }
\end{equation}
 Furthermore,  $\bar\mu^N(0,d\theta) =  \int_{\mathbb{R}}x\bar{L}^N(dx d\theta)$ converges in distribution in $\clm_S$ to some limit $\bar\mu(0,d\theta)$ along the same subsequence, and
 \begin{equation}
\label{integrablelimit}\bar\mu(0,d\theta) =  \int_{\mathbb{R}}x\bar{L}(dx d\theta),\; a.s.
\end{equation}
\end{lemma}
\begin{lemma}\label{convtopistar} 
Let $\pi^* \in \clp(\mathbb{R} \times S)$ be such that its second marginal is the Lebesgue measure on $S$. Define
\begin{equation}
\label{muchoice}\bar\Phi_i^N(dx) \doteq N\int_{(i-1)/N}^{i/N}\pi^*_1(dx|\theta)d\theta, \ \ 1 \le i \le N,
\end{equation}
where $\pi^*(dx,d\theta) = \pi_1^*(dx | \theta) d\theta$. Suppose that $R(\pi^*\|\pi_0)<\infty$, where $\pi_0$ was defined in \eqref{pizerodef}.
Let $ \bar X^N(0)\doteq (\bar X_1(0), \dots, \bar X_N(0))$ be a $\RR^N$-valued random variable with distribution 
$$\Pi^N(dx) \doteq \bar\Phi_1^N(dx_1)\ldots \bar\Phi_N^N(dx_N).$$ Then $\bar L^N$ defined by \eqref{eq:lndxdt} converges in probability to $\pi^*$. 
\end{lemma}    
Proofs of Lemmas \ref{lbartight}, \ref{limitssame}, \ref{UnifInt} and \ref{convtopistar} are quite standard, however for completeness, details are given  in Section \ref{prooflemmas}.

\subsubsection{Characterizing subsequential  limits of controlled processes}\label{subHL}
The following theorem characterizes subsequential hydrodynamic limits of the controlled processes $\{\bar \mu^N\}$ and, in particular, establishes that any subsequential hydrodynamic limit has a density which is a solution to \eqref{ratepde}.  Let for $N \in \NN$, $\psi^N = (\psi^N_1, \ldots \psi^N_N) \in L^2([0,T]:\RR^N)$.  Associated with such a $\psi^N$, define $u_N = u_N(\psi^N) \in L^2([0,T]\times S)$ by 
\begin{equation}\label{eq:psitoun}
u_N(t,\theta) \doteq \sum_{i=1}^N \psi_i^N(t)\mathbb{I}_{((i-1)/N,i/N]}(\theta), \; (t,\theta) \in [0,T]\times S.
\end{equation}
Note that
\begin{equation*}
\int_{[0,T]\times S} |u_N(t,\theta)|^2 dt d\theta =\frac{1}{N}\sum_{i=1}^N \int_0^T |\psi_i^N(t)|^2 dt.
\end{equation*}
In particular if $\{\psi^N\}$ is a sequence as in Lemma \ref{entlemma} satisfying the first bound in \eqref{Hzeroeq}, then the associated sequence $\{u_N\}$, $u_N=u_N(\psi^N)$
takes value in the set
$$
\mathcal{S}_{C_0} \doteq \left\lbrace u \in L^2([0,T] \times S) :  \int_{[0,T]\times S} |u(t,\theta)|^2 d\theta dt \le C_0 \right\rbrace.
$$
Equipped with the topology of weak convergence on the Hilbert space $L^2([0,T] \times S)$, $\mathcal{S}_{C_0}$ is a compact metric space and thus  $\{u_N\}_{N\in \NN}$ regarded as a sequence of
$\mathcal{S}_{C_0}$-valued random variables is automatically tight.
\begin{thm}\label{HL} Suppose that 
 $\Pi^N$, $\psi^N$, $\bar X^N$ are as in Lemma \ref{mubartight} and suppose that along some subsequence $\{\bar{\mu}^N, u_N\}$ converges in distribution to $(\bar\mu, u)$ as $\Omega \times \mathcal{S}_{C_0}$-valued random variables. Then the following hold almost surely.
\begin{itemize}
\item[(i)]  There is a measurable  function $\bar m$ on $[0,T]\times S$ such that for almost every $t\in[0, T]$, $\bar m(t, \cdot) \in L^1(S)$ is the density of 
$\bar{\mu}(t, d\theta)$, namely $\bar{\mu}(t, d\theta) = \bar{m}(t,\theta)d\theta$.
\item[(ii)] $\bar{m}(0, \theta)$ is the density of $\bar{\mu}(0, d\theta)$, namely $\bar{\mu}(0,d\theta) = \bar m(0,\theta)d\theta$.
\item[(iii)]  $ \int_{[0,T]\times S} h(\bar{m}(t, \theta)) dt  d \theta < \infty$ and $ \int_{ S} h(\bar{m}(0, \theta))   d \theta < \infty$.
\item[(iv)]  For a.e. $t\in [0,T]$, the map $\theta \mapsto h'(\bar{m}(t, \theta))$ is weakly differentiable and $$\int_{[0,T]\times S}\left[\partial_{\theta}\left(h'(\bar{m}(t, \theta))\right)\right]^2 dt d \theta  < \infty.$$
\item[(v)] 
$\bar{m}$ is a weak solution to (\ref{ratepde}), i.e. for any smooth $J$ on $S$ and $t\in [0,T]$, \eqref{eq:weakpde} is satisfied with $m$ replaced by $\bar m$, $m_0 = \bar m(0, \cdot)$, and $u$ as above.


\end{itemize}

\end{thm}
Theorem \ref{HL} will be proved in Section \ref{ProofHL}.

\subsubsection{Completing the proof of  Laplace upper bound}\label{LUB}
We now complete the proof of the inequality in \eqref{eq:lapuppmain}.
Fix $F$ bounded and continuous on $\Om$ and let $\epsilon \in (0,1).$ Using Lemma \ref{lem:repnmod} we can choose for each $N \in \NN$, $\Pi^N\in\mathcal{P}(\mathbb{R}^N)$, a system $\sys_{\Pi^N}$ and
$\psi^N \in \cla^N_s(\sys_{\Pi^N})$ such that 
\begin{equation}\label{simpleas}
-\frac{1}{N} \log \mathbb{E} \exp(-NF(\mu^N))
\ge \bar{\mathbb{E}}_{\Pi^N} \left[\frac{1}{N}\sum_{i=1}^N\left(R(\mathbf{\bar\Phi}_i^N\|\Phi)+\frac{1}{2}\int_0^T|\psi_i^N(s)|^2 ds \right) + F(\bar{\mu}^N)\right] - \epsilon.
\end{equation}
Since $F$ is bounded, there is a $C \in (0,\infty)$ such that  
$$\sup_{N\in \NN} \bar{\mathbb{E}}_{\Pi^N} \left(\frac{1}{N}\sum_{i=1}^N R(\mathbf{\bar\Phi}_i^N\|\Phi)\right) \le C, \; 
\sup_{N \in \NN} \bar{\mathbb{E}}_{\Pi^N}\left(\frac{1}{2N}\sum_{i=1}^N \int_0^T|\psi_i^N(s)|^2 ds\right) \le C.$$
By a standard localization argument and by choosing a larger $C$ if needed, we can assume without loss of generality that for every $N$
$$\frac{1}{2N}\sum_{i=1}^N\int_0^T|\psi_i^N(s)|^2 ds \le C \;\; a.s.$$
Define $u_N = u_N(\psi^N)$ as in \eqref{eq:psitoun}. 
By the lemmas in Section \ref{tightness}, we may find a common subsequence along which $(u_N,\bar{L}^N, \nu^N, \bar\mu^N)$   converge in distribution, in
$\mathcal{S}_{C}\times (\clp(\RR\times S))^2 \times \Om$  to $(u,\bar{L}, \nu, \bar\mu)$.
Using Fatou's Lemma and the fact that the map $f \mapsto \int_0^T \int_{S} |f(s, \theta)|^2dsd\theta$ is lower semicontinuous on $L^2([0,T] \times S)$ with respect to the weak topology, we have
\begin{equation}\label{Fat}
\liminf_{N\rightarrow\infty} \bar{\mathbb{E}}_{\Pi^N} \int_0^1\int_S |u_N(s,\theta)|^2 dsd\theta \geq \bar{\mathbb{E}} \int_0^1\int_S |u(s,\theta)|^2 dsd\theta. 
\end{equation}
Note that $\bar{L}=\nu$ by Lemma \ref{limitssame} and that
$\bar \mu(0, d\theta) = \int_{\mathbb{R}}x\bar{L}(dx d\theta)$ by Lemma \ref{UnifInt}. 
Furthermore, from Theorem \ref{HL}, $\nu \in \clp_*(\RR \times S)$ and  $\bar\mu\in\mathcal{M}_\infty(u,\nu)$ a.s.

Define the random measure on $\RR\times S$ as
$$
m^N(dx\;d\theta) \doteq \sum_{i=1}^N \mathbf{\bar{\Phi}}_i^N(dx) \mathbb{I}_{(i/N, (i+1)/N]} (\theta)d\theta.
$$
By integrating against uniformly continuous test functions on $\mathbb{R} \times S$, it is clear that $m^N$ converges weakly to the same limit as $\nu^N$, namely $\nu$. Moreover, by the chain rule for relative entropies (see for example \cite[Theorem C.3.1]{dupell4}),
$
\frac{1}{N}\sum_{i=1}^NR(\mathbf{\bar{\Phi}}_i^N\|\Phi) = R(m^N\| \pi_0).
$
Therefore, by the lower semicontinuity of $R(\cdot \| \pi_0)$,
\begin{equation}\label{Rcon}
\liminf_{N \rightarrow \infty} \frac{1}{N}\sum_{i=1}^NR(\mathbf{\bar{\Phi}}_i^N\|\Phi) = \liminf_{N \rightarrow \infty} R(m^N\| \pi_0) \ge R(\nu \| \pi_0).
\end{equation}
Thus, by \eqref{simpleas}, \eqref{Fat}, \eqref{Rcon} and the continuity of $F$, we have
\begin{align*}
&\liminf_{N\rightarrow\infty} -\frac{1}{N}\log \mathbb{E} \exp\left(-NF\left(\mu^N\right)\right) + \epsilon\\
&\geq \liminf_{N\rightarrow\infty}\bar{\mathbb{E}}_{\Pi^N}\left(F\left(\bar\mu^N\right) + \frac{1}{N}\sum_{i=1}^N\left(R(\mathbf{\bar\Phi}_i^N\|\Phi)+\frac{1}{2}\int_0^T|\psi_i^N(s)|^2 ds\right)\right)\\
&= \liminf_{N\rightarrow\infty}\bar{\mathbb{E}}_{\Pi^N}\left(F\left(\bar\mu^N\right) + R(m^N\| \pi_0)+\frac{1}{2}\int_0^T\int_S|u_N(s,\theta)|^2dsd\theta\right)\\
&\geq \bar{\mathbb{E}} \left(F\left(\bar\mu\right) + R(\nu \|\pi_0)+\frac{1}{2}\int_0^T\int_S|u(s,\theta)|^2dsd\theta\right)\\
&\geq \inf_{\mu\in \Omega} \left[F(\mu) + I(\mu)\right],
\end{align*}
where the last inequality uses the fact that $\bar\mu\in\mathcal{M}_\infty(u,\nu)$ a.s.
This completes the proof of \eqref{eq:lapuppmain}. \hfill \qed
\subsection{Laplace Lower Bound}
\label{sec:laplb}
In this section we establish the complementary bound to \eqref{eq:lapuppmain}, namely we show that for every bounded and continuous $F: \Om \to \RR$,
\begin{equation}\label{eq:mainlaplowbd}
\limsup_{N\rightarrow\infty} -\frac{1}{N} \log \mathbb{E} \exp (-NF(\mu^N)) \leq \inf_{{\mu}\in\Omega}\{F({\mu}) + I({\mu})\}.
\end{equation}
The two bounds together will complete the proof of Theorem  \ref{thm:lapasymp}.
We begin with some results on existence and uniqueness of solutions of controlled PDE in \eqref{ratepde}.
\subsubsection{Existence and uniqueness of solutions to (\ref{ratepde})}\label{s1}
In this subsection, we present  two lemmas which establish the existence, uniqueness and continuity of solutions (with respect to $u$) of the ``controlled hydrodynamic limit" equation given in \eqref{ratepde}. The first lemma shows the existence of a solution to (\ref{ratepde}) if $u$ is smooth. The second lemma shows that if if $m_1$ and $m_2$ are solutions to (\ref{ratepde}) with same initial condition and
$u_1$ and $u_2$ in place of $u,$ then the distance between the corresponding elements of $\Omega$ is controlled by the $L^2$-distance between $u_1$ and $u_2.$ In particular, if we choose $u_1=u_2,$ this will imply that any solution to (\ref{ratepde}) is unique (within a suitable class). 
\begin{lemma}\label{existence}
Let $u\in C^\infty([0,T]\times S).$ Then for any $m_0\in L^1(S)$ satisfying $\int_S h\left(m_0(\theta)\right)d\theta < \infty$, there exists a unique solution to the PDE \eqref{ratepde}.
%
Furthermore,  $\mu(t,d\theta) = m(t,\theta)d\theta$, $t\in [0,T]$,   defines an element of $\Om$, the function $\theta\to m(t,\theta)$ is weakly differentiable and satisfies (\ref{hl1}) and (\ref{hl2}).
\end{lemma}

\begin{lemma}\label{continu}
	\begin{enumerate}[(i)]
\item Suppose that $\mu_1, \mu_2 \in \Omega$ are such that for  $0\leq t\leq T,$ $\mu_i(t,d\theta)$ has a density $m_i(t,\theta)$, namely, $\mu_i(t,d\theta)=m_i(t,\theta)d\theta,$ and that $m_i$ satisfies (\ref{hl1}) and (\ref{hl2})  for $i=1,2.$ 
Let $u_1, u_2 \in L^2([0,T]\times S),$ let $m_0\in L^1(S)$ satisfy $\int_S h \left(m_0(\theta)\right)d\theta < \infty$, and suppose that $m_1$ and $m_2$ are weak solutions to 
\eqref{ratepde} with $u$ replaced with $u_1$ and $u_2$ respectively and initial density $m_0$ as above.
Then $$
d_{*}(\mu_1, \mu_2)= \sup_{0 \le t \le T}d_{BL}(\mu_1(t,\cdot), \mu_2(t,\cdot)) \le e^{T/2} \|u_1-u_2\|_2.
$$
In particular, for any $u\in L^2([0,T]\times S),$ there is at most one $\mu\in\Omega$ with a density $m(t,\cdot)$ for $0\leq t \leq T$ that satisfies (\ref{hl1}), (\ref{hl2}) and solves (\ref{ratepde}) with $m_0$ as above.

\item Suppose $u_n$ is a sequence in $C^{\infty}([0,T]\times S)$ that converges to $u$ in $L^2([0,T]\times S)$. Define $\mu_n \in \Omega$ associated to $u_n$ by $\mu_n(t,d\theta) = m_n(t,\theta)d\theta$ where $m_n$ is the weak solution to \eqref{ratepde} with $u_n$ in place of $u$ and $m_0$ as in part (i). Suppose there exists a weak solution $m$ to \eqref{ratepde} associated with the limiting $u$ and the chosen $m_0$. Define $\mu \in \Omega$ by $\mu(t,d\theta) = m(t,\theta)d\theta$. Then $d_*(\mu_n,\mu)\to 0$  and  the sequence $\{\mu_n\}_{n \ge 1}$ is uniformly bounded in total variation norm, namely
$\sup_{n\in\NN} \sup_{0\le t\le T} \|\mu_n(t)\|_{TV} <\infty$.
In particular, $\mu_n$ converges to $\mu$ in $\Om$.
\end{enumerate}
\end{lemma}

Lemmas \ref{existence} and \ref{continu} will be proved in Section \ref{weakuniquesec}.
\subsubsection{Completing the proof of Laplace lower bound}\label{LLB}
The goal of this section is to show the bound in \eqref{eq:mainlaplowbd}
for all bounded and continuous $F$.
We begin with the following lemma.
\begin{lemma}
	\label{lem:smoothapprox}
	Suppose $\pi^* \in \clp_*(\RR\times S)$ such that $R(\pi^*\|\pi_0)<\infty$ and $u \in C^{\infty}([0,T]\times S)$. Define for $i=1, \ldots, N$, $\bar{\Phi}^N_i \in \clp(\RR)$ as in \eqref{muchoice} and $\psi^N_i\in L^2([0,T]:\RR)$ as 
	\begin{equation}\label{psichoice}
	\psi_i^N(t) \doteq \sum_{j=1}^{N}u\left(\frac{jT}{N}, \frac{i}{N}\right) \mathbb{I}_{(jT/N, (j+1)T/N]}(t), \; t \in [0,T].
	\end{equation}
 Define $\Pi^N(dx) = \bar\Phi_1^N(dx_1)\ldots\bar\Phi_N^N(dx_N)$ and $\mathbf{\bar{\Phi}}_i^N\doteq \bar\Phi_i^N$. 
Associated with $\Pi^N$ and $\{\psi^N_i\}$ as above, let $\bar \mu^N$ be defined as in
Section \ref{subsec:varrep}. Then 
\begin{equation}\label{ueq}
\lim_{N\rightarrow\infty} \frac{1}{N}\int_0^T\sum_{i=1}^N|\psi_i^N(t)|^2dt = \int_0^T\int_S |u(t,\theta)|^2 d\theta dt,
\end{equation}
\begin{align}\label{mutight}
\frac{1}{N}\sum_{i=1}^NR(\mathbf{\bar\Phi}_i^N\|\Phi) \le  
R(\pi^*\|\pi_0),\; \mbox{ for all } N \in \NN
\end{align}
and $\bar \mu^N$ converges to $\bar \mu$ in distribution in $\Om$ where $\bar \mu(t, d\theta) = m(t,\theta) d\theta$ and $m$ is the unique weak solution of \eqref{ratepde} with $u$ as above and $m_0(\theta) \doteq \int_{\RR} x \pi_1^*(dx|\theta)$, $\theta \in S$.
\end{lemma}
\begin{proof}
	The first statement in the lemma is immediate from the uniform continuity of $u$. The second is a consequence of Jensen's inequality:
	\begin{align}
	\frac{1}{N}\sum_{i=1}^NR(\mathbf{\bar\Phi}_i^N\|\Phi) &= \frac{1}{N}\sum_{i=1}^NR\left(N\int_{(i-1)/N}^{i/N}\pi_1^*(dx|\theta)d\theta\|\Phi\right)\nonumber\\
	&\leq \sum_{i=1}^N\int_{(i-1)/N}^{i/N}R(\pi_1^*(dx|\theta)\|\Phi)d\theta
	=R(\pi^*\|\pi_0). \label{mutightA} 
	\end{align}
	Now consider the final statement.  
From the convergence in \eqref{ueq} and  from the chain rule for relative entropies,
$$\frac{1}{N} R(\Pi^N \| \Phi^N) = \frac{1}{N}\sum_{i=1}^NR(\mathbf{\bar\Phi}_i^N\|\Phi) \le R(\pi^*\|\pi_0)<\infty,$$
and thus, by Lemma $\ref{mubartight}$,  $\{\bar\mu^N\}$ is tight.

	Let $\bar\mu$ be any subsequential weak limit of $\bar\mu^N$.  By Lemma \ref{UnifInt} and Lemma \ref{convtopistar}, $\bar\mu(0, d\theta)=m_0(\theta)d\theta,$ and by construction $u_N$ converges to $u$ in $L^2([0,T] \times S)$. By Theorem \ref{HL} we now see that $\bar\mu(t,d\theta)$ has a density $\bar m(t,\theta)$ for  $0\leq t\leq T$ and that $\bar m(t,\theta)$ solves (\ref{ratepde})  with $u$ as above and initial condition $m_0$.
	The unique solvability of this equation is a consequence of Lemma \ref{continu}.  The result follows. \end{proof}

	 %

We now complete the proof of the Laplace lower bound \eqref{eq:mainlaplowbd}.
Fix $F$ bounded and continuous, and let $\epsilon>0.$ Choose $\bar{\mu}^*\in \Om$ such that 
\begin{equation}\label{eq:approxlowbd}
F(\bar{\mu}^*) + I(\bar{\mu}^*) \leq \inf_{\mu \in \Omega} \{F(\mu) + I(\mu)\} + \epsilon,
\end{equation}
and then choose $u^* \in L^2([0,T]\times S)$ and $\pi^* \in \clp_*(\RR \times S)$ such that $\bar{\mu}^*\in\mathcal{M}_\infty(u^*,\pi^*)$ and 
\begin{equation}\label{eq:eq844}
I(\bar{\mu}^*) + \epsilon \geq \frac{1}{2}\left[\int_0^T\int_S |u^*(s,\theta)|^2 d\theta ds\right] + R(\pi^*\|\pi_0).
\end{equation}

Fix $\delta \in (0,1)$ and let $u^{**}\in C^{\infty}([0,T]\times S)$ be such that $\|u^{**}-u^*\|_2\leq \frac{\delta}{2(1+ \|u^*\|_2)}$. 
Let $m_0(\theta) \doteq \int_{\RR} x \pi^*_1(dx|\theta)$ for $\theta \in S$. Note that
$$\int_S h(m_0(\theta)) d\theta \le \int_S R(\pi^*_1(\cdot|\theta) \| \Phi) d\theta = R(\pi^*\|\pi_0) <\infty.$$
Therefore, by Lemma \ref{existence} there exists  $\bar\mu^{**} \in \Om$ such that $\bar\mu^{**}(t,d\theta)$ has a density $\bar m^{**}(t,\theta)$ for  $0\leq t \leq T,$ that  satisfies (\ref{hl1}) and (\ref{hl2}),
and is the unique weak solution of (\ref{ratepde}) (with $m$ replaced with $\bar m^{**}$) with the above choice of $m_0$ and
$u$ replaced by $u^{**}$.
In particular, $\bar\mu^{**}\in\mathcal{M}(u^{**},\pi^*).$ By the last statement in Lemma \ref{continu} and the continuity of $F$  we have that $|F(\bar\mu^*) -F(\bar\mu^{**})|\leq \epsilon$ if $\delta$ is chosen to be sufficiently small. 
Define $\bar{\Phi}_i^N$ and $\psi_i^N$ by \eqref{muchoice} and \eqref{psichoice}, respectively, with $(\pi, u)$ replaced with $(\pi^*, u^{**})$.
Let $\Pi^N$ be defined using $\pi^*$ as in the statement of Lemma \ref{lem:smoothapprox}. From Lemma \ref{lem:smoothapprox} it then follows that $\bar \mu^N$ associated with $(\Pi^N, \psi^N)$ converges
to $\bar \mu^{**}$ in distribution and \eqref{ueq}, \eqref{mutight} are satisfied.
Thus, by (\ref{expectvar}),  
\begin{align*}
&\limsup_{N\rightarrow\infty} - \frac{1}{N} \log \mathbb{E} \exp\left(-NF\left(\mu^N\right)\right) \\
&\leq \limsup_{N\rightarrow\infty} \bar{\mathbb{E}}_{\Pi^N}\left(F\left(\bar\mu^N\right)  + \frac{1}{N} \sum_{i=1}^N \left(R(\mathbf{\bar\Phi}_i^N\|\Phi)+\frac{1}{2}\int_0^T|\psi_i^N(s)|^2 ds\right)\right)\\
&\leq F(\bar\mu^{**}) + R(\pi^*\|\pi_0)+\frac{1}{2}\int_S\int_0^T|u^{**}(s,\theta)|^2dsd\theta\\
&\leq F(\bar\mu^{*}) + R(\pi^*\|\pi_0)+\frac{1}{2}\int_S\int_0^T|u^*(s,\theta)|^2dsd\theta + \epsilon + 2\delta\\
&\leq F(\bar\mu^*)+I(\bar\mu^*)+2\epsilon+2\delta\\
&\leq \inf_{\mu\in\Omega} \{F(\mu)+I(\mu)\}+3\epsilon+2\delta,
\end{align*} 
where the second inequality uses the convergence $\bar \mu^N\to \bar \mu^{**}$, the continuity of $F$, \eqref{ueq}, and \eqref{mutight}, the third inequality makes use of our choice 
of $\delta$, the fourth follows on using \eqref{eq:eq844} and the last inequality
uses \eqref{eq:approxlowbd}.
Sending $\delta$ and $\epsilon$ to $0$ completes the proof of the Laplace lower bound. \hfill \qed


\section{Proof of Lemma \ref{ALM}}
\label{sec:prooflemalm}
Let $\{\mu^n\}_{n=1}^\infty\subset \Gamma_{l,M}$. 
For each $n \ge 1,$ we can find $\pi^n\in\mathcal{P}_*(\mathbb{R}\times S)$ and $u^n\in L^2([0,T]\times S)$ such that $\mu^n \in \mathcal{M}^{\infty}(u^n,\pi^n)$ and 
\begin{equation}\label{nearoptimal}
R(\pi^n\|\pi_0) + \frac{1}{2}\int_0^T\int_S |u^n(t,\theta)|^2 d\theta ds\leq I(\mu^n) + \frac{1}{n}.
\end{equation}
Let for $n\in \NN$, $m^n_0(\theta)= \int_{\RR} x \pi^n_1(dx|\theta)$, $\theta \in S$. Then, for all $n\in \NN$,
$\int_{S} h(m^n_0(\theta)) d\theta \le R(\pi^n \|\pi_0) \le M+1.$
Using Lemmas \ref{existence} and \ref{continu} we may choose $\delta_n \in (0,1/n)$ and $u^{n,*}\in C^{\infty}([0,T]\times S)$ 
such that $\|u^n-u^{n,*}\|_2^2 \leq \delta_n$, and the unique weak solution $m^{n,*}$ of \eqref{ratepde} with $m_0 = m^n_0$
and $u= u^{n,*}$ has the property that $d_*(\mu^n,\mu^{n,*})\leq \frac{1}{n},$ where 
$\mu^{n,*}(t,d\theta) = m^{n,*}(t,\theta) d\theta$ for $t\in [0,T]$.
For $N\in \NN$, define  $\{\bar\Phi^{N,n}_i\}_{i=1}^N$ and $\{\psi^{N,n}_i\}_{i=1}^N$  by \eqref{muchoice} and 
\eqref{psichoice}, respectively,  replacing $(u,\pi)$ with $(u^{n,*}, \pi^n)$. Define $\Pi^{N,n}$ as we defined $\Pi^N$
in the statement of Lemma \ref{lem:smoothapprox}, with $\bar \Phi^{N}_i$ replaced with $\bar\Phi^{N,n}_i$, and let for each $n \in \NN$, the sequences
$\bar X^{N,n}$, $\bar \mu^{N,n}$  be constructed using $(\Pi^{N,n}, \psi^{N,n})$ as in Section \ref{subsec:varrep}. For each fixed $n$, from Lemma \ref{lem:smoothapprox}, $\bar \mu^{N,n}$ converges in probability, in $\Om$, as $N \rightarrow \infty$ to
$\mu^{n,*}$. From Lemma \ref{lem:topprops}(b) we must have $d_*(\bar \mu^{N,n}, \mu^{n,*}) \to 0$ in probability as
$n\to \infty$. 
Also, defining $\bar L^{N,n}$ as in \eqref{eq:lndxdt} with $\bar X^N$ replaced with $\bar X^{N,n}$ we see from Lemma \ref{convtopistar} that for each $n$, $\bar L^{N,n}$ converges to $\pi^n$ in probability as $N\to \infty$.
So for each $n$ we may choose $N_n$ such that
\begin{equation}\label{constar}
\bar{\mathbb{P}}\left(d_*(\bar\mu^{N_n,n},\mu^{n,*}) > 2^{-n}\right) < 2^{-n},
\end{equation}
\begin{equation}\label{constar3}
\|u^{n,*}- u^{n,N}\|_2^2 \le \frac{1}{n}
\end{equation}
where $u^{n,N}$ is defined by the right side of \eqref{eq:psitoun} by replacing $\psi^N$ with $\psi^{N,n}$,  and
\begin{equation}\label{constar2}
\bar{\mathbb{P}}\left(d_{BL}(\bar L^{N_n,n},\pi^{n}) > 2^{-n}\right) < 2^{-n},
\end{equation}
where $d_{BL}$ here denotes the bounded-Lipschitz distance on $\clp(\RR\times S)$.
Note that the sequence $(\Pi^{N_n,n})_{n=1}^\infty$ satisfies $R(\Pi^{N_n,n}\|\Phi^{N_n})\leq N_n(M+1)$
and, due to \eqref{constar3} and \eqref{nearoptimal},
$$\sup_{n\in \NN} \frac{1}{N_n} \sum_{i=1}^{N_n} \int_0^T |\psi^{N_n,n}_i(s)|^2 ds \le 4(M+3).$$
So again using Lemma \ref{mubartight}, the sequence $\bar\mu^{N_n,n}$ is tight.
Consider a subsequence, denoted again as $\{n\}$, along which $\bar\mu^{N_n,n}$ converges in distribution, in $\Om$, to some limit $\mu^*.$ By (\ref{nearoptimal}), $u^{n}$ is uniformly bounded in $L^2([0,T] \times S)$ and $\pi^n$ are tight, so we may restrict attention to a further subsequence (denoted again as $\{n\}$) along which $u^n$ (and therefore also $u^{n,*}$ and
$u^{N_n,n}$) converge weakly in $L^2([0,T] \times S)$ to some $u^*$ and $\pi^n$ converge weakly to some limit $\pi^*.$ Note that
from \eqref{constar2} we also have that $\bar L^{N_n,n}$ converges in probability to $\pi^*$. 
From the lower semi-continuity of relative entropy and \eqref{nearoptimal} we see that $\pi^* \in \clp_*(\RR\times S)$. Thus from Theorem  \ref{HL} we have that $\mu^* \in\mathcal{M}^\infty(u^*,\pi^*)$ a.s.
Furthermore,
\begin{align*}
I(\mu^*)&\leq R(\pi^*\|\pi_0) + \frac{1}{2}\int_0^T\int_S |u^*(t,\theta)|^2 d\theta dt\\
&\leq \liminf_{n\rightarrow\infty} \left(R(\pi^n\|\pi_0) + \frac{1}{2} \int_0^T\int_S |u^{n,*}(t,\theta)|^2 d\theta dt\right)\leq M
\end{align*}
and therefore $\mu^* \in \Gamma_{l,M}$.
Finally, from \eqref{constar} and the fact that along the subsequence $\bar\mu^{N_n,n}$ converges in distribution, in $\Om$, to  $\mu^*$ we have that, along the same subsequence,
$d_*(\mu^{n,*}, \mu^*) \rightarrow 0$ (in particular $\mu^*$ is non-random) and combining this with the fact  that $d_*(\mu^n,\mu^{n,*})\leq \frac{1}{n},$ we now have that $d_*(\mu^{n}, \mu^*) \rightarrow 0$ along the subsequence. Thus we have
constructed a subsequence of the original sequence $\{\mu^n\}$ that converges in $\Om^l$ to $\mu^* \in \Gamma_{l,M}$ which proves the result.
\hfill \qed

\section{Proof of Theorem \ref{HL}}\label{ProofHL}

\begin{proof}
	Proof is based on the proof of  \cite[Theorem 5.1]{guopapvar}, which is an analogous result for the uncontrolled process, and therefore we will only comment on steps that are different. 
Parts (i)-(iii) follow from  \cite[Lemma 6.3]{guopapvar} using the entropy bound \eqref{initentest}
given in Lemma \ref{entlemma} and Fubini's Theorem. 

Part (iv) follows from  \cite[Lemma 6.6]{guopapvar} using in addition to the entropy bound in Lemma \ref{entlemma}, the Dirichlet form bound in Lemma \ref{Ilemma},  Fubini's Theorem and the observation that
$$
\mathbb{E}_{\bar{Q}}\left[\int_0^T\int_{S}\left[\partial_{\theta}\left(h'(\bar{m}(t, \theta))\right)\right]^2 d \theta dt \right] = 
T \mathbb{E}_{\frac{1}{T}\int_0^T\bar{Q}_s ds}\left[\int_{S}\left[\partial_{\theta}\left(h'(m(\theta))\right)\right]^2 d \theta \right],
$$
where $\bar{Q}$ is the law of $\bar\mu,$ $\bar{Q}_t$ denotes the marginal of $\bar{Q}$ at time $t$ and $\mathbb{E}_{\bar{Q}}$ denotes expectation with respect to the underlying measure $\bar{Q}$ (similarly for $\mathbb{E}_{\frac{1}{T}\int_0^T\bar{Q}_s ds}$). In particular, on the right side, $\mu(d\theta)= m(\theta) d\theta$ is a $\clm_S$-valued random variable with probability law
$\frac{1}{T}\int_0^T\bar{Q}_s ds$.

We now prove (v). Define for $l \in \NN$, the cutoff function $\phi'_l$  given by 
\begin{equation}
  \phi'_l(x) \doteq
  \begin{cases}
                                   \phi'(x) & \text{if $|\phi'(x)|\leq l$} \\
                                   l & \text{if $\phi'(x)>l$} \\
  -l & \text{if $\phi'(x)<-l$},
  \end{cases} \label{eq:cutoff}
\end{equation}
and let 
$$
\widetilde{\phi'}_l(x) = \frac{1}{e^{\rho(\lambda)}}\int_{\mathbb{R}} e^{\lambda y - \phi(y)}\phi'_l(y)dy, \; \mbox{ where } \lambda = h'(x).
$$
Note that for $N\in \NN$, $u_N$ defined by \eqref{eq:psitoun} is a  $\mathcal{S}_{C_0}$-valued random variables and thus we can extract a subsequence which converges weakly, in distribution,  to some $u$ with
values  in $\mathcal{S}_{C_0}$.

Let $\epsilon >0$. For fixed $t \in [0,T]$ and a  smooth function $J$ on $S$,  define 
\begin{align}\label{fform}
H^{N,t}_{l, \epsilon} 
&\doteq \int_{S} J(\theta) \bar{\mu}_N(t,d\theta) - \int_{S} J(\theta)\bar{\mu}_N(0,d\theta)\nonumber\\
 &\qquad - \frac{1}{2N}\int_0^t\sum_{i=1}^N J''(i/N)\widetilde{\phi'}_l\left(\frac{\sum_{j=i-[N\epsilon]}^{j=i+[N\epsilon]}\bX_j^N(s)}{1 + [2N\epsilon]}\right) ds - \int_0^t\int_{S}J'(\theta)u_N(s,\theta) d\theta ds.
\end{align}
Note that, as $N\to \infty$, $H^{N,t}_{l, \epsilon}$ converges in distribution to
\begin{align*}
H^{t}_{l,\epsilon} &\doteq \int_{S} J(\theta) \bar{\mu}(t,d\theta) - \int_{S} J(\theta)\bar{\mu}(0,d\theta)\\
 &\quad- \frac{1}{2}\int_0^t\int_{S} J''(\theta)\widetilde{\phi'}_l\left(\frac{\bar{\mu}(s, [\theta- \epsilon, \theta + \epsilon])}{2\epsilon}\right) d\theta ds - \int_0^t\int_{S}J'(\theta)u(s,\theta) d\theta ds
\end{align*}
and therefore for each $l\in \NN$ and $\epsilon \in (0,\infty)$
\begin{equation}\label{limsup}
\bar{\mathbb{E}}\left[|H^{t}_{l,\epsilon}|\right] \le \limsup_{N \rightarrow \infty} \bar{\mathbb{E}}_{\Pi^N}\left[|H^{N,t}_{l, \epsilon}|\right].
\end{equation}
To prove (v) of the theorem, we first show that 
\begin{equation}
	\label{eq:limsuplimsup}
	\limsup_{l\to\infty} \limsup_{\epsilon\to 0} \limsup_{N \rightarrow \infty} \bar{\mathbb{E}}_{\Pi^N}\left[|H^{N,t}_{l, \epsilon}|\right]=0.
\end{equation}
 To see this,  write
\begin{align}
&\int_{S} J(\theta) \bar{\mu}_N(t,d\theta) - \int_{S} J(\theta)\bar{\mu}_N(0,d\theta) - \int_0^t\int_{S}J'(\theta)u_N(s,\theta) d\theta ds\nonumber\\
&\quad= \frac{1}{N}\sum_{i=1}^N J(i/N)\bX_i^N(t) - \frac{1}{N}\sum_{i=1}^N J(i/N)\bX_i^N(0) - \int_0^t\int_{S}J'(\theta)u_N(s,\theta) d\theta ds\nonumber\\
&\quad= \frac{N}{2}\int_0^t \sum_{i=1}^N \left(J((i+1)/N) - 2J(i/N) + J((i-1)/N)\right) \phi'(\bX_i^N(s))ds + M_N(t) \label{eq:martweak}
\end{align}
where $M_N$ is a martingale  given by
$$
M_N(t) \doteq \int_0^t\sum_{i=1}^N \left(J(i/N) - J((i-1)/N)\right)dB_i(s).
$$
Using a straightforward estimate on the second moment of $M_N$ (see  \cite[equation (5.3)]{guopapvar}) we see from \eqref{eq:martweak}, as in the proof of  \cite[equation (5.4)]{guopapvar}, that
\begin{multline}\label{ito1a}
 \lim_{N \rightarrow \infty}\bar{\mathbb{E}}_{\Pi^N}\left|\frac{1}{N}\sum_{i=1}^N J(i/N)\bX_i^N(t) - \frac{1}{N}\sum_{i=1}^N J(i/N)\bX_i^N(0) - \int_0^t\int_{S}J'(\theta)u_N(s,\theta) d\theta ds\right.\\
\left.- \frac{1}{2N}\int_0^t \sum_{i=1}^NJ''(i/N)\phi'(\bX_i^N(s))ds\right| = 0.
\end{multline}
Also, \cite[equation  (5.6)]{guopapvar} carries over verbatim (with $\psi_l$ replaced by $\phi'_l$) and we obtain
\begin{multline}\label{ito1}
\lim_{l \rightarrow \infty} 	 \lim_{N \rightarrow \infty}\bar{\mathbb{E}}_{\Pi^N}\left|\frac{1}{N}\sum_{i=1}^N J(i/N)\bX_i^N(t) - \frac{1}{N}\sum_{i=1}^N J(i/N)\bX_i^N(0) - \int_0^t\int_{S}J'(\theta)u_N(s,\theta) d\theta ds\right.\\
	\left.- \frac{1}{2N}\int_0^t \sum_{i=1}^NJ''(i/N)\phi'_l(\bX_i^N(s))ds\right| = 0.
\end{multline}
Recall that for $t\in [0,T]$, $\bar{\mathbb{Q}}_{\Pi^N}(t)$ denotes the probability law of $\bar X^N(t)= (\bar X^N_1(t), \ldots \bar X^N_N(t))$.
%
%
%
%
%
%
From Lemmas \ref{entlemma} and \ref{Ilemma} we see  that for some $C_1, C_2 \in (0,\infty)$
the density of $\frac{1}{T} \int_0^T \bar{\mathbb{Q}}_{\Pi^N}(t) dt$ lies in the class $\underline{A}_{N,C_1, C_2}$ defined  in \cite[page 43]{guopapvar} for all $N\in \NN$. Thus, we can apply  \cite[Theorem 4.1]{guopapvar} to obtain
\begin{equation}\label{block}
\lim_{\epsilon \rightarrow \infty} \limsup_{N \rightarrow \infty}\bar{\mathbb{E}}_{\Pi^N}\left|\frac{1}{2N}\int_0^t \sum_{i=1}^NJ''(i/N)\phi'_l(\bX_i^N(s))ds - \frac{1}{2}\int_0^t\sum_{i=1}^N J''(i/N)\widetilde{\phi'}_l\left(\frac{\sum_{j=i-[N\epsilon]}^{j=i+[N\epsilon]}\bX_j^N(s)}{1 + [2N\epsilon]}\right) ds\right| = 0
\end{equation}
for every $l$.
Using \eqref{ito1} and \eqref{block} in \eqref{fform}, we obtain \eqref{eq:limsuplimsup}.
This, combined with \eqref{limsup}, yields
$
\lim_{l \rightarrow \infty} \limsup_{\epsilon \rightarrow 0}\bar{\mathbb{E}}|H^{t}_{l, \epsilon}|=0.
$
The limit  as $\epsilon \rightarrow 0$ can be take inside the expectation because the third term in $H^{t}_{l, \epsilon}$ is uniformly bounded
and converges to $\frac{1}{2}\int_0^t\int_{S} J''(\theta)\widetilde{\phi'}_l\left(\bar{m}(s, \theta)\right) d\theta ds$.
Together with part (i), this yields
\begin{multline*}
\liminf_{l \rightarrow \infty} \left|\int_{S} J(\theta) \bar{m}(t,\theta)d\theta - \int_{S} J(\theta)\bar{m}(0,\theta)d\theta\right.\\
\left. - \frac{1}{2}\int_0^t\int_{S} J''(\theta)\widetilde{\phi'}_l\left(\bar{m}(s, \theta)\right) d\theta ds - \int_0^t\int_{S}J'(\theta)u(s,\theta) d\theta ds\right|=0.
\end{multline*}
From \cite[Lemma 6.5]{guopapvar}, for every $\sigma>0$
$$|\widetilde{\phi'}_l(x)| \le \frac{1}{\sigma} \log \int e^{\sigma |\phi'(y)| - \phi(y)} dy + \frac{h(x)}{\sigma}.$$
Also, from \cite[Lemma 6.4]{guopapvar}, $\widetilde{\phi'}_l(x) \to h'(x)$ as $l\to \infty$. Combining these two facts with part (iii) of the theorem, and sending $l\to \infty$, we see that
\begin{align*}
 \int_{S} J(\theta) \bar{m}(t,\theta)d\theta - \int_{S} J(\theta)\bar{m}(0,\theta)d\theta
 - \frac{1}{2}\int_0^t\int_{S} J''(\theta)h'\left(\bar{m}(s, \theta)\right) d\theta ds - \int_0^t\int_{S}J'(\theta)u(s,\theta) d\theta ds=0
\end{align*}by the dominated convergence theorem
which proves part (v).
\end{proof}
\section{Entropy and Dirichlet Form Bounds}\label{entsection}
In this section, we establish the key bounds on relative entropy and Dirichlet forms stated in Lemmas  \ref{entlemma} and  \ref{Ilemma}. A key ingredient in the proof of Lemma \ref{Ilemma} is a suitable regularity of the density
of the controlled process $\bar X^N(t)$. This is studied in Section \ref{sec:regdens} and proofs of Lemmas \ref{entlemma} and \ref{Ilemma} are given in Section \ref{tech}.
Throughout this section $\Pi^N, \psi^N$ and $\bar X^N$ are as in the statement of Lemma \ref{entlemma}.
\subsection{Regularity of the Density of $\bar X^{N}(t)$}
\label{sec:regdens}
In this subsection we will show that $\bar X^{N}(t)$ has a density $\bar p_N(t,x)$ with respect to the measure $\Phi^N(dx)$ which is continuously differentiable in time and twice continuously differentiable in space along the vector fields $V_1,\dots,V_N$. This will allow us to apply It\^{o}'s formula in the following subsection. The following is the main regularity lemma of this subsection.

Recall that $\bar \QQ_{\Pi^N}(t)$ denotes the probability law of $\bar X^{N}(t)$. The following lemma holds for each fixed $N\in \NN$. Recall that for each $N$, the control $\psi^N \in \cla^N_s(\sys_{\Pi^N})$ is defined in terms of
a partition $0=t_0 \le \cdots \le t_k=T$ and random variables $\{U_{ij}\}$ as in \eqref{eq:defnsimpcont} and that these random variables satisfy a uniform bound as in \eqref{controlas}. The partition and the uniform bound may depend on the control. This special structure of the control is important in the proof.

\begin{lem}\label{reglemma}
	For every $t>0$, $\bar \QQ_{\Pi^N}(t)$ is absolutely continuous with respect to $\Phi^N$.
Let $\bar p(t,\cdot)$ denote the corresponding density with respect to the measure $\Phi^N(dx)$. Then $\bar p(\cdot,\cdot)$ is continuously differentiable in time and twice continuously differentiable in space along the vector fields $V_1,\dots,V_N$ for $\{t \in(t_j, t_{j+1}): 0\le j \le K-1\}$ and $x \in \Reals^N$.
\end{lem}
\begin{proof}
	Some steps in the proof are standard PDE estimates but we give full details to keep the presentation self-contained.
As $\phi$ is twice continuously differentiable, the measure $\Phi^N(dx)$ has a twice continuously differentiable density $f_N(x) =  \exp{\left(-\sum_{i=1}^N\phi(x_i)\right)}$ with respect to Lebesgue measure. 
It will be convenient to work with the law of $(\bar X^{N}_1(t),\dots, \bar X^{N}_{N-1}(t), \bar S^{N}(t))$ where $\bar S^{N}(t) = \bar X^{N}_1(t) + \cdots + \bar X^{N}_{N}(t)$. Let $\Sigma_S = \{(x_1,\dots, x_N) \in \mathbb{R}^N: x_1 + \cdots + x_N = S\}$. As the vector fields $V_1,\dots,V_N$ are parallel to $\Sigma_S$, $\bar S^{N}(t) = \bar S^{N}(0)$ for all $t \ge 0$. Therefore, the process $(\bar X^{N}_1(t),\dots, \bar X^{N}_{N-1}(t), S^{N}(t))$ started at $(y_1, \dots, y_{N-1}, s)$ lives in the hyperplane $\Sigma_s$ for all time and by Girsanov's Theorem (see for example \cite{karshr}), for $t>0$, $(\bar X^{N}_1(t),\dots, \bar X^{N}_{N-1}(t))$ has a density $\{\bar q(t, x_1,\dots, x_{N-1} \mid (y_1, \dots, y_{N-1},s)): (x_1,\dots, x_{N-1}) \in \mathbb{R}^{N-1}\}$ with respect to the Lebesgue measure on $\mathbb{R}^{N-1}$. Denoting the density of $(\bar X^{N}_1(0),\dots, \bar X^{N}_{N-1}(0), \bar S^{N}(0))$ by $\zeta_N(y_1,\dots, y_{N-1},s)$, it is straightforward to check that the law of $(\bar X^{N}_1(t),\dots, \bar X^{N}_{N-1}(t), \bar S^{N}(t))$ has a density with respect to Lebesgue measure on $\mathbb{R}^N$ given by
\begin{equation}\label{hyp1}
\bar q(t,x_1, \dots, x_{N-1},s) = \int_{\mathbb{R}^{N-1}}\bar q(t, x_1,\dots, x_{N-1} \mid y_1, \dots, y_{N-1}, s) \zeta_N(y_1,\dots, y_{N-1},s)dy_1 \dots dy_{N-1}.
\end{equation}
In particular, $(\bar X^{N}_1(t),\dots, \bar X^{N}_{N}(t))$ has a density with respect to the Lebesgue measure which we write as $f_N(\cdot)\bar p(t, \cdot)$.
Now, consider for $j=0, 1, \ldots k-1$ the time interval $(t_j, t_{j+1})$. For brevity, write $y^{(N-1)} = (y_1, \dots, y_{N-1})$ and $x^{(N-1)} = (x_1, \dots, x_{N-1})$. As $V_1,\dots,V_N$ are parallel to $\Sigma_s$ for any $s \in \mathbb{R}$, to prove Lemma \ref{reglemma}, it suffices to prove that $(t, x^{(N-1)},s)\mapsto\bar q(t,x^{(N-1)},s)$ is continuously differentiable in time and twice continuously differentiable in the space variables $x_1, \dots, x_{N-1}$. 
By the representation \eqref{hyp1} and the dominated convergence theorem, it suffices to prove that for each $(y^{(N-1)}, s)$, $(t, x^{(N-1)})\mapsto \bar q(t, x^{(N-1)} \mid y^{(N-1)}, s)$ is once continuously differentiable in $t$ and twice continuously differentiable in $x^{(N-1)}$, and
for any three compact intervals $I \subset (t_j,t_{j+1})$, $J \in \Reals$ and $K \subset \Reals^{N-1}$,
\begin{multline}\label{dom1}
\sup_{y^{(N-1)}\in \Reals^{N-1}}\sup_{(t,s, x^{(N-1)}) \in I \times J \times K} \Big(|\bar q(t, x^{(N-1)} \mid y^{(N-1)}, s)| + |\partial_t \bar q(t, x^{(N-1)} \mid y^{(N-1)}, s)| \\
 + \sum_{i=1}^{N-1}|\partial_i \bar q(t, x^{(N-1)} \mid y^{(N-1)}, s)| + \sum_{i,k=1}^{N-1} |\partial_i \partial_k \bar q(t, x^{(N-1)} \mid y^{(N-1)}, s)|\Big) < \infty,
\end{multline}
where $\partial_i$ denotes partial derivative with respect to $x_i$.
Let $\alpha(dx^{(N-1)}_j, d\u_j \mid y^{(N-1)}, s)$ be the joint distribution of $(\bar X^{N}_1(t_j),\dots, \bar X^{N}_{N-1}(t_j))$ and $(U_{ij})_{1 \le i \le N}$ when $(\bar X^{N}_1(0),\dots, \bar X^{N}_{N-1}(0), S^{N}(0))=(y_1, \dots, y_{N-1}, s)$. Recall from \eqref{controlas} that $|U_{ij}| \le C$ for all $1 \le i,j \le N$. We will denote the box $[-C,C]^N$ by $\ball^N$. For $t \in (t_j, t_{j+1})$ and $x^{(N-1)} \in \Reals^{N-1}$, $\bar q$ has the representation
\begin{align}\label{align:densityu}
\bar q(t, x^{(N-1)} \mid y^{(N-1)}, s) &= \int_{\Reals^{N-1} \times \ball^{N}} \bar q(x^{(N-1)}_j,s,\u_j;t-t_j,x^{(N-1)}) \alpha(dx^{(N-1)}_j, d\u_j \mid y^{(N-1)}, s).
\end{align}
Here $\bar q(x^{(N-1)}_j,s,\u_j;t, \cdot)$ is the density with respect to Lebesgue measure of the process on $\Reals^{N-1}$ at time $t$ started from $x^{(N-1)}_j$ whose generator is given by
$$
\L^{\u_j,s}(x_1, \dots, x_{N-1})\doteq \frac{N^2}{2} \sum_{i=1}^N \hat V_i^2 - \frac{N^2}{2}\sum_{i=1}^N\left[\phi'(x_i) - \phi'(x_{i+1}) \right] \hat V_i - N\sum_{i=1}^N \u_{(i+1)j}\hat V_i
$$
where $\hat V_i = \partial_i - \partial_{i+1}$ for $1 \le i \le N-2$, $\hat V_{N-1} = \partial_{N-1}$ and $\hat V_N = -\partial_1$ are the projections of the vector fields $V_1, \dots, V_N$ onto $\mathbb{R}^{N-1}$, $x_N = s-x_1-\dots-x_{N-1}$, and $\u_{N+1,j}= \u_{1j}$. Thus, by the representation \eqref{align:densityu} and the dominated convergence theorem, in order to prove the lemma it suffices to show that
 $(t,x)\mapsto \bar q(z^{(N-1)},s,\u;t, x)$ is once continuously differentiable in the $t$ variable and twice continuously differentiable in $x$, on $(t_j, t_{j+1}) \times \RR^{N-1}$
for every $z^{(N-1)},s,\u$ and $j$, and for any three compact intervals $I \subset (t_j, t_{j+1})$, $J \subset \mathbb{R}$ and $K \subset \Reals^{N-1}$,
\begin{align}
&\sup_{\u \in \ball^{N}}\sup_{z^{(N-1)} \in \Reals^{N-1}}\sup_{(t, s, x^{(N-1)}) \in I \times J \times K} \Big(|\bar q(z^{(N-1)},s,\u;t,x^{(N-1)})| + |\partial_t \bar q(z^{(N-1)},s,\u;t,x^{(N-1)})| \nonumber\\
&\quad\quad  + \sum_{i=1}^{N-1}|\partial_i \bar q(z^{(N-1)},s,\u;t,x^{(N-1)})| + \sum_{i,k=1}^{N-1} |\partial_i \partial_k \bar q(z^{(N-1)},s,\u;t,x^{(N-1)})|\Big) < \infty. \label{eq:mainestim}
\end{align}
In order to prove the above statements we will use the broad outline of the proof of part (a) on page 21 of \cite{jorkinott}. 
Fix $\u = (u_1,\dots,u_N) \in \ball^{N}$. To avoid cumbersome notation, we will write $x$ for $x^{(N-1)} \in K$ and $z$ for $z^{(N-1)} \in \mathbb{R}^{N-1}$. 
For any $t^* \in (t_j,t_{j+1})$ and any  compactly supported smooth function $\zeta$ on $[t_j,t^*] \times \Reals^{N-1}$,
\begin{align}\label{align:dummy}
\int_{\Reals^{N-1}}\zeta(t^*,x)\bar q(z,s,\u;t^*,x)dx &= \int_{t_j}^{t^*}\int_{\Reals^{N-1}}\left(\partial_t \zeta + \L^{\u,s}\zeta\right)\bar q(z,s,\u;t,x)dx dt\nonumber\\
& \quad \quad + \int_{\Reals^{N-1}}\zeta(t_j,x)\bar q(z,s,\u;t_j,x)dx.
\end{align}
Let $\eta$ be a smooth compactly supported function on $\Reals^{N-1}$ such that $\eta(x)=1$ for all $x \in K.$ 
Then for any smooth $\zeta$ on  $[t_j,t^*] \times \Reals^{N-1}$, we have on  
substituting $\eta \zeta$ in place of $\zeta$ in the above equation, 
\begin{align}
&\int_{\Reals^{N-1}}\zeta(t^*,x)\eta(x)\bar q(z,s,\u;t^*,x)dx \nonumber\\
&\quad =  \int_{t_j}^{t^*}\int_{\Reals^{N-1}}\left(\partial_t \zeta + \frac{N^2}{2}\sum_{i=1}^N \hat V_i^2\zeta\right)\bar q(z,s,\u;t,x)\eta(x)dx dt\nonumber\\
&\quad\quad+ \int_{t_j}^{t^*}\int_{\Reals^{N-1}}\left(\L^{\u,s}\eta(x)\right)\bar q(z,s,\u;t,x)\zeta(t,x)dx dt\nonumber\\
&\quad\quad+ \int_{t_j}^{t^*}\int_{\Reals^{N-1}}\sum_{i=1}^N\left(N^2 \hat V_i \eta(x) + \frac{N^2}{2}(V_i \log f_N)(x) \eta(x) - Nu_i \eta(x)\right)\bar q(z,s,\u;t,x)\hat V_i\zeta(t,x)dx dt\nonumber\\
&\quad\quad+ \int_{\Reals^{N-1}}\eta(x)\zeta(t_j,x)\bar q(z,s,\u;t_j,x)dx,\label{testeq}
\end{align}
where we have used the fact that $f_N(x) =  \exp{\left(-\sum_{i=1}^N\phi(x_i)\right)}$ is the density of $\Phi^N(dx)$ with respect to Lebesgue measure on $\Reals^N$.
Next consider the equation
\begin{align}\label{align:PDEnodrift}
\partial_t G(t, x) = \frac{N^2}{2}\sum_{i=1}^N \hat V_i^2G (t, x),
\ \ G(0,x) = \delta_{0}(x).
\end{align}
It can be checked that \eqref{align:PDEnodrift} is solved by the density, with respect to Lebesgue measure on $\Reals^{N-1}$, of the process $N(B_1-B_N,B_2 - B_1, \dots, B_{N-1} - B_{N-2})$ (where $(B_1,\dots, B_N)$ is an $N$-dimensional Brownian motion), given by
$$
G(t,x)\doteq \frac{1}{|\Sigma|^{1/2}(2\pi tN^2)^\frac{N-1}{2}} \exp\left(-\frac{1}{2tN^2}x^{T} \Sigma^{-1} x\right)
$$
where $\Sigma = \left(\Sigma_{ij}\right)_{1 \le i,j \le N-1}$ is given by $\Sigma_{ii}=2, \Sigma_{ij} = -1$ if $|i-j|=1$ and $\Sigma_{ij}=0$ otherwise. For fixed $x^* \in \Reals^{N-1}$ and each $\delta>0$, using the function
$$
\zeta_{\delta}(t,x)\doteq G(t^*+\delta-t, x^*-x), \; t \in [t_j, t^*], \; x \in \RR^{N-1},
$$
 in place of $\zeta$ in \eqref{testeq} and taking the limit $\delta \downarrow 0$ in $L^1(\RR^{N-1})$, we have for a.e. $x^*$
\begin{multline}\label{align:maineq}
\eta(x^*)\bar q(z,s,\u;t^*,x^*) =\int_{t_j}^{t^*}\int_{\Reals^{N-1}}\L^{\u,s}\eta(x)\bar q(z,s,\u;t,x)G(t^*-t,x^*-x)dx dt\\
 + \int_{t_j}^{t^*}\int_{\Reals^{N-1}}\sum_{i=1}^N\left(N^2\hat V_i \eta(x) + \frac{N^2}{2}(V_i \log f_N)(x) \eta(x)- Nu_i \eta(x)\right)\bar q(z,s,\u;t,x)\hat V_iG(t^*-t,x^*-x)dx dt\\
 + \int_{\Reals^{N-1}}\eta(x)\bar q(z,s,\u;t_j,x)G(t^*-t_j,x^*-x)dx.
\end{multline}
One can readily obtain the following estimates on $G$
\begin{align}\label{align:heatest}
||G(t,\cdot)||_m = \gamma_mt^{\left(\frac{1}{m}-1\right)\frac{N-1}{2}}, \ \ 
||\partial_iG(t,\cdot)||_m = \gamma_mt^{\frac{N-1}{2m}-\frac{N}{2}},\; i = 1, \ldots N,
\end{align}
for $m>1,$
where $\gamma_m \in (0,\infty)$  depends only on $m$ and $||\cdot||_m$ denotes the $L^m$ norm on $\RR^N$. Write 
$$f^{\u,s}_1(x)=\L^{\u,s}\eta(x), f_2^{\u,s,(i)}(x) = N^2\hat V_i \eta(x) -\frac{N^2}{2}(V_i \log f_N)(x) \eta(x) - Nu_i \eta(x), f_3(x)= \eta(x).$$
 Note that $f^{\u,s}_1$ and $f_3$ are compactly supported smooth functions and $f_2^{\u,s,(i)}$ are compactly supported $C^1$ functions (as $\phi$, and hence $f_N$, is a $C^2$ function). Moreover,  the functions $f^{\u,s}_1, f_2^{\u,s,(i)}, f_3$ and their derivatives are uniform bounded for $(\u, s) \in \ball^N \times J$.

We will now show that for any $m>1$ and compact $K \subset \RR^{N-1}$, $I\subset (t_j, t_{j+1})$, there is a $\theta_m\in (0,\infty)$ such that
\begin{align}\label{align:locintmain}
\sup_{(s,\u,z) \in J\times \ball^N \times\RR^{N-1}}\sup_{t^* \in I}||\mathbb{I}_K(\cdot)\bar q(z,s,\u;t^*,\cdot)||_m \le \theta_m.
\end{align}
This will be done by a standard bootstrapping procedure and an iterative application of Young's inequality. Fix a $m\in (1, \frac{N-1}{N-2})$, then from \eqref{align:maineq}
we have by an application of Young's inequality, for any $t^* \in (t_j, t_{j+1})$
\begin{align*}
||\eta(\cdot)\bar q(z,s,\u;t^*,\cdot)||_m &\le \kappa_1\int_{t_j}^{t^*}||f^{\u,s}_1(\cdot)\bar q(z,s,\u;t,\cdot)||_1||G(t^*-t,\cdot)||_m dt\nonumber\\
&\quad + \kappa_1\int_{t_j}^{t^*}\sum_{i=1}^N||f_2^{\u,s,(i)}(\cdot)\bar q(z,s,\u;t,\cdot)||_1||\hat V_iG(t^*-t,\cdot)||_m dt\nonumber\\
&\quad + \kappa_1||f_3(\cdot)\bar q(z,s,\u;t_j,\cdot)||_1||G(t^*-t_j,\cdot)||_m.
\end{align*}
As $f^{\u,s}_1$, $f_2^{\u,s,(i)}$ and $f_3$ are compactly supported functions that do not depend on $z$ and are bounded by a finite constant that does not depend on $\u, s$; $\eta(x)=1$ for all $x \in K$; and $||\bar q(z,s,\u;t,\cdot)||_1=1$ for all $t \in (t_j, t_{j+1})$  and all $s,\u,z$,  we have from the estimates in \eqref{align:heatest} and the above equation,
that \eqref{align:locintmain} holds for any 
$m \in (1, \frac{N-1}{N-2})$. 
Next, for
%
%
%
$m,n \in (1, \frac{N-1}{N-2})$ and $l$ satisfying $1 + \frac{1}{l} = \frac{1}{m} + \frac{1}{n}$, applying  Young's inequality in \eqref{align:maineq} gives us, for any $t'_j\in (t_j, t^*)$
\begin{align*}
||\eta(\cdot)\bar q(z,s,\u;t^*,\cdot)||_l &\le \kappa_2\int_{t'_j}^{t^*}||f^{\u,s}_1(\cdot)\bar q(z,s,\u;t,\cdot)||_m||G(t^*-t,\cdot)||_n dt\nonumber\\
&\quad + \kappa_2\int_{t'_j}^{t^*}\sum_{i=1}^N||f_2^{\u,s,(i)}(\cdot)\bar q(z,s,\u;t,\cdot)||_m||\hat V_iG(t^*-t,\cdot)||_n dt\nonumber\\
&\quad + \kappa_2||f_3(\cdot)\bar q(z,s,\u;t_j,\cdot)||_m||G(t^*-t'_j,\cdot)||_n.
\end{align*}
From the estimate in \eqref{align:heatest} and that \eqref{align:locintmain} holds for every compact $I \subset (t_j, t_{j+1})$,
 the right-hand side in the above equation is seen to be bounded by a finite constant independent of $\u \in \ball^N, z \in \Reals^{N-1}$ $t^* \in I$, $s \in J$.
Thus we have proved \eqref{align:locintmain} for any $m \in (1,  \frac{N-1}{N-3})$. This bootstrapping argument can be applied repeatedly to establish \eqref{align:locintmain} for all $m>1.$

Now, applying H\"older's inequality in \eqref{align:maineq} with $m,n$ such that $n \in (1,\frac{N-1}{N-2})$ and $\frac{1}{m} + \frac{1}{n}=1$, we get that for any $x^* \in K$, $t^* \in I$,
and $t'_j> t_j$ such that $I \subset (t'_j, t_{j+1})$
\begin{align}\label{align:first}
|\bar q(z,s,\u;t^*,x^*)| &\le \kappa_3\int_{t'_j}^{t^*}||f^{\u,s}_1(\cdot)\bar q(z,s,\u;t,\cdot)||_m||G(t^*-t,\cdot)||_n dt\nonumber\\
&\quad + \kappa_3\int_{t'_j}^{t^*}\sum_{i=1}^N||f_2^{\u,s,(i)}(\cdot)\bar q(z,s,\u;t,\cdot)||_m||\hat V_iG(t^*-t,\cdot)||_n dt\nonumber\\
&\quad + \kappa_3||f_3(\cdot)\bar q(z,s,\u;t'_j,\cdot)||_m||G(t^*-t'_j,\cdot)||_n \le \kappa_4,
\end{align}
where $\kappa_4$ does not depend on $\u \in \ball^N, z \in \Reals^{N-1}$, $t^* \in I$, $s \in J$, $x^* \in K$, and the last bound holds since \eqref{align:locintmain} holds for all compact $K$ and $I$. 

To establish the existence of the derivatives $(\partial_i \bar q(z,s,\u;t,\cdot))_{1\le i \le N-1}$ and a result analogous to \eqref{align:first} for the derivatives, we will need to establish the H\"older continuity of $\bar q(z,s,\u;t,\cdot)$ on $K$. For $x_1, x_2 \in K$, we  use the representation \eqref{align:maineq} to write
\begin{multline*}
\bar q(z,s,\u;t^*,x_1) - \bar q(z,s,\u;t^*,x_2)\\
= \int_{t'_j}^{t^*}\int_{\Reals^{N-1}}f^{\u,s}_1(y)\bar q(z,s,\u;t,y)\left(G(t^*-t,x_1-y)-G(t^*-t,x_2-y)\right) dy dt\\
+ \int_{t'_j}^{t^*}\sum_{i=1}^N\int_{\Reals^{N-1}}f_2^{\u,s,(i)}(y) \bar q(z,s,\u;t,y)\left(\hat V_iG(t^*-t,x_1 - y) - \hat V_iG(t^*-t,x_2 - y)\right)dy dt\\
+ \int_{\Reals^{N-1}}f_3(y)\bar q(z,s,\u;t_j,y)\left(G(t^*-t'_j,x_1 - y) - G(t^*-t'_j,x_2 - y)\right)dy.
\end{multline*}
Take $n \in (1, \frac{N-1}{N-2})$ and $m$ such that $\frac{1}{m} + \frac{1}{n}=1$. Using the estimate \eqref{align:locintmain}
(for a compact set $\widetilde{K}$ which contains the support of $\eta$) and H\"older's inequality in the above representation, we have that for some $\kappa_5\in (0,\infty)$
and all $\u \in \ball^N, z \in \Reals^{N-1}$, $t^* \in I$, $s \in J,$
\begin{align}
|\bar q(z,s,\u;t^*,x_1) - \bar q(z,s,\u;t^*,x_2)| &\le \kappa_5\int_{t'_j}^{t^*}||G(t^*-t,x_1-\cdot)-G(t^*-t,x_2-\cdot)||_n dt\nonumber\\
&\quad + \kappa_5\int_{t'_j}^{t^*}\sum_{i=1}^N ||\hat V_iG(t^*-t,x_1 - \cdot) - \hat V_iG(t^*-t,x_2 - \cdot)||_n dt\nonumber\\
&\quad + \kappa_5||G(t^*-t'_j,x_1 - \cdot) - G(t^*-t'_j,x_2 - \cdot)||_n.\label{eq:holdest}
\end{align}
Now, by standard computations (for example, see  \cite[Chapter 4, Section 2]{ladsolura}), we see that there exist $\tilde \gamma_n\in (0,\infty)$  and $\theta>0$ such that for any  $x_1, x_2\in K$,
$t^* \in (t_j,t_{j+1})$ and $t \in [t_j, t^*)$,
\begin{align*}
||G(t^*-t,x_1-\cdot)-G(t^*-t,x_2-\cdot)||_n &\le \tilde \gamma_n|x_1 - x_2|^{\theta},\\
||\hat V_iG(t^*-t,x_1-\cdot)-\hat V_iG(t^*-t,x_2-\cdot)||_n &\le \tilde \gamma_n|x_1 - x_2|^{\theta}, \ \ 1 \le i\le N.
\end{align*}
This, in view of \eqref{eq:holdest} implies that for every compact $K\subset \RR^{N-1}$ there exists $\kappa_6\in(0,\infty)$ such that for all $\u \in \ball^N, z \in \Reals^{N-1}$ $t^* \in I$, $s \in J$, and $x_1, x_2 \in K$,
\begin{align}\label{align:Holder}
|\bar q(z,s,\u;t^*,x_1) - \bar q(z,s,\u;t^*,x_2)| & \le \kappa_6|x_1 - x_2|^{\theta}.
\end{align}
To see how H\"older continuity implies the existence of the derivatives $(\partial_i \bar q(z,s,\u;t,\cdot))_{1\le i \le N-1}$, note that although $\int_{t_j}^{t^*}||\partial_i\partial_kG(t^*-t, \cdot)||_1dt = \infty$, for any $\theta>0$, there is a $\gamma_{\theta} \in (0,\infty)$ such that for all $\bar t \in [t_j, t^*]$
\begin{align}\label{align:Gbound}
\int_{\bar t}^{t^*}\int_{\Reals^{N-1}}|z|^{\theta}|\partial_i\partial_kG(t^*-t, z)|dz dt= \gamma_{\theta}\int_{\bar t}^{t^*}\frac{1}{(t^*-t)^{1-\frac{\theta}{2}}}dt < \infty.
\end{align}
We will use this fact to prove the existence of the partial derivatives of $\bar q$. For $0<h < t^*-t'_j$, (where as before  $t'_j> t_j$ such that $I \subset (t'_j, t_{j+1})$) define the function
\begin{align*}
\bar q_h(z,s,\u;t^*,x^*) &\doteq \int_{t'_j}^{t^*-h}\int_{\Reals^{N-1}}f^{\u,s}_1(x)\bar q(z,s,\u;t,x)G(t^*-t,x^*-x)dx dt\nonumber\\
&\quad + \int_{t'_j}^{t^*-h}\int_{\Reals^{N-1}}\sum_{i=1}^Nf_2^{\u,s,(i)}(x)\bar q(z,s,\u;t,x)\hat V_iG(t^*-t,x^*-x)dx dt\nonumber\\
&\quad + \int_{\Reals^{N-1}}f_3(x)\bar q(z,s,\u;t_j,x)G(t^*-t'_j,x^*-x)dx.
\end{align*}
From \eqref{align:locintmain} it is clear that $\bar q_h(z,s,\u;t^*,\cdot)$ converges uniformly to $\bar q(z,s,\u;t^*,\cdot)$ on $K$ as $h \rightarrow 0$. By the smoothness of the map $(t,x) \mapsto G(t^*-t, x)$ in an open set containing $[t_j, t^*-h] \times K$ and using the estimate \eqref{align:locintmain} once again, we obtain for $1\le k \le N-1$,
\begin{multline*}
\partial_k\bar q_h(z,s,\u;t^*,x^*) =\int_{t_j}^{t^*-h}\int_{\Reals^{N-1}}f^{\u,s}_1(x)\bar q(z,s,\u;t,x)\partial_kG(t^*-t,x^*-x)dx dt\\
 + \int_{t_j}^{t^*-h}\int_{\Reals^{N-1}}\sum_{i=1}^N(f_2^{\u,s,(i)}(x)\bar q(z,s,\u;t,x) - f_2^{\u,s,(i)}(x^*)\bar q(z,s,\u;t,x^*))\partial_k \hat V_iG(t^*-t,x^*-x)dx dt\\
+ \int_{\Reals^{N-1}}f_3(x)\bar q(z,s,\u;t_j,x)\partial_kG(t^*-t_j,x^*-x)dx.
\end{multline*}
Here, the adjustment in the second term is justified because $\int_{\Reals^{N-1}}\partial_j\hat V_iG(t^*-t,x^*-x)dx=0$ as $G(t^*-t,\cdot)$ is a probability density. By the uniform H\"older continuity of $\bar q$ on $K$ for an arbitrary compact $K$ given in \eqref{align:Holder}, the $C^1$ property of $f_2$, and the estimate \eqref{align:Gbound}, we conclude that $\partial_k\bar q_h(z,s,\u;\cdot,\cdot)$ converges uniformly to the right-hand side of the above equation with $h=0$ on $I\times K$ as $h \downarrow 0$. From this and the continuity of $(t,x)\mapsto \bar q_h(z,s,\u;t,x)$, we conclude that $\partial_k\bar q(z,s,\u;\cdot,\cdot) $ exists, is continuous, and for every $(t^*,x^*) \in I\times K$ takes  the form,
\begin{multline}\label{align:derivativeform}
\partial_k\bar q(z,s,\u;t^*,x^*) =\int_{t'_j}^{t^*}\int_{\Reals^{N-1}}f^{\u,s}_1(x)\bar q(z,s,\u;t,x)\partial_kG(t^*-t,x^*-x)dx dt\\
 + \int_{t'_j}^{t^*}\int_{\Reals^{N-1}}\sum_{i=1}^N(f_2^{\u,s,(i)}(x)\bar q(z,s,\u;t,x) - f_2^{\u,s,(i)}(x^*)\bar q(z,s,\u;t,x^*))\partial_k \hat V_iG(t^*-t,x^*-x)dx dt\\
+ \int_{\Reals^{N-1}}f_3(x)\bar q(z,s,\u;t_j,x)\partial_kG(t^*-t'_j,x^*-x)dx.
\end{multline}
Using the above equation, \eqref{align:first}, and \eqref{align:Holder}, we obtain that for some $\kappa_7\in (0,\infty)$
and all $\u \in \ball^N, z \in \Reals^{N-1}$, $t^* \in I$, $s \in J$ and $x^* \in K$.
\begin{align*}
|\partial_k\bar q(z,s,\u;t^*,x^*)| &\le \kappa_7\int_{t'_j}^{t^*}\int_{\Reals^{N-1}}|\partial_kG(t^*-t,x^*-x)|dx dt\nonumber\\
&\quad + \kappa_7\int_{t'_j}^{t^*}\sum_{i=1}^N \int_{\Reals^{N-1}} |x^*-x|^{\theta}|\partial_k \hat V_iG(t^*-t,x^*-x)|dx dt\nonumber\\
&\quad + \kappa_7\int_{\Reals^{N-1}}|\partial_kG(t^*-t'_j,x^*-x)| dx.
\end{align*}
Finally, using \eqref{align:heatest} and \eqref{align:Gbound} in the above, we obtain for all compact $I,J,K$
\begin{align}\label{align:second}
\sup_{\u \in \ball^{N}}\sup_{z \in \Reals^{N-1}}\sup_{(t, s, x) \in I \times J \times K} |\partial_k\bar q(z,s,\u;t,x)| < \infty, \ \ 1 \le k \le N-1.
\end{align}
To deduce the existence of the second derivatives $(\partial_l \partial_k\bar q(z,s,\u;t,\cdot))_{1\le l,k \le N-1}$, note that using integration by parts, we can rewrite \eqref{align:derivativeform} as 
\begin{multline*}
\partial_k\bar q(z,s,\u;t^*,x^*) =-\int_{t'_j}^{t^*}\int_{\Reals^{N-1}}\partial_k\left(f^{\u,s}_1(\cdot)\bar q(z,s,\u;t,\cdot)\right)(x)G(t^*-t,x^*-x)dx dt\\
 - \int_{t'_j}^{t^*}\int_{\Reals^{N-1}}\sum_{i=1}^N\partial_k\left(f_2^{\u,s,(i)}(\cdot)\bar q(z,s,\u;t,\cdot)\right)(x) \hat V_iG(t^*-t,x^*-x)dx dt\\
- \int_{\Reals^{N-1}}\partial_k\left(f_3(\cdot)\bar q(z,s,\u;t_j,\cdot)\right)(x)G(t^*-t'_j,x^*-x)dx.
\end{multline*}
Now, along the same line of argument used to prove the existence of $(\partial_k \bar q(z,s,\u;t,\cdot))_{1\le k \le N-1}$, we prove the H\"older continuity of the derivatives $(\partial_k\bar q(z,s,\u;t,\cdot))_{1\le k \le N-1}$ and use it to deduce that $(\partial_l \partial_k\bar q(z,s,\u;\cdot,\cdot))_{1\le l,k \le N-1}$ exist, are continuous, and satisfy 
\begin{multline}\label{align:doubleform}
\partial_l\partial_k\bar q(z,s,\u;t^*,x^*) =-\int_{t'_j}^{t^*}\int_{\Reals^{N-1}}\partial_k\left(f^{\u,s}_1(\cdot)\bar q(z,s,\u;t,\cdot)\right)(x)\partial_lG(t^*-t,x^*-x)dx dt\\
 - \int_{t'_j}^{t^*}\int_{\Reals^{N-1}}\sum_{i=1}^N\left(\partial_k\left(f_2^{\u,s,(i)}(\cdot)\bar q(z,s,\u;t,\cdot)\right)(x)\right.\\
\hspace{5cm} \left. - \partial_k\left(f_2^{\u,s,(i)}(\cdot)\bar q(z,s,\u;t,\cdot)\right)(x^*)\right) \partial_l\hat V_iG(t^*-t,x^*-x)dx dt\\
- \int_{\Reals^{N-1}}\partial_k\left(f_3(\cdot)\bar q(z,s,\u;t_j,\cdot)\right)(x)\partial_lG(t^*-t'_j,x^*-x)dx,
\end{multline}
for $1\leq l,k\leq N-1.$
 Next, using \eqref{align:first}, \eqref{align:second}, and the H\"older estimates for $\partial_lG$ and $\partial_l\hat V_iG$ in \eqref{align:doubleform}, we conclude the following for all compact $I,J,K$
\begin{align}\label{align:third}
\sup_{\u \in \ball^{N}}\sup_{z \in \Reals^{N-1}}\sup_{(t, s, x) \in I \times J \times K} |\partial_l\partial_k\bar q(z,s,\u;t,x)| < \infty, \ \ 1 \le l, k \le N-1.
\end{align}
Finally, for the existence and regularity of the time derivative $\partial_t \bar q$, we rewrite \eqref{align:maineq} for $x^* \in K$ as
\begin{align*}
\bar q(z,s,\u;t^*,x^*) &=\int_{t'_j}^{t^*}\int_{\Reals^{N-1}}f^{\u,s}_1(x)\bar q(z,s,\u;t,x)G(t^*-t,x^*-x)dx dt\nonumber\\
&\quad - \int_{t'_j}^{t^*}\int_{\Reals^{N-1}}\sum_{i=1}^N\hat V_i\left(f_2^{\u,s,(i)}(\cdot)\bar q(z,s,\u;t,\cdot)\right)(x)G(t^*-t,x^*-x)dx dt\nonumber\\
&\quad + \int_{\Reals^{N-1}}f_3(x)\bar q(z,s,\u;t_j,x)G(t^*-t'_j,x^*-x)dx.
\end{align*}
Using this representation and the H\"older continuity of the derivatives $(\partial_k \bar q(z,s,\u;t,\cdot))_{1\le k \le N-1}$, we can argue along the same lines as before to derive the existence and continuity of $\partial_t \bar q$ and the following representation:
\begin{multline}\label{align:timeform}
\partial_t\bar q(z,s,\u;t^*,x^*) =\int_{t'_j}^{t^*}\int_{\Reals^{N-1}}\left(f^{\u,s}_1(x)\bar q(z,s,\u;t,x)-f^{\u,s}_1(x^*)\bar q(z,s,\u;t,x^*)\right)\partial_tG(t^*-t,x^*-x)dx dt\\
 - \int_{t'_j}^{t^*}\int_{\Reals^{N-1}}\sum_{i=1}^N\left(\hat V_i\left(f_2^{\u,s,(i)}(\cdot)\bar q(z,s,\u;t,\cdot)\right)(x)-\hat V_i\left(f_2^{\u,s,(i)}(\cdot)\bar q(z,s,\u;t,\cdot)\right)(x^*)\right)\partial_tG(t^*-t,x^*-x)dx dt\\
+ \int_{\Reals^{N-1}}\left(f_3(x)\bar q(z,s,\u;t'_j,x) - f_3(x^*)\bar q(z,s,\u;t'_j,x^*)\right)\partial_tG(t^*-t'_j,x^*-x)dx\\
+ f^{\u,s}_1(x^*)\bar q(z,s,\u;t^*,x^*) - \sum_{i=1}^N\hat V_i\left(f_2^{\u,s,(i)}(\cdot)\bar q(z,s,\u;t^*,\cdot)\right)(x^*).
\end{multline}
Once more, using \eqref{align:first}, \eqref{align:second} and the H\"older estimate for $\partial_tG$ in \eqref{align:timeform}, we conclude for all compact $I,J,K$
\begin{align}\label{align:fourth}
\sup_{\u \in \ball^{N}}\sup_{z \in \Reals^{N-1}}\sup_{(t, s, x) \in I \times J \times K} |\partial_t\bar q(z,s,\u;t,x)| < \infty.
\end{align}
This finishes the proof of \eqref{eq:mainestim}  and therefore of the lemma.
\end{proof}
\subsection{Proofs of Lemma \ref{entlemma} and Lemma \ref{Ilemma}}\label{tech}
To avoid cumbersome notation,  in this section we will write $\mathcal{L}$ for the operator $\mathcal{L}^N$ introduced in
\eqref{eq:defln} and consider for a function $\eta:[0,T]\to \RR^N$ the `controlled generator' $\L^{\eta}$ defined as
\begin{equation}\label{equation:genc}
(\L^{\eta}f)(s,x) \doteq \L f(x) - N\sum_{i=1}^N \eta_{i+1}(s) (V_i f)(x), \; f: \RR^N \to \RR, \; (s,x)\in [0,T]\times \RR^N,
\end{equation}
where $\eta=(\eta_i)_{i=1}^N$ and $\eta_{N+1}=\eta_1$.

Let for $t\ge 0$, $\bar{p}_N(t,x)$ be the density with respect to the measure $\Phi^N(dx)$ of $\bar X^N(t)$, given as
in Lemma \ref{reglemma}.  Recall from \eqref{Hzeroeq} that $\Pi^N(dx) = \bar p_N(0,dx) \Phi^N(dx)$ satisfies the  relative entropy bound for all $N\in \NN$:
\begin{equation}
\label{initest}H_N(0) = \int_{\mathbb{R}^N} \bar{p}_N(0,x)\log(\bar{p}_N(0,x)) \Phi^N(dx) \leq C_0N.
\end{equation}

\begin{proof}[Proof of Lemma \ref{entlemma}] 
Let $\QQ_{\Phi^N}$ denote the probability law of $X^N$ on $C([0,T]:\RR^N)$. We can disintegrate $\QQ_{\Phi^N}$	as
$\QQ_{\Phi^N}(d\om) = \Phi^N(d\om_0)\QQ_{\om_0}(d\om)$.
Denote the probability law of $\bar X^N$ on $C([0,T]:\RR^N)$ by $\bar \QQ_{\Pi^N}$. Then $\bar \QQ_{\Pi^N}$ can be disintegrated as
$\bar \QQ_{\Pi^N}(d\om) = \Pi^N(d\om_0) \bar \QQ_{\om_0}(d\om)$.
By chain rule of relative entropies
$$R(\bar \QQ_{\Pi^N} \| \QQ_{\Phi^N}) = R(\Pi^N\|\Phi^N) + \int_{\RR^N} R(\bar \QQ_x \| \QQ_x) \Pi^N(dx).$$
By \eqref{initest}
$R(\Pi^N\|\Phi^N) \le C_0N$. Also, by Girsanov theorem
\begin{align*}
\int_{\RR^N} R(\bar \QQ_x \| \QQ_x) \Pi^N(dx) &= R(\bar \QQ_{\boldsymbol{\cdot}}\otimes \Pi^N \| \QQ_{\boldsymbol{\cdot}}\otimes \Pi^N)\\
&	\le R(\bar \PP_{\Pi^N}\| \PP_{\Pi^N}) \le \bar \EE_{\Pi^N}\left(\frac{1}{2}\sum_{i=1}^N\int_0^T \psi_i^2(s)ds\right) \le \frac{1}{2}C_0N.
\end{align*}
Finally, denoting by $\om(t)$ the coordinate projection on $C([0,T]:\RR^N)$ at time  $t$,
$$H_N(t) = R(\bar{\mathbb{Q}}_{\Pi^N}(t) || \Phi^N) = R(\bar{\mathbb{Q}}_{\Pi^N} \circ (\om(t))^{-1} || \QQ_{\Phi^N} \circ (\om(t))^{-1})
\le R(\bar{\mathbb{Q}}_{\Pi^N}  || \QQ_{\Phi^N} ) \le \frac{3}{2}C_0N$$
 which completes the proof of Lemma \ref{entlemma}.
 \end{proof}

The following lemma shows that if, for a fixed $N$, the initial density of the controlled process $\bar p_N(0, \cdot)$
is bounded, then the densities  $\bar{p}_N(t,\cdot)$ are uniformly bounded in $L^2(\Phi^N)$, over $[0,T]$.
\begin{lemma}\label{ltwo}
Fix $N \in \NN$ and suppose there is a  $M\in (0,\infty)$ such that $\bar{p}_N(0,x) \le M$ for all $x \in \mathbb{R}^N$. Then there exists a  $C^*= C^*(N,M) \in (0,\infty)$  such that for all $t \in [0,T]$,
$
\int_{\mathbb{R}^N}\bar{p}_N^2(t,x) \Phi^N(dx) \le C^*. 
$
\end{lemma}
\begin{proof}
	Recall the system $\sys_{\Pi^N} \doteq (\bar \clv, \bar \clf, \{\bar \clf_t\}, \bar \PP, \bar X^N(0), \bfB^N)$  from Section \ref{subsec:varrep} on which the process $\{\bar X^N(t)\}$ is given. Also recall that we denote the measure $\bar \PP$
	as $\bar \PP_{\Pi^N}$ to emphasize its dependence on the initial measure $\Pi^N$.
	Define a probability measure $ \PP_{\Pi^N}$ on $(\bar \clv, \bar \clf)$ through the relation.
	\begin{equation*}
	\frac{d\mathbb{P}_{\Pi^N}}{d\bar{\mathbb{P}}_{\Pi^N}} = \exp\left\{-\sum_{i=1}^N \int_0^T \psi_i^N(s) dB_i(s) + \frac{1}{2}\sum_{i=1}^N \int_0^T|\psi_i^N(s)|^2ds\right\}.
	\end{equation*}	
	Recalling that $\psi^N \in \cla^N_s(\sys_{\Pi^N})$, we have for some $\kappa_N\in (0,\infty)$
	and all non-negative measurable $g:\RR^N\to \RR$
	$$\bar{\mathbb{E}}_{\Pi^N}(g(\bar X^N(t))) = \mathbb{E}_{\Pi^N}\left(g(\bar X^N(t)) \frac{d\bar{\mathbb{P}}_{\Pi^N}}{d\mathbb{P}_{\Pi^N}}\right) \le \kappa_N \left(\mathbb{E}_{\Pi^N}(g(\bar X^N(t)))^2\right)^{1/2}.$$
	Also, since $\QQ_{\Phi^N}$ is the probability measure on $C([0,T]:\RR^N)$ induced by $X^N$, we have by Girsanov's theorem
	\begin{align*}\mathbb{E}_{\Pi^N}(g(\bar X^N(t)))^2 &= \int_{C([0,T]:\RR^N)}g^2(\om(t)) \bar p_N(0, \om(0))\QQ_{\Phi^N}(d\om)\\
	&\le M \int_{\RR^N} g^2(\om(t)) \QQ_{\Phi^N}(d\om) = M \int_{\RR^N} g^2(x) \Phi^N(dx),\end{align*}
	where the last equality is from the stationarity of $\QQ_{\Phi^N}$. 
	Thus for all non-negative $g$
	$$\int_{\RR^N} g(x) \bar p_N(t,x) \Phi^N(dx) \le M^{1/2} \kappa_N \left(\int_{\RR^N} g^2(x) \Phi^N(dx)\right)^{1/2}.$$
	Taking $g(x) = \bar p_N(t,x) \wedge L$ for fixed $L<\infty$, we have
	\begin{align*}\int_{\RR^N}  (\bar p_N(t,x)\wedge L)^2 \Phi^N(dx)
	&\le \int_{\RR^N}  (\bar p_N(t,x)\wedge L)\bar p_N(t,x) \Phi^N(dx) \\
	&\le M^{1/2} \kappa_N \left(\int_{\RR^N}  (\bar p_N(t,x)\wedge L)^2 \Phi^N(dx)\right)^{1/2}.
\end{align*}
	The result now follows on dividing by $\left(\int_{\RR^N}  (\bar p_N(t,x)\wedge L)^2 \Phi^N(dx)\right)^{1/2}$ and then sending $L\to \infty$.

\end{proof}

\begin{proof}[Proof of Lemma \ref{Ilemma}]
By Lemma \ref{reglemma}, we know that $\bp_N(\cdot,\cdot)$ is $C^{1,2}$ for $\{t \in(t_j, t_{j+1}): 1\le j \le K\}$ and $x \in \Reals^N.$ Therefore, we will assume without loss of generality that $\bar{p}_N(t,x)$ is $C^{1,2}$ for $t \in (0,T)$ and $x \in \Reals^N$. In the general case, the same proof can be employed by applying It\^{o}'s formula on the time intervals $(t_j, t_{j+1})$ to give us the desired result.

To avoid cumbersome notation, we will suppress the $N$ dependence in the notation and write the functional $I_N$ as $I$. The first step will be to show that
\begin{align}\label{align:Ifin}
\int_0^T I(\bar{p}_N(s, \cdot))ds<\infty.
\end{align} 
 First assume that there exists $M>0$ such that $\bar{p}_N(0,x)) \le M$ for all $x \in \mathbb{R}^N$.
We begin by observing that one can find an increasing sequence of smooth functions $(\eta_n)_{n \ge 1}$ on $\Reals^N$ with compact support such that, (i) $\eta_n(x) = 1$ when $|x| \le n$, (ii) $0<\eta(x) \le 1$ when $|x| < n+1$ and $\eta_n(x) = 0$ when $|x| \ge n+1$, (iii) there exists $\gamma(\eta) \in (0,\infty)$  such that for all  $n,N \in \mathbb{N}$, $|\partial_i \eta_n(x)| \le \gamma(\eta)$, $|\partial_i \partial_j\eta_n(x)| \le \gamma(\eta)$ for all $x \in \mathbb{R}^N$ and all $1 \le i,j \le N$.
 In what follows we write $\bar X^N_t$ as $\bar X_t$.
Define the `localized entropy'
$$
H_{n,N}(t) \coloneqq \int_{\Reals^N} \bar{p}_N(t,x) \log (\bar{p}_N(t,x)) \eta_n(x)\Phi^N(dx) = \bar{\mathbb{E}}_{\Pi^N}\left(\log(\bar{p}_N(t,\bX_t))\eta_n(\bX_t)\right), \ \ 0 \le t \le T.
$$
For $\epsilon>0$, define $\bar{p}_N^{(\epsilon)}(t,x) = \bar{p}_N(t, x) + \epsilon$ and
$$
H^{(\epsilon)}_{n,N}(t) \coloneqq \int_{\Reals^N} \bar{p}_N(t,x) \log (\bar{p}_N^{(\epsilon)}(s,x)) \eta_n(x)\Phi^N(dx) = \bar{\mathbb{E}}_{\Pi^N}\left(\log(\bar{p}_N^{(\epsilon)}(t,\bX_t))\eta_n(\bX_t)\right), \ \ 0 \le t \le T.
$$
By the continuity of $\bar{p}_N$ and the  monotone convergence theorem, for each $t$,
$
\lim_{\epsilon \rightarrow 0}H^{(\epsilon)}_{n,N}(t) = H_{n,N}(t).
$
As $\bar{p}_N$ is $C^{1,2}$ and $\eta_n$ is smooth, we can apply It\^{o}'s formula to $\log(\bar{p}_N^{(\epsilon)}(t,\bX_t))\eta_n(\bX_t)$ to obtain for $0 < t_1 < t_2 <T$,
\begin{multline*}
H^{(\epsilon)}_{n,N}(t_2) - H^{(\epsilon)}_{n,N}(t_1)\\
= \bar{\mathbb{E}}_{\Pi^N}\left(\int_{t_1}^{t_2}(\partial_s + \L^{\psi}) \left(\log \bar{p}_N^{(\epsilon)}(\cdot, \cdot) \eta_n(\cdot)\right)(s, \bX_s)ds + \int_{t_1}^{t_2}\sum_{i=1}^N\eta_n(\bX_s) \frac{V_i\bar{p}_N^{(\epsilon)}(s, \bX_s)}{\bar{p}_N^{(\epsilon)}(s, \bX_s)}dB_i(s)\right),
\end{multline*}
where $\left(\log \bar{p}_N^{(\epsilon)}(\cdot, \cdot) \eta_n(\cdot)\right)(s, x) = \log \bar{p}_N^{(\epsilon)}(s,x) \eta_n(x)$. As $\bar{p}_N^{(\epsilon)}(t,x) \ge \epsilon$ for all $(t,x)$, the local martingale part in the above equation is, in fact, a  martingale and we deduce
\begin{align}\label{align:enteq}
H^{(\epsilon)}_{n,N}(t_2) - H^{(\epsilon)}_{n,N}(t_1) &= \bar{\mathbb{E}}_{\Pi^N}\left(\int_{t_1}^{t_2}(\partial_s + \L^{\psi})\left(\log \bar{p}_N^{(\epsilon)}(\cdot, \cdot) \eta_n(\cdot)\right)(s, \bX_s)ds\right).
\end{align}
By the entropy bound obtained in Lemma \ref{entlemma}, there is a  $\gamma(\bar p)\in (0,\infty)$  such that
for all $t \in [0,T]$,  $N\in \mathbb{N}$, and $M\in (0,\infty)$ (recall that $M$ is the bound on the initial density)
\begin{align}\label{align:absen}
\int_{\Reals^N}\bar{p}_N(t,x) |\log (\bar{p}_N(t,x))| \Phi^N(dx) \le \gamma(\bar p) N.
\end{align} 
In fact $\gamma(\bar p)$ can be taken to be $3C_0/2$ where $C_0$ is as in \eqref{Hzeroeq}.
 Therefore, by the dominated convergence theorem,
\begin{align}\label{hlim}
\lim_{n \rightarrow \infty}\lim_{\epsilon \rightarrow 0}(H_{n,N}^{(\epsilon)}(t_2)-H_{n,N}^{(\epsilon)}(t_1))=\lim_{n \rightarrow \infty}(H_{n,N}(t_2)-H_{n,N}(t_1)) = H_N(t_2)-H_N(t_1),
\end{align}
where $H_N$ is as defined in \eqref{initentest}.
Recalling that $\L^{\psi}=\L - N\sum_{i=1}^{N} \psi_{i+1} V_i$, we write
\begin{align}
&\bar{\mathbb{E}}_{\Pi^N}\left(\int_{t_1}^{t_2}(\partial_s + \L^{\psi})\left(\log \bar{p}_N^{(\epsilon)}(\cdot, \cdot) \eta_n(\cdot)\right)(s, \bX_s)ds\right)\nonumber\\
&\quad=\bar{\mathbb{E}}_{\Pi^N}\left(\int_{t_1}^{t_2}\partial_s  \left(\log \bar{p}_N^{(\epsilon)}(\cdot, \cdot) \eta_n(\cdot)\right)(s, \bX_s)\right) + \bar{\mathbb{E}}_{\Pi^N}\left(\int_{t_1}^{t_2}\L\left(\log \bar{p}_N^{(\epsilon)}(\cdot, \cdot) \eta_n(\cdot)\right)(s, \bX_s)ds\right)\nonumber\\
&\quad - \bar{\mathbb{E}}_{\Pi^N}\left(N\int_{t_1}^{t_2}\sum_{i=1}^N \psi_{i+1}(s) V_i\left(\log \bar{p}_N^{(\epsilon)}(\cdot, \cdot) \eta_n(\cdot)\right)(s, \bX_s)ds\right).\label{eq:eq435}
\end{align}      
For the middle term on the right side we have
\begin{align*}
&\bar{\mathbb{E}}_{\Pi^N}\left(\L\left(\log \bar{p}_N^{(\epsilon)}(\cdot, \cdot) \eta_n(\cdot)\right)(s, \bX_s)\right)\\
&= \bar{\mathbb{E}}_{\Pi^N}\left(\eta_n(\bX_s)\frac{\L \bar{p}_N^{(\epsilon)} (s, \bX_s)}{\bar{p}_N^{(\epsilon)}(s,\bX_s)} + (\log \bar{p}_N^{(\epsilon)}(s, \bX_s)) \L\eta_n(\bX_s) + N^2 \sum_{i=1}^N V_i\eta_n(\bX_s) \frac{V_i \bar{p}_N^{(\epsilon)}(s,\bX_s)}{\bar{p}_N^{(\epsilon)}(s,\bX_s)}\right)\\
 &\quad\quad- \bar{\mathbb{E}}_{\Pi^N}\left(\frac{N^2}{2}\eta_n(\bX_s)\sum_{i=1}^N \frac{(V_i \bar{p}_N^{(\epsilon)})^2(s,\bX_s)}{(\bar{p}_N^{(\epsilon)})^2(s,\bX_s)}\right)
\end{align*}
and for the third term in \eqref{eq:eq435} we have
\begin{align*}
\bar{\mathbb{E}}_{\Pi^N}\left(N\sum_{i=1}^N \psi_{i+1}(s) V_i\left(\log \bar{p}_N^{(\epsilon)}(\cdot, \cdot) \eta_n(\cdot)\right)(s, \bX_s)\right)&=\bar{\mathbb{E}}_{\Pi^N}\left(N\sum_{i=1}^N \psi_{i+1}(s) (V_i\eta_n)(\bX_s)\log \bar{p}_N^{(\epsilon)}(s, \bX_s)\right)\\
&\quad + \bar{\mathbb{E}}_{\Pi^N}\left(N\sum_{i=1}^N \psi_{i+1}(s) \eta_n(\bX_s)\frac{V_i\bar{p}_N^{(\epsilon)}(s,\bX_s)}{\bar{p}_N^{(\epsilon)}(s,\bX_s)}\right).
\end{align*}
From the above expressions, we can write 
$$
\bar{\mathbb{E}}_{\Pi^N}\left(\int_{t_1}^{t_2}(\partial_s + \L^{\psi})\left(\log \bar{p}_N^{(\epsilon)}(\cdot, \cdot) \eta_n(\cdot)\right)(s, \bX_s)ds\right) \doteq T_1^{(\epsilon)}(n) + T_2^{(\epsilon)}(n),
$$
where
\begin{align}
T_1^{(\epsilon)}(n)&=\bar{\mathbb{E}}_{\Pi^N}\left(\int_{t_1}^{t_2}\partial_s  \left(\log \bar{p}_N^{(\epsilon)}(\cdot, \cdot) \eta_n(\cdot)\right)(s, \bX_s)\right)\nonumber\\
&\quad + \int_{t_1}^{t_2}\bar{\mathbb{E}}_{\Pi^N}\left(\eta_n(\bX_s)\frac{\L \bar{p}_N^{(\epsilon)}(s, \bX_s)}{\bar{p}_N^{(\epsilon)}(s,\bX_s)} + (\log \bar{p}_N^{(\epsilon)}(s, \bX_s)) \L\eta_n(\bX_s)\right)ds\nonumber\\
 &\quad - \int_{t_1}^{t_2}\bar{\mathbb{E}}_{\Pi^N}\left(N\sum_{i=1}^N \psi_{i+1}(s) (V_i\eta_n)(\bX_s)\log \bar{p}_N^{(\epsilon)}(s, \bX_s)\right)ds \nonumber\\
&\quad + \int_{t_1}^{t_2}\bar{\mathbb{E}}_{\Pi^N}\left(N^2 \sum_{i=1}^N V_i\eta_n(\bX_s) \frac{V_i \bar{p}_N^{(\epsilon)}(s,\bX_s)}{\bar{p}_N^{(\epsilon)}(s,\bX_s)}\right)ds \label{eq:t1epsn}
\end{align}
and
\begin{align}
T_2^{(\epsilon)}(n)&= -\int_{t_1}^{t_2}\bar{\mathbb{E}}_{\Pi^N}\left(\frac{N^2}{2}\eta_n(\bX_s)\sum_{i=1}^N \frac{(V_i \bar{p}_N^{(\epsilon)})^2(s,\bX_s)}{(\bar{p}_N^{(\epsilon)})^2(s,\bX_s)}\right)ds\nonumber\\
 &\quad - \int_{t_1}^{t_2}\bar{\mathbb{E}}_{\Pi^N}\left(N\sum_{i=1}^N \psi_{i+1}(s) \eta_n(\bX_s)\frac{V_i\bar{p}_N^{(\epsilon)}(s,\bX_s)}{\bar{p}_N^{(\epsilon)}(s,\bX_s)}\right)ds.\label{eq:t2epsn}
\end{align}
We will now show that 
\begin{equation}\label{tone}
	\limsup_{n \rightarrow \infty} \lim_{\epsilon \rightarrow 0} T_1^{(\epsilon)}(n) \le 0.
\end{equation}
As $\int_{\Reals^N}\bar{p}_N(t,x)\Phi^N(dx)=1$ for all $t \in [0,T]$, by dominated convergence theorem,
\begin{align}
&\lim_{n \rightarrow \infty} \lim_{\epsilon \rightarrow 0}\bar{\mathbb{E}}_{\Pi^N}\left(\int_{t_1}^{t_2}\partial_s  \left(\log \bar{p}_N^{(\epsilon)}(\cdot, \cdot) \eta_n(\cdot)\right)(s, \bX_s)\right)\nonumber \\
&\quad  = \lim_{n \rightarrow \infty}\lim_{\epsilon \rightarrow 0}\int_{t_1}^{t_2} \bar{\mathbb{E}}_{\Pi^N}\left(\frac{\partial_s \bar{p}_N^{(\epsilon)}(\cdot, \cdot)}{\bar{p}_N^{(\epsilon)}(\cdot,\cdot)} \eta_n(\cdot)(s, \bX_s)\right)\nonumber \\
&\quad = \lim_{n \rightarrow \infty}\lim_{\epsilon \rightarrow 0}\int_{t_1}^{t_2}\int_{\Reals^N} \partial_s \bar{p}_N(s,x)\eta_n(x) \frac{\bar{p}_N(s,x)}{\bar{p}_N^{(\epsilon)}(s,x)}\Phi^N(dx) ds\nonumber\\
&\quad = \lim_{n \rightarrow \infty}\int_{t_1}^{t_2}\int_{\Reals^N} \partial_s \bar{p}_N(s,x)\eta_n(x) \mathbb{I}_{\{\bar{p}_N(s,x) >0\}} \Phi^N(dx) ds\nonumber \\
&\quad = \lim_{n \rightarrow \infty}\int_{t_1}^{t_2}\int_{\Reals^N} \partial_s \bar{p}_N(s,x)\eta_n(x) \Phi^N(dx) ds\nonumber \\
&\quad= \lim_{n \rightarrow \infty} \left(\int_{\Reals^N}  \bar{p}_N(t_2,x)\eta_n(x) \Phi^N(dx) - \int_{\Reals^N}\bar{p}_N(t_1,x)\eta_n(x) \Phi^N(dx)\right)\nonumber \\
&\quad =\int_{\Reals^N}  \bar{p}_N(t_2,x) \Phi^N(dx) - \int_{\Reals^N}\bar{p}_N(t_1,x) \Phi^N(dx) = 0.\label{one}
\end{align}
In the fourth equality above, we have used the fact that if $\bar{p}_N(s,x)=0$ for some $(s,x)$, global non-negativity of $\bar{p}_N$ will imply that $\partial_s\bar{p}_N(s,x) =0$. Again by dominated convergence theorem,
\begin{align}
\lim_{\epsilon \rightarrow 0}\int_{t_1}^{t_2}\bar{\mathbb{E}}_{\Pi^N}\left(\eta_n(\bX_s)\frac{\L \bar{p}_N^{(\epsilon)} (s, \bX_s)}{\bar{p}_N^{(\epsilon)}(s,\bX_s)}\right)ds &= \lim_{\epsilon \rightarrow 0}\int_{t_1}^{t_2}\int_{\Reals^N}\eta_n(x) \L \bar{p}_N^{(\epsilon)}(s,x)\frac{\bar{p}_N(s,x)}{\bar{p}_N^{(\epsilon)}(s,x)}\Phi^N(dx) ds\nonumber\\
 &=\int_{t_1}^{t_2}\int_{\Reals^N} \eta_n(x) \L\bar{p}_N(s,x)\mathbb{I}_{\{\bar{p}_N(s,x) >0\}}\Phi^N(dx) ds\nonumber\\
 &\le \int_{t_1}^{t_2}\int_{\Reals^N} \eta_n(x) \L\bar{p}_N(s,x)\Phi^N(dx) \nonumber \\
&= \int_{t_1}^{t_2}\int_{\Reals^N}\L \eta_n(x) \bar{p}_N(s,x)\Phi^N(dx) ds.\label{eq:eq501}
\end{align}
In obtaining the inequality above, we have used the fact that if $\bar{p}_N(s,x)=0$ for some $(s,x)$, then from non-negativity of $\bar{p}_N$ it follows that $V_i\bar{p}_N(s,x) =0$ for each $i$, and, as such an $(s,x)$ is a local minimum, $V_i^2\bar{p}_N(s,x) \ge 0$ for each $i$ and therefore
\begin{align*}
&\int_{t_1}^{t_2}\int_{\Reals^N} \eta_n(x) \L\bar{p}_N(s,x)\mathbb{I}_{\{\bar{p}_N(s,x) = 0\}}\Phi^N(dx) ds\\
&\quad = \int_{t_1}^{t_2}\int_{\Reals^N} \eta_n(x) \frac{N^2}{2}\sum_{i=1}^N (V_i^2\bar{p}_N)(s,x)\mathbb{I}_{\{\bar{p}_N(s,x) = 0\}}\Phi^N(dx) ds \ge 0.
\end{align*}
The last equality in \eqref{eq:eq501} follows from the fact that $\L$ is symmetric with respect to the measure $\Phi^N(dx)$.

By part (iii) of the set of conditions satisfied by $(\eta_n)_{n \ge 1}$, there is a $\gamma_1(\eta)\in (0,\infty)$  such that for all  $n,N \in \NN$ and $x \in \Reals^N$ 
$|\L \eta_n(x)| \le \gamma_1(\eta)N^2\sum_{i=1}^N |\phi'(x_i)|$. Moreover, by \eqref{sigmaass} and the entropy bound \eqref{align:absen}, there is a $\kappa_1 \in (0,\infty)$ such that for all $n,N \in \NN$ and $0<t_1<t_2<T$
\begin{multline*}
\int_{t_1}^{t_2}\int_{\Reals^N}\sum_{i=1}^N |\phi'(x_i)| \bar{p}_N(s,x)\Phi^N(dx) ds \le \int_{t_1}^{t_2}\log\left(\int_{\Reals^N} e^{\sum_{i=1}^N |\phi'(x_i)|} \Phi^N(dx)\right)ds + \gamma(\bar p) N \le \kappa_1 N.
\end{multline*} 
Therefore, as $\L \eta_n$ converges to zero pointwise as $n \rightarrow \infty$, by dominated convergence theorem
and \eqref{eq:eq501},
\begin{align}\label{two}
\limsup_{n \rightarrow \infty}\lim_{\epsilon \rightarrow 0}\int_{t_1}^{t_2}\bar{\mathbb{E}}_{\Pi^N}\left(\eta_n(\bX_s)\frac{\L \bar{p}_N^{(\epsilon)} (s, \bX_s)}{\bar{p}_N^{(\epsilon)}(s,\bX_s)}\right)ds \le \lim_{n \rightarrow \infty}\int_{t_1}^{t_2}\int_{\Reals^N}\L \eta_n(x) \bar{p}_N(s,x)\Phi^N(dx) ds = 0.
\end{align}
Next, we consider the third term in \eqref{eq:t1epsn}. By the  monotone convergence theorem,
\begin{align*}
\lim_{\epsilon \rightarrow 0}\int_{t_1}^{t_2}\bar{\mathbb{E}}_{\Pi^N}\left(\log \bar{p}_N^{(\epsilon)}(s, \bX_s) \L\eta_n(\bX_s)\right)ds &= \lim_{\epsilon \rightarrow 0}\int_{t_1}^{t_2}\int_{\mathbb{R}^N}\L\eta_n(x)\bar{p}_N(s, x)\log \bar{p}_N^{(\epsilon)}(s, x) \Phi^N(dx)ds\\
&= \int_{t_1}^{t_2}\int_{\mathbb{R}^N}\L\eta_n(x)\bar{p}_N(s, x)\log \bar{p}_N(s, x) \Phi^N(dx)ds.
\end{align*}
Also
$$
|\L \eta_n(x)|\bar{p}_N(s, x)|\log \bar{p}_N(s, x)| \le  \gamma_1(\eta) N^2\left(\sum_{i=1}^N |\phi'(x_i)|\right)\bar{p}_N(s, x)|\log \bar{p}_N(s, x)|.
$$
Since $\bar{p}_N(0,\cdot) \le M$, we have by Lemma \ref{ltwo} that for each $N\in \mathbb{N}$, 
$$
\sup_{t\in [0,T]}\int_{\mathbb{R}^N}\bar{p}_N^2(t,x) \Phi^N(dx) < \infty.
$$
Thus, applying H\"{o}lder's inequality with $p>2$ and $p^{-1} + q^{-1} = 1$,  along with \eqref{sigmaass}, yields
for each $N\in \NN$
\begin{multline*}
\int_{t_1}^{t_2}\int_{\mathbb{R}^N} \sum_{i=1}^N |\phi'(x_i)|\bar{p}_N(s, x)|\log \bar{p}_N(s, x)| \Phi^N(dx) ds\\
\le \int_{t_1}^{t_2} \left(\int_{\mathbb{R}^N} \left(\sum_{i=1}^N |\phi'(x_i)|\right)^p\Phi^N(dx)\right)^{1/p}\left(\int_{\mathbb{R}^N} \left(\bar{p}_N(s, x)|\log \bar{p}_N(s, x)|\right)^q\Phi^N(dx)\right)^{1/q}ds  < \infty.
\end{multline*}
Therefore, by dominated convergence theorem,
\begin{align}
 &\lim_{n \rightarrow \infty}\lim_{\epsilon \rightarrow 0}\int_{t_1}^{t_2}\bar{\mathbb{E}}_{\Pi^N}\left(\log \bar{p}_N^{(\epsilon)}(s, \bX_s) \L\eta_n(\bX_s)\right)ds\nonumber\\
 &\quad =\lim_{n \rightarrow \infty}\int_{t_1}^{t_2}\int_{\mathbb{R}^N}\L\eta_n(x)\bar{p}_N(s, x)\log \bar{p}_N(s, x) \Phi^N(dx)ds=0.\label{three}
\end{align}
Consider now the fourth term in the definition of $T_1^{(\epsilon)}(n)$.
From \eqref{controlas} and the Cauchy-Schwarz inequality, we conclude that for each $N \in \NN$ there is a $c_1(N)\in (0,\infty)$ such that for all $n\in \NN$
\begin{align*}
\sum_{i=1}^N N|\psi_{i+1}(s) (V_i\eta_n)(\bX_s)\log \bar{p}_N^{(\epsilon)}(s, \bX_s)| \le c_1(N)\left(\sum_{i=1}^N(V_i\eta_n)^2(\bX_s)\right)^{1/2}|\log \bar{p}_N^{(\epsilon)}(s, \bX_s)|.
\end{align*}
Again from part (iii) of the set of conditions satisfied by $(\eta_n)_{n \ge 1}$,  there is a  $\gamma_2(\eta)\in (0,\infty)$  such that for all $n,N \in \NN$ and $x\in \RR^N$
$
\sum_{i=1}^N(V_i\eta_n)^2(x) \le N \gamma_2(\eta).
$
Therefore, by \eqref{align:absen} and dominated convergence theorem, we conclude that
\begin{align}
&\lim_{n \rightarrow \infty}\lim_{\epsilon \rightarrow 0}\left|\int_{t_1}^{t_2}\bar{\mathbb{E}}_{\Pi^N}\left(N\sum_{i=1}^N \psi_{i+1}(s) (V_i\eta_n)(\bX_s)\log \bar{p}_N^{(\epsilon)}(s, \bX_s)\right)ds\right| \nonumber\\
&\quad \le \lim_{n \rightarrow \infty}\lim_{\epsilon \rightarrow 0}
\int_{t_1}^{t_2}\int_{\mathbb{R}^N} c_1(N)\left(\sum_{i=1}^N(V_i\eta_n)^2(x)\right)^{1/2} |\log \bar{p}_N^{(\epsilon)}(s, x)|\bar{p}_N(s, x)\Phi^N(dx)ds\nonumber\\
&\quad =\lim_{n \rightarrow \infty}
\int_{t_1}^{t_2}\int_{\mathbb{R}^N} c_1(N)\left(\sum_{i=1}^N(V_i\eta_n)^2(x)\right)^{1/2} |\log \bar{p}_N(s, x)|\bar{p}_N(s, x)\Phi^N(dx)ds
 =0.
\label{four}
\end{align}
Finally, for the last term in \eqref{eq:t1epsn} observe that, from dominated convergence theorem, the form of the operator $\L$, and \eqref{two},
\begin{align}
&\lim_{n \rightarrow \infty}\lim_{\epsilon \rightarrow 0}\int_{t_1}^{t_2}\bar{\mathbb{E}}_{\Pi^N}\left(N^2 \sum_{i=1}^N V_i\eta_n(\bX_s) \frac{V_i \bar{p}_N^{(\epsilon)}(s,\bX_s)}{\bar{p}_N^{(\epsilon)}(s,\bX_s)}\right)ds\nonumber\\
&\quad= \lim_{n \rightarrow \infty}\lim_{\epsilon \rightarrow 0}\int_{t_1}^{t_2}N^2 \sum_{i=1}^N V_i\eta_n(x) V_i \bar{p}_N^{(\epsilon)}(s,x) \frac{\bar{p}_N(s,x)}{\bar{p}_N^{(\epsilon)}(s,x)} \Phi^N(dx)ds\nonumber\\
&\quad= \lim_{n \rightarrow \infty}\int_{t_1}^{t_2}N^2 \sum_{i=1}^N V_i\eta_n(x) V_i \bar{p}_N(s,x) \mathbb{I}_{\{\bar{p}_N(s,x) >0\}}\Phi^N(dx)ds\nonumber\\
&\quad= \lim_{n \rightarrow \infty}\int_{t_1}^{t_2}N^2 \sum_{i=1}^N V_i\eta_n(x) V_i \bar{p}_N(s,x)\Phi^N(dx)ds\nonumber\\
&\quad= \lim_{n \rightarrow \infty}-2\int_{t_1}^{t_2} \L\eta_n(x) \bar{p}_N(s,x) \Phi^N(dx)ds = 0.
\label{five}
\end{align}
In the third equality above, we have once more used the fact that if $\bar{p}_N(s,x)=0$ for some $(s,x)$, then $V_i\bar{p}_N(s,x) =0$ for each $i$.

From \eqref{one}, \eqref{two}, \eqref{three}, \eqref{four} and \eqref{five}, we conclude that \eqref{tone} is satisfied.
We will now obtain an upper bound on $T_2^{(\epsilon)}(n)$ in terms of 
 $I_n^{(\epsilon)}$ defined as 
\begin{align}\label{trunk}
I_n^{(\epsilon)} = \int_{t_1}^{t_2}\int_{\mathbb{R}^N}\eta_n(x)\sum_{i=1}^N \frac{(V_i \bar{p}_N)^2(s,x)}{(\bar{p}_N^{\epsilon})^2(s, x)}\bar{p}_N(s,x)\Phi^N(dx)ds.
\end{align}
Observe that
\begin{align}\label{twoI}
\int_{t_1}^{t_2}\bar{\mathbb{E}}_{\Pi^N}\left(\frac{N^2}{2}\eta_n(\bX_s)\sum_{i=1}^N \frac{(V_i \bar{p}_N^{\epsilon})^2(s,\bX_s)}{(\bar{p}_N^{\epsilon})^2(s,\bX_s)}\right)ds=\frac{N^2}{2}I_n^{(\epsilon)}.
\end{align}
Using the Cauchy-Schwarz inequality,
\begin{multline*}
\left\lvert\int_{t_1}^{t_2}\bar{\mathbb{E}}_{\Pi^N}\left(N\sum_{i=1}^N \psi_{i+1}(s) \eta_n(\bX_s)\frac{V_i\bar{p}_N^{\epsilon}(s,\bX_s)}{\bar{p}_N^{\epsilon}(s,\bX_s)}\right)ds\right\rvert\\
\le N\int_{t_1}^{t_2}\bar{\mathbb{E}}_{\Pi^N}\left[\sqrt{\eta_n(\bX_s)}\left(\sum_{i=1}^N |\psi_{i+1}(s)|^2\right)^{1/2} \left(\sum_{i=1}^N\eta_n(\bX_s)\frac{(V_i\bar{p}_N^{\epsilon})^2(s,\bX_s)}{(\bar{p}_N^{\epsilon})^2(s,\bX_s)}\right)^{1/2}\right]ds.
\end{multline*}
By \eqref{controlas} and parts (i) and (ii) of conditions on $(\eta_n)_{n \ge 1}$ applied to the right-hand side above,
for each $N\in \NN$
there is a $c_2(N) \in (0,\infty)$ such that for all $n\in \NN$, $0< t_1<t_2 < T$ and $M \in (0, \infty)$ (the bound on the initial density)
\begin{align}
&\left\lvert\int_{t_1}^{t_2}\bar{\mathbb{E}}_{\Pi^N}\left(N\sum_{i=1}^N \psi_{i+1}(s) \eta_n(\bX_s)\frac{V_i\bar{p}_N^{\epsilon}(s,\bX_s)}{\bar{p}_N^{\epsilon}(s,\bX_s)}\right)ds\right\rvert\nonumber\\
&\le c_2(N) \int_{t_1}^{t_2}\bar{\mathbb{E}}_{\Pi^N} \left(\sum_{i=1}^N\eta_n(\bX_s)\frac{(V_i\bar{p}_N^{\epsilon})^2(s,\bX_s)}{(\bar{p}_N^{\epsilon})^2(s,\bX_s)}\right)^{1/2}
ds\nonumber\\
&\le c_2(N)\sqrt{T}\sqrt{I_n^{(\epsilon)}}, \label{threeI}
\end{align}
From \eqref{twoI} and \eqref{threeI}, we obtain
\begin{align}\label{ttwo}
T_2^{(\epsilon)}(n) \le - \frac{N^2}{2}I_n^{(\epsilon)} + c_2(N)\sqrt{T}\sqrt{I_n^{(\epsilon)}}.
\end{align}
We have from Lemma \ref{entlemma} and \eqref{hlim}, that for some $\kappa_2\in (0,\infty)$
such that for all $0<t_1<t_2<T$, $N\in \NN$ and $M\in (0,\infty)$ 
\begin{align}\label{hexp}
\lim_{n \rightarrow \infty}\lim_{\epsilon \rightarrow 0} (H_{n,N}^{(\epsilon)}(t_2) - H_{n,N}^{(\epsilon)}(t_1)) = H_N(t_2)-H_N(t_1) \ge -\gamma(\bar p) N.
\end{align}
Recalling that 
\begin{equation}\label{eq:hnepst1t2}
H_n^{(\epsilon)}(t_2) - H_n^{(\epsilon)}(t_1)= T_1^{(\epsilon)}(n) + T_2^{(\epsilon)}(n).
\end{equation}
and \eqref{tone}, we have
\begin{equation}\label{lastone}
-\gamma(\bar p)N \le \liminf_{n \rightarrow \infty}\lim_{\epsilon \rightarrow 0} T_1^{(\epsilon)}(n) + \liminf_{n \rightarrow \infty}\liminf_{\epsilon \rightarrow 0} T_2^{(\epsilon)}(n) \le \liminf_{n \rightarrow \infty}\liminf_{\epsilon \rightarrow 0} T_2^{(\epsilon)}(n)
\end{equation}
From the bound on $T_2^{(\epsilon)}(n)$ obtained in \eqref{ttwo}, observe that 
$
\liminf_{n \rightarrow \infty}\liminf_{\epsilon \rightarrow 0}T_2^{(\epsilon)}(n) = -\infty
$
if 
$
\limsup_{n \rightarrow \infty}\lim_{\epsilon \rightarrow 0}I_n^{(\epsilon)} = \infty,
$
which yields a contradiction by virtue of \eqref{lastone}. Thus, we have
$$
\lim_{n \rightarrow \infty}\lim_{\epsilon \rightarrow 0} I_n^{(\epsilon)} = \int_{t_1}^{t_2} \sum_{i=1}^N\int_{\mathbb{R}^N} \frac{(V_i \bar{p}_N)^2(s,x)}{\bar{p}_N(s,x)}\Phi^N(dx) ds< \infty.
$$
Moreover, from \eqref{ttwo} and \eqref{hexp}, it follows that for each $N\in \NN$ there is a  $c_3(N) \in (0,\infty)$
such that for all $0<t_1<t_2<T$ and $M\in (0,\infty)$
$$
\int_{t_1}^{t_2} \sum_{i=1}^N\int_{\mathbb{R}^N} \frac{(V_i \bar{p}_N)^2(s,x)}{\bar{p}_N(s,x)}\Phi^N(dx) ds \le c_3(N).
$$
 Taking limits $t_1 \downarrow 0$ and $t_2 \uparrow T$, we obtain that for each $N\in \NN$
$$
\int_{0}^{T} \sum_{i=1}^N\int_{\mathbb{R}^N} \frac{(V_i \bar{p}_N)^2(s,x)}{\bar{p}_N(s,x)}\Phi^N(dx) ds \le c_3(N),
$$
which proves \eqref{align:Ifin} when the initial density $\bar{p}_N(0,\cdot)$ of the controlled process is bounded.

Now, we address the general case. First, it follows from a variational representation of the function $I$ (see \cite{guopapvar}) that $I$ is lower semicontinuous under the topology of weak convergence of measures, i.e., if a sequence of measures $\mu^{(k)}(dx) = p^{(k)}(x)dx$ converges weakly to $\mu(dx) = p(x)dx$, then
 $$
 I(p(\cdot)) \le \liminf_{k \rightarrow \infty} I(p^{(k)}(\cdot)).
 $$
 Consider the sequence of densities 
$$
\bar{p}^{(k)}_N(0,x) = \frac{\bar{p}_N(0,x) \wedge k}{\int_{\mathbb{R}^N} (\bar{p}_N(0,z) \wedge k) \Phi^N(dz)}, k \ge k_0
$$
where $k_0$ is chosen large enough to ensure that the denominator in the above expression is positive for all $k \ge k_0$. It is straightforward to check that for fixed $N$, the law at time $t$ of the controlled process with initial measure having density $\bar{p}^{(k)}_N(0,\cdot)$ replacing $\bar{p}_N(0,\cdot)$, written as $\bX^{N,(k)}(t)$, converges weakly  as $k \rightarrow \infty$ to that of our original controlled process at time $t$, namely $\bX^N(t)$. Moreover, for fixed $N$,
$$
\lim_{k \rightarrow \infty} \int_{\Reals^N}\bar{p}^{(k)}_N(0,x) |\log (\bar{p}^{(k)}_N(0,x))| \Phi^N(dx) = \int_{\Reals^N}\bar{p}_N(0,x) |\log (\bar{p}_N(0,x))| \Phi^N(dx) \le C_0N
$$
and therefore, exactly as in the proof of Lemma \ref{entlemma}, there is $K \in \NN$ and  $\gamma_1(\bar p) \in (0,\infty)$  such that for all $t \in [0,T]$, $N\in \NN$ and $k\ge K$
\begin{align}\label{align:absenk}
\int_{\Reals^N}\bar{p}^{(k)}_N(t,x) |\log (\bar{p}^{(k)}_N(t,x))| \Phi^N(dx) \le \gamma_1(\bar p)N.
\end{align} 
The proof of the case with bounded initial density given above now gives that there is a $c_4(N) \in (0,\infty)$ such that for each $k \ge K$,
$$
\int_{0}^{T} \sum_{i=1}^N\int_{\mathbb{R}^N} \frac{(V_i \bar{p}^{(k)}_N)^2(s,x)}{\bar{p}^{(k)}_N(s,x)}\Phi^N(dx) ds \le c_4(N).
$$
 Using lower semicontinuity  of the functional $I$ and Fatou's lemma, we obtain
\begin{align*}
\int_{0}^{T} \sum_{i=1}^N\int_{\mathbb{R}^N} \frac{(V_i \bar{p}_N)^2(s,x)}{\bar{p}_N(s,x)}\Phi^N(dx) ds &\le \int_{0}^{T} \liminf_{k \rightarrow \infty}\sum_{i=1}^N\int_{\mathbb{R}^N} \frac{(V_i \bar{p}^{(k)}_N)^2(s,x)}{\bar{p}^{(k)}_N(s,x)}\Phi^N(dx) ds\\
&\le  \liminf_{k \rightarrow \infty}\int_{0}^{T} \sum_{i=1}^N\int_{\mathbb{R}^N} \frac{(V_i \bar{p}^{(k)}_N)^2(s,x)}{\bar{p}^{(k)}_N(s,x)}\Phi^N(dx) ds \le c_4(N),
\end{align*}
which proves \eqref{align:Ifin} for the general case.

Now, we proceed to prove the bound claimed in the lemma, namely \eqref{INest}. As before, we first assume that there exists $M>0$ such that $\bar{p}_N(0,x) \le M$ for all $x \in \mathbb{R}^N$. 
We recall the expression $T_2^{(\epsilon)}(n)$ from \eqref{eq:t2epsn}.
To estimate the last term in the expression for $T_2^{(\epsilon)}(n)$, we use parts (i) and (ii) of the set of conditions satisfied by $(\eta_n)_{n \ge 1}$ and apply Cauchy-Schwarz inequality  to obtain
\begin{align*}
&\left\lvert\int_{t_1}^{t_2}\bar{\mathbb{E}}_{\Pi^N}\left(N\sum_{i=1}^N \psi_{i+1}(s) \eta_n(\bX_s)\frac{V_i\bar{p}_N^{(\epsilon)}(s,\bX_s)}{\bar{p}_N^{(\epsilon)}(s,\bX_s)}\right)ds \right\rvert\\
&\le N\left(\int_{t_1}^{t_2}\bar{\mathbb{E}}_{\Pi^N}\sum_{i=1}^N |\psi_i(s)|^2ds\right)^{1/2} \left(\int_{t_1}^{t_2}\bar{\mathbb{E}}_{\Pi^N}\sum_{i=1}^N\frac{(V_i\bar{p}_N^{(\epsilon)})^2(s,\bX_s)}{(\bar{p}_N^{(\epsilon)})^2(s,\bX_s)}ds\right)^{1/2}\\
&= N\left(\int_{t_1}^{t_2}\bar{\mathbb{E}}_{\Pi^N}\sum_{i=1}^N |\psi_i(s)|^2ds\right)^{1/2} \left(\int_{t_1}^{t_2}\sum_{i=1}^N\int_{\mathbb{R}^N}\frac{(V_i\bar{p}_N^{(\epsilon)})^2(s,x)}{(\bar{p}_N^{(\epsilon)})^2(s, x)}\bar{p}_N(s,x)\Phi^N(dx)ds\right)^{1/2}.
\end{align*}
Using the bound in \eqref{Hzeroeq} and monotone convergence theorem in the bound above, we get
\begin{multline}\label{fin3}
\limsup_{n \rightarrow \infty} \limsup_{\epsilon \rightarrow 0}\left\lvert\int_{t_1}^{t_2}\bar{\mathbb{E}}_{\Pi^N}\left(N\sum_{i=1}^N \psi_{i+1}(s) \eta_n(\bX_s)\frac{V_i\bar{p}_N^{(\epsilon)}(s,\bX_s)}{\bar{p}_N^{(\epsilon)}(s,\bX_s)}\right)ds \right\rvert\\
\le \sqrt{C_0}N^{3/2} \left(\int_{t_1}^{t_2} \sum_{i=1}^N\int_{\mathbb{R}^N} \frac{(V_i \bar{p}_N)^2(s,x)}{\bar{p}_N(s,x)}\Phi^N(dx) ds\right)^{1/2}.
\end{multline}
Using \eqref{fin3}, \eqref{align:Ifin} and monotone convergence theorem,
\begin{align}
\limsup_{n \rightarrow \infty} \limsup_{\epsilon \rightarrow 0} T_2^{(\epsilon)}(n) &\le -\frac{N^2}{2}\left(\int_{t_1}^{t_2} \sum_{i=1}^N\int_{\mathbb{R}^N} \frac{(V_i \bar{p}_N)^2(s,x)}{\bar{p}_N(s,x)}\Phi^N(dx) ds\right)\nonumber\\
 &\quad+ \sqrt{C_0}N^{3/2} \left(\int_{t_1}^{t_2} \sum_{i=1}^N\int_{\mathbb{R}^N} \frac{(V_i \bar{p}_N)^2(s,x)}{\bar{p}_N(s,x)}\Phi^N(dx) ds\right)^{1/2}.\label{ttwobound}
\end{align}
Note that by \eqref{align:Ifin} the terms on the right side are finite.
Now, using \eqref{lastone} and \eqref{ttwobound} in \eqref{eq:hnepst1t2} 
we obtain
\begin{align}
-\gamma(\bar p) N &\le -\frac{N^2}{2}\left(\int_{t_1}^{t_2} \sum_{i=1}^N\int_{\mathbb{R}^N} \frac{(V_i \bar{p}_N)^2(s,x)}{\bar{p}_N(s,x)}\Phi^N(dx) ds\right)\nonumber\\
 &\quad+ \sqrt{C_0}N^{3/2} \left(\int_{t_1}^{t_2} \sum_{i=1}^N\int_{\mathbb{R}^N} \frac{(V_i \bar{p}_N)^2(s,x)}{\bar{p}_N(s,x)}\Phi^N(dx) ds\right)^{1/2}. \label{good}
\end{align}
Letting
$$
y\doteq N^{1/2}\left(\int_{t_1}^{t_2} \sum_{i=1}^N\int_{\mathbb{R}^N} \frac{(V_i \bar{p}_N)^2(s,x)}{\bar{p}_N(s,x)}\Phi^N(dx) ds\right)^{1/2},
$$
\eqref{good} can be rewritten as
$
y^2 - 2\sqrt{C_0}y - 2\gamma(\bar p) \leq 0$.
 This in turn implies that $y \leq \gamma_2(\bar p)$
where $\gamma_2(\bar p) = \sqrt{C_0} + \sqrt{C_0 + 2\gamma(\bar p)}$, namely
$$
\left(\int_{t_1}^{t_2} \sum_{i=1}^N\int_{\mathbb{R}^N} \frac{(V_i \bar{p}_N)^2(s,x)}{\bar{p}_N(s,x)}\Phi^N(dx) ds\right)^{1/2} \le \frac{\gamma_2(\bar p)}{N^{1/2}}.
$$
By taking limits $t_1 \downarrow 0$ and $t_2 \uparrow T$ in the above bound, we get
$$
\int_0^T I(\bar{p}_N(s,\cdot))ds=\int_{0}^{T} \sum_{i=1}^N\int_{\mathbb{R}^N} \frac{(V_i \bar{p}_N)^2(s,x)}{\bar{p}_N(s,x)}\Phi^N(dx) ds \le \frac{[\gamma_2(\bar p)]^2}{N}.
$$
For the general case when $\bar{p}_N(0,\cdot)$ is not bounded, approximate $\bar{p}^{(k)}_N(0,\cdot)$ by $\bar{p}^{(k)}_N(0,\cdot)$ as before and let $K$ and $\gamma_1(\bar p)$ be as above \eqref{align:absenk}.
%
%
 The proof of the case when $\bar{p}_N(0,\cdot)$ bounded given above and the bound \eqref{align:absenk} now gives for each $k \ge K$,
$$
\int_{0}^{T} \sum_{i=1}^N\int_{\mathbb{R}^N} \frac{(V_i \bar{p}^{(k)}_N)^2(s,x)}{\bar{p}^{(k)}_N(s,x)}\Phi^N(dx) ds \le \frac{[\gamma_3(\bar p)]^2}{N},
$$
where $\gamma_3(\bar p) = \sqrt{C_0} + \sqrt{C_0 + 2\gamma_1(\bar p)}$. Using lower semicontinuity and Fatou's lemma, we obtain
\begin{align*}
\int_{0}^{T} \sum_{i=1}^N\int_{\mathbb{R}^N} \frac{(V_i \bar{p}_N)^2(s,x)}{\bar{p}_N(s,x)}\Phi^N(dx) ds 
\le \liminf_{k \rightarrow \infty}\int_{0}^{T} \sum_{i=1}^N\int_{\mathbb{R}^N} \frac{(V_i \bar{p}^{(k)}_N)^2(s,x)}{\bar{p}^{(k)}_N(s,x)}\Phi^N(dx) ds \le \frac{[\gamma_3(\bar p)]^2}{N}.
\end{align*}
Finally, we obtain \eqref{INest} from the above bound by noting that the functional $I$ is convex in its argument and hence,
$$
I\left(\frac{1}{T}\int_0^T \bar{p}_N(s,\cdot)ds\right) \le \frac{1}{T}\int_0^T I(\bar{p}_N(s,\cdot))ds \le \frac{[\gamma_3(\bar p)]^2}{NT}.
$$
\end{proof}

\section{Tightness of $\bar\mu^N$ in $C([0,T], \mathcal{M}_S)$} \label{proofmubartight}
In this section, we will prove Lemma \ref{mubartight} which establishes the tightness of $\bar\mu^N$ in $\Omega=C\left([0,T]: \mathcal{M}_S\right)$ when the sequence $\{\psi^N, \Pi^N\}$ is as in Lemma \ref{entlemma}.

\begin{proof}[Proof of Lemma \ref{mubartight}]
Tightness of $\bar\mu^N$ in $\Omega$ will be established by showing the following two equalities:
\begin{equation}
\label{tightness1}\lim_{l\rightarrow \infty} \limsup_{N\rightarrow\infty} \bar{\mathbb{P}}_{\Pi^N} ( \bar\mu^N\notin \Omega^l)=0,
\end{equation}
and for every $\epsilon>0$ and  smooth function $J$ on $S$,
\begin{equation}
\label{tightness2}\lim_{\delta \searrow 0} \limsup_{N\rightarrow\infty}\bar{\mathbb{P}}_{\Pi^N}\left(\sup_{0\leq t, s\leq 1,\\ |t-s|\leq \delta} |\langle J,\bar{\mu}^N(t)\rangle-\langle J,\bar{\mu}^N(s)\rangle|>\epsilon\right) = 0.
\end{equation}
The equation in \eqref{tightness1} gives the tightness of the marginals of $\bar \mu^N$ and \eqref{tightness2} gives an equicontinuity estimate.
Together, (\ref{tightness1}) and (\ref{tightness2})  imply that $\bar\mu^N$ is tight.
\subsection{Proof of (\ref{tightness1})}
Recall that $\{\bar X^N(t)\}$ is given on the probability space $(\bar \clv, \bar \clf, \bar \PP)$ associated with a system
$\sys_{\Pi^N} \doteq (\bar \clv, \bar \clf, \{\bar \clf_t\}, \bar \PP, \bar X^N(0), \bfB^N)$. Further, recall that the probability measure 
$\bar \PP$ is also denoted as $\bar \PP_{\Pi^N}$. By enlarging the space if needed, we can construct a $\bar \clf_0$
measurable $\RR^N$-valued random variable $\bar V^N(0)$ with probability law $\Phi^N$ and construct the controlled process
$\{\bar V^N(t)\}$ on this probability space as $\{\bar X^N(t)\}$ was defined in Section \ref{subsec:varrep} using the
same control processes $\{\psi^N_i\}$. We denote the probability law of $\bar V^N$ on $C([0,T]:\RR^N)$ as
$\bar \QQ_{\Phi^N}$. Also recall the measure $\QQ_{\Phi^N}$ introduced in the proof of Lemma \ref{ltwo}.
For $t \in [0,T]$ and $i= 1, \ldots N$, let $w^i_t: C([0,T]:\RR^N) \to \RR$ be the canonical coordinate process, namely
$w^i_t(\om) = \om_i(t)$, for $\om = (\om_1, \ldots \om_N) \in C([0,T]:\RR^N)$. Let, abusing notation,
$\mu^N(t,d\theta) \doteq \frac{1}{N} \sum_{i=1}^N w^i_t \delta_{i/N}(d\theta).$
We begin by establishing an exponential estimate on $\bar{\mathbb{Q}}_{\Phi^N}(\mu^N\notin\Omega^l).$ 
By the Cauchy-Schwarz inequality
\begin{align*}
\bar{\mathbb{Q}}_{\Phi^N}( \mu^N\notin \Omega^l) &= \int_{C([0,T]:\RR^N)} \mathbb{I}_{\{ \mu^N\notin \Omega^l\}}\frac{d\bar{\mathbb{Q}}_{\Phi^N}}{d\mathbb{Q}_{\Phi^N}} d\mathbb{Q}_{\Phi^N}\\
 &\leq \left[\mathbb{Q}_{\Phi^N}( \mu^N\notin \Omega^l)\right]^{1/2}\left[\int_{C([0,T]:\RR^N)}\left[\left( \frac{d\bar{\mathbb{Q}}_{\Phi^N}}{d\mathbb{Q}_{\Phi^N}}\right)^2\right]d\mathbb{Q}_{\Phi^N}\right]^{1/2}.
\end{align*}
From \eqref{GPVest} recall that for some $C_1,C_2, l_0 \in (0,\infty)$, 
$\mathbb{P}( \mu^N\notin \Omega^l) \leq C_1 e^{-C_2Nl}$ for all  $l\ge l_0$  and  $N \in \NN$.
By Girsanov's theorem and recalling that $\psi^N$ satisfy the bound in \eqref{Hzeroeq}, we have that for some $C_3 \in (0,\infty)$ and all $N\in \NN$
\begin{equation}
\label{RNest} \int_{C([0,T]:\RR^N)} \left[\left( \frac{d\bar{\mathbb{Q}}_{\Phi^N}}{d\mathbb{Q}_{\Phi^N}}\right)^2\right]
d\mathbb{Q}_{\Phi^N}
 \leq e^{C_3N}.
\end{equation}
Thus, combining (\ref{GPVest}) and (\ref{RNest}) we have for all $l\ge l_0$ and $N\in \NN$
$
\bar{\mathbb{Q}}_{\Phi^N}( \mu^N\notin \Omega^l) \leq C_1 e^{N(-C_2l+C_3)}.
$
Assume without loss of generality that $l_0 > 2C_3/C_2$. Then, 
with $C_4 = \frac{1}{2} C_2$  we have  for all $l\ge L$  and $N\in \NN$
\begin{align}
\label{expdec}\bar{\mathbb{Q}}_{\Phi^N}(\mu^N\notin \Omega^l) \leq C_1 e^{-C_4Nl}.
\end{align}

The rest of the proof is the same as \cite{guopapvar}. We give the details for the sake of completeness.  
Let $A$ be the event $\{\mu^N\notin \Omega^l\}$, let $g = \theta \mathbb{I}_A$ where $\theta = \log\left(1+1/\bar{\mathbb{Q}}_{\Phi^N}(A)\right)$.  Applying the chain rule for relative entropy  we have

\begin{equation}\label{Rlaw}
R(\bar{\mathbb{Q}}_{\Pi^N}\|\bar{\mathbb{Q}}_{\Phi^N}) = R(\Pi^N\|\Phi^N) \leq C_0N,
\end{equation}
where $\bar{\mathbb{Q}}_{\Pi^N}$ is as introduced in the proof of Lemma \ref{entlemma} and the inequality is from \eqref{Hzeroeq}.
Therefore using the Donsker-Varadhan variational formula (see for example \cite[Lemma 1.4.3]{dupell4})
\begin{equation*}
\int_{C([0,T]:\RR^N)} g(\om) d \bar{\mathbb{Q}}_{\Pi^N}
 \leq \log \int_{C([0,T]:\RR^N)} e^{g(\om)} d \bar{\mathbb{Q}}_{\Phi^N} + C_0N.
\end{equation*}
By the definition of $g$ and $\theta$ and  (\ref{expdec}) we have
\begin{equation*}
	\bar{\mathbb{P}}_{\Pi^N} ( \bar\mu^N\notin \Omega^l) = 
\bar{\mathbb{Q}}_{\Pi^N}(A) \leq \frac{\log(2)+C_0N}{\log\left(1+1/\bar{\mathbb{Q}}_{\Phi^N}(A)\right)} \leq \frac{\log(2)+C_0N}{\log\left(1+C_1^{-1}e^{C_4Nl}\right)}.
\end{equation*}
Letting $N\rightarrow\infty$ and then $l\to \infty$ we have
$
\lim_{l\to \infty}\limsup_{N\rightarrow\infty}\bar{\mathbb{P}}_{\Pi^N} ( \bar\mu^N\notin \Omega^l) =0
$
which completes the proof of (\ref{tightness1}).
\subsection{Proof of (\ref{tightness2})}
Fix a smooth test function $J$ on $S$. Then
\begin{equation*}
\langle J,\bar{\mu}^N(t)\rangle = \frac{1}{N} \sum_{i=1}^N J\left(\frac{i}{N}\right) \bar{X}_i^N(t), \; t \in [0,T].
\end{equation*}
Recalling the definition of $\bar X^N$ from \eqref{align:controleq} we see that it suffices to show that
\begin{align}
\label{firstineq}&\lim_{\delta \searrow 0} \limsup_{N\rightarrow\infty}\bar{\mathbb{P}}_{\Pi^N}\left(\sup_{0\leq t, s\leq 1, |t-s|\leq \delta} 
\left|\int_s^t \frac{1}{N}\sum_{j=1}^N J''\left(\frac{j}{N}\right)\phi'(\bar{X}_j^N(\sigma))d\sigma\right|>\epsilon\right) = 0,\\
\label{secondineq}&\lim_{\delta \searrow 0} \limsup_{N\rightarrow\infty}\bar{\mathbb{P}}_{\Pi^N}\left(\sup_{0\leq t, s\leq 1, |t-s|\leq \delta} 
\left| \frac{1}{N}\sum_{j=1}^N J'\left(\frac{j}{N}\right)(B_j(t)-B_j(s))\right|>\epsilon\right) = 0, \\
\label{thirdineq}& \lim_{\delta \searrow 0} \limsup_{N\rightarrow\infty}\bar{\mathbb{P}}_{\Pi^N}\left(\sup_{0\leq t, s\leq 1, |t-s|\leq \delta} 
\left|\int_s^t \frac{1}{N}\sum_{j=1}^N J'\left(\frac{j}{N}\right)\psi_j(\sigma)d\sigma\right|>\epsilon\right) = 0.
\end{align}
Proofs of \eqref{secondineq} and \eqref{thirdineq} are straightforward. In particular, (\ref{secondineq}) follows from L\'evy's modulus of continuity theorem and for \eqref{thirdineq} note that 
\begin{align*}
\left|\int_s^t\frac{1}{N}\sum_{j=1}^N J'\left(\frac{j}{N}\right)\psi_j(\sigma)d\sigma\right| 
\leq (t-s)^{1/2} \|J'\|_{\infty} \left(\frac{1}{N}\sum_{j=1}^N \int_0^T |\psi_j(s)|^2 ds \right)^{1/2}
\leq (t-s)^{1/2} \|J'\|_{\infty} C_0,
\end{align*}
 where $C_0$ is as in \eqref{Hzeroeq}.

The proof of 
(\ref{firstineq}) follows by the same argument as in \cite{guopapvar}, however we give the details for completeness. 
Once again we abbreviate $\bar X^N$, $X^N$ as $\bar X$, $X$, respectively.
Since $J''$ is bounded, it suffices to show
\begin{equation*}
\lim_{\delta \searrow 0} \limsup_{N\rightarrow\infty}\bar{\mathbb{P}}_{\Pi^N}\left(\sup_{0\leq t, s\leq 1, |t-s|\leq \delta} 
\int_s^t \frac{1}{N}\sum_{j=1}^N \left|\phi'(\bar{X}_j(\sigma))\right|d\sigma>\epsilon\right) = 0
\end{equation*}
Recall the cutoff function $\phi'_l$ from \eqref{eq:cutoff} and note that 
 (\ref{firstineq}) holds clearly when $\phi'$ is replaced by $\phi'_l$. Thus to prove  (\ref{firstineq}) it suffices to show that
\begin{equation*}
\lim_{l\rightarrow\infty}\limsup_{N\rightarrow\infty} \bar{\mathbb{P}}_{\Pi^N}\left(\int_0^T\frac{1}{N}\sum_{i=1}^N |\phi'_l(\bar X_i(t))-\phi'(\bar X_i(t))| dt >\epsilon\right) = 0 
\end{equation*}
for all $\epsilon>0.$ Note that
\begin{equation*}
\bar{\mathbb{P}}_{\Pi^N}\left(\int_0^T\frac{1}{N}\sum_{i=1}^N |\phi'_l(\bar{X}_i(t))-\phi'(\bar{X}_i(t))| dt >\epsilon\right)
\leq \frac{1}{\epsilon} \int_0^1\frac{1}{N}\sum_{i=1}^N \bar{\mathbb{E}}_{\Pi^N}|\phi'_l(\bar{X}_i(t))-\phi'(\bar{X}_i(t))| dt,
\end{equation*}
and by the Donsker-Varadhan variational formula  and Lemma \ref{entlemma},  for any $0\leq t \leq T$ and any $\gamma>0$
\begin{align*}
\bar{\mathbb{E}}_{\Pi^N}\left(\gamma\sum_{i=1}^N |\phi'_l(\bar{X}_i(t))-\phi'(\bar{X}_i(t))|\right) 
&\leq \log\mathbb{E} \left(\exp\left(\gamma\sum_{i=1}^N|\phi'_l({X}_i(t))-\phi'({X}_i(t))|\right) \right)+ R(\bar{\mathbb{Q}}_{\Pi^N}(t)\| \Phi^N)\\
&\leq N\log\mathbb{E} \left(\exp\left(\gamma|\phi'_l({X}_1(0))-\phi'({X}_1(0))|\right) \right)+ C_0N,
\end{align*}
where the last inequality follows from the stationarity of $\{X(t)\}$ and $C_0$ is as in \eqref{Hzeroeq}.
Dividing by $N\gamma$ we have
\begin{equation*}
\bar{\mathbb{E}}_{\Pi^N}\left(\frac{1}{N}\sum_{i=1}^N |\phi'_l(\bar{X}_i(t))-\phi'(\bar{X}_i(t))|\right) \leq \frac{C_0}{\gamma} + \frac{1}{\gamma}
\log\left(\mathbb{E} \left(\exp(\gamma|\phi'(X_1(0))|)\mathbb{I}_{|\phi'(X_1(0))|>l}\right)+1\right),
\end{equation*}
since $\Phi^N$ is the stationary measure for $X_t.$
Assumption (\ref{sigmaass}) implies that for all $l$ sufficiently large 
\begin{equation*}
\log\left(\mathbb{E} \left(\exp(\gamma |\phi'(X_1(0))|)\mathbb{I}_{|\phi'(X_1(0))|>l}\right)\right)\leq 0.
\end{equation*}
Therfore,
\begin{equation*}
\lim_{l\rightarrow\infty}\limsup_{N\rightarrow\infty} \bar{\mathbb{P}}_{\Pi^N} \left(\int_0^1\frac{1}{N}\sum_{i=1}^N
 |\phi'_l(\bar{X}_i(t))-\phi'(\bar{X}_i(t))| dt > \epsilon\right) \leq \frac{C_0+\log 2}{\gamma\epsilon}.
\end{equation*}
Letting $\gamma\rightarrow\infty$ completes the proof of (\ref{firstineq}) and hence also the proof of (\ref{tightness2}).
\end{proof}

\section{Tightness and Subsequential Limits of $(\bar{L}^N,\nu^N)$} \label{prooflemmas}
In this section, we will prove Lemmas \ref{lbartight}, \ref{limitssame},  \ref{UnifInt}, and  \ref{convtopistar} which establish tightness and characterize subsequential limits of $(\bar{L}^N,\nu^N)$, where $\bar{L}^N$ 
(defined in \eqref{eq:lndxdt}) and $\nu^N = N^{-1} \sum_{i=1}^N \nu^N_i$ (with $\nu^N_i$  defined  as in \eqref{nudef}) are random measures constructed from the initial collection $\{\bar X^N_i(0)\}_{i=1}^N$.
\subsection{Proof of Lemma \ref{lbartight}}
Let $\lambda_N \in \clp(S)$ be defined as $\lambda_N(A) = \frac{1}{N} \sum_{i=1}^N \delta_{i/N}(A)$ for $A \in \clb(S)$.
By (\ref{Rbound}), the convexity of relative entropy, and Jensen's inequality
\begin{align*}
R(\bar{\mathbb{E}}_{\Pi^N}\nu^N\|\Phi\times\lambda_N) \leq \mathbb{E}_{\Pi^N}\left(\frac{1}{N}\sum_{i=1}^N R(\nu_i^N\|\Phi\times\delta_{i/N})\right)
= \mathbb{E}_{\Pi^N}\left(\frac{1}{N}\sum_{i=1}^N R(\mathbf{\Phi}_i^N\|\Phi)\right)\leq C.
\end{align*}
 Since for all $\alpha \in\mathbb{R}$, $\int_{\RR} e^{\alpha y} \Phi(dy) <\infty$ by \eqref{genfun}, we have from the above relative entropy bound that 
$\bar{\mathbb{E}}_{\Pi^N}\nu^N$ is tight sequence in $\clp(\RR \times S)$ (see e.g. Lemma 1.4.3.d in \cite{dupell4}). Consequently $\{\nu^N\}$ is a tight 
sequence of $\clp(\RR \times S)$-valued random variables. Next
we claim that $\bar{\mathbb{E}}_{\Pi^N} \bar L^N = \bar{\mathbb{E}}_{\Pi^N}\nu^N$, from which it will then follow that $\bar L^N$ is a tight sequence of  $\mathcal{P}(\mathbb{R}\times S)$-valued random variables as well.
Let $f$ be a bounded, continuous function on $\mathbb{R}\times S$. Then,
\begin{align*}
\bar{\mathbb{E}}_{\Pi^N}\int_{\mathbb{R}\times S} f(x,\theta)\bar{L}^N(dxd\theta) &= \bar{\mathbb{E}}_{\Pi^N} \frac{1}{N} \sum_{i=1}^N f\left(\bar{X}_i^N(0),\frac{i}{N}\right)\\
&=\bar{\mathbb{E}}_{\Pi^N} \frac{1}{N} \sum_{i=1}^N \int_{\mathbb{R}\times S}f\left(x,\theta\right)\nu_i^N(dxd\theta)\\
&=\bar{\mathbb{E}}_{\Pi^N} \int_{\mathbb{R}\times S}f\left(x,\theta\right)\nu^N(dxd\theta).\\
\end{align*}
This proves the claim and hence completes the proof of the lemma. \hfill \qed

\subsection{Proof of Lemma \ref{UnifInt}}\label{UnifIntSec}
From Lemma \ref{mubartight}, $\{\bar \mu^N\}$ is a tight sequence of $\Om$-valued random variables.
We restrict attention to a subsequence along which $(\bar \mu^N(0,\cdot), \bar L^N)$ converges in distribution to $(\bar \mu(0, \cdot), \bar L)$
in $\clm_S\times \clp(S\times \RR)$. 
It suffices to show that 
\begin{equation}\label{eq:eq500}
	\int_{\mathbb{R}}|x|\bar{L}(dx d\theta) < \infty \mbox{ a.s., }
	\end{equation} and that for all continuous $f: S \rightarrow \mathbb{R}$ 
\begin{equation}\label{eq:eq449}
\limsup_{N\rightarrow\infty}\left| \bar{\mathbb{E}}_{\Pi^N} \int_S \int_\mathbb{R} f(\theta) x \bar L^N(dxd\theta) -\bar{\mathbb{E}}\int_S\int_\mathbb{R} f(\theta) x \bar L(dxd\theta)\right| = 0.
\end{equation}
For $M \in (0,\infty)$, let $g_M(x)\doteq (x\wedge M)\vee (-M)$. Then for every $M$
\begin{equation*}
\int_S\int_\mathbb{R} f(\theta) g_M(x) \bar L^N(dxd\theta) \rightarrow \int_S\int_\mathbb{R} f(\theta) g_M(x) \bar L(dxd\theta),
\end{equation*}
in distribution.
%
Let $\bar L^N_0(dx)$ be the measure on $\mathbb{R}$ defined by
$
\bar L^N_0(dx) \doteq \frac{1}{N}\sum_{i=1}^N \delta_{\bar{X}_i^N(0)}(dx).
$
In order to prove \eqref{eq:eq500} and \eqref{eq:eq449} it then suffices to show that
\begin{equation}
\lim_{M\rightarrow \infty}\limsup_{N\rightarrow\infty}\mathbb{E} \int_{\mathbb{R}} |x| \mathbb{I}_{\{|x|\geq M\}} \bar L^N_0(dx) = 0.\label{book38}
\end{equation}
To prove (\ref{book38}),
we use the inequality that for all $a,b\geq 0, \sigma \geq 1,$
$
ab\leq e^{\sigma a}+ \frac{1}{\sigma}(b\log b - b + 1 ).
$
 Thus, with $\bar{\Phi}^N= \frac{1}{N}\sum_{i=1}^N \bar{\Phi}_i^N$,
\begin{align*}
\bar{\mathbb{E}}_{\Pi^N}\int_{\mathbb{R}} |x| 1_{|x|\geq M} \bar{L}^N_0(dx) &= \bar{\mathbb{E}}_{\Pi^N} \int_{\mathbb{R}} |x|1_{\{|x|\geq M\}} \bar{\Phi}^N(dx) \\
&= 
\int_{\mathbb{R}} |x| 1_{\{|x|\geq M\}} \frac{d\bar{\Phi}^N}{d\Phi}(x) \Phi(dx)\nonumber\\
&\le \int_{\mathbb{R}} e^{\sigma|x|}1_{\{|x|\geq M\}} \Phi(dx) + \frac{1}{\sigma} R(\bar{\Phi}^N\|\Phi)\\
&\le \int_{\mathbb{R}} e^{\sigma|x|}1_{\{|x|\geq M\}} \Phi(dx) + \frac{C}{\sigma},
\end{align*} 
where the last inequality is from \eqref{Rbound}. The equality in \eqref{book38} now follows on    sending first $N$, then $M\to \infty$ and and finally $\sigma$ to $\infty$.
\hfill \qed

\subsection{Proof of Lemma \ref{convtopistar}}\label{pistarsec}

By (\ref{mutight}), 
$
\frac{1}{N}\sum_{i=1}^N R(\bar\Phi_i^N\|\Phi) \leq
R(\pi^*\|\pi_0),$
and therefore by  Lemma \ref{lbartight}, $\bar L^N$ is tight. It thus suffices to show that if $\bar{L}$ is any subsequential limit of $\bar L^N,$ then $\bar{L}=\pi^*.$
For that, it  in turn suffices to show that for every bounded uniformly continuous $f$ on $\RR\times S$
\begin{equation*}
\bar{\mathbb{P}}\left(\int_{\mathbb{R}\times S} f(x,\theta) \bar L(dxd\theta) = \int_{\mathbb{R}\times S}  f(x,\theta) \pi^*(dxd\theta)\right) = 1.
\end{equation*}
Let $\delta >0$ and  $N_0 \in \NN$ be such that  $|f(x,\theta)-f(x,\theta')|<\delta$ whenever $|\theta-\theta'|\leq \frac{1}{N_0}$. Let for $N\ge N_0$
\begin{equation*}
\Delta_i^N \doteq N\int_{\mathbb{R}\times ((i-1)/N, i/N]} f(x,\theta)\pi^*(dxd\theta) -f\left(\bar X_i(0),\frac{i}{N}\right)
\end{equation*}
so that 
\begin{equation*}
\int_{\mathbb{R}\times S} f(x,\theta)\pi^*(dxd\theta) - \int_{\mathbb{R}\times S} f(x,\theta)\bar{L}^N(dxd\theta) = \frac{1}{N}\sum_{i=1}^N \Delta_i^N.
\end{equation*}
By Markov's inequality,
\begin{equation}
\bar{\mathbb{P}}_{\Pi^N}\left(\left|\int_{\mathbb{R}\times S} f(x,\theta)\pi^*(dxd\theta) - \int_{\mathbb{R}\times S} f(x,\theta)\bar L^N(dxd\theta)\right|>\epsilon\right) \leq \frac{1}{N^2\epsilon^2}\bar{\mathbb{E}}_{\Pi^N}\left(\sum_{i,j}\Delta_i^N\Delta_j^N\right).
\label{eq:markineq}
\end{equation}
Note that
$\bar{\mathbb{E}}_{\Pi^N} |\Delta_i^N|^2 \leq 4\|f\|_{\infty}^2$.
 We claim that for $i\neq j$
$
|\bar{\mathbb{E}}_{\Pi^N}\Delta_i\Delta_j|\leq 2\delta \|f\|_{\infty}$.
To see this note that for $i>j$, and with $\clg_i = \sigma\{\bar X_k(0): k \le i\}$,
$
\bar{\mathbb{E}}_{\Pi^N}(\Delta_i^N\Delta_j^N) = \bar{\mathbb{E}}_{\Pi^N}(\Delta_j^N\bar{\mathbb{E}}_{\Pi^N}(\Delta_i^N|\mathcal{G}_{i-1})),
$
and
\begin{align*}
|\bar{\mathbb{E}}_{\Pi^N}(\Delta_i^N|\mathcal{G}_{i-1})| &= \left|N\int_{\mathbb{R} \times ((i-1)/N, i/N]} f(x,\theta)\pi^*(dxd\theta) - \int_\mathbb{R}f\left(x,\frac{i}{N}\right)\Phi_i^N(dx)\right|\\
&=\left|N\int_\mathbb{R}\int_{(i-1)/N}^{i/N} f(x,\theta)\pi_1^*(dx|\theta)d\theta - N\int_\mathbb{R}f\left(x,\frac{i}{N}\right)\int_{(i-1)/N}^{i/N}\pi^*_1(dx|\theta)d\theta\right|\\
&\leq N\int_\mathbb{R}\int_{(i-1)/N}^{i/N} \left|f(x,\theta)-f\left(x,\frac{i}{N}\right)\right|\pi_1^*(dx|\theta)d\theta\\
&\leq \delta.
\end{align*}
The claim now follows since $|\Delta_j^N|\leq 2 \|f\|_{\infty}$.
Using the above observations on the right side of \eqref{eq:markineq}, we have
\begin{equation*}
\frac{1}{N^2\epsilon^2}\bar{\mathbb{E}}_{\Pi^N}\sum_{i,j}\Delta_i^N\Delta_j^N \leq \frac{1}{\epsilon^2N^2} (4N\|f\|_{\infty}^2+2N(N-1)\delta \|f\|_{\infty})
\end{equation*}
Letting $N\rightarrow \infty$ and then $\delta\rightarrow 0$ we have that
\begin{equation*}
\bar{\mathbb{P}}\left(\left|\int_\mathbb{R}\int_0^1 f(x,\theta)\pi^*(dxd\theta) - \int_\mathbb{R}\int_0^1 f(x,\theta)\bar L(dxd\theta)\right|>\epsilon\right) =0.
\end{equation*}
The result follows since $\epsilon>0$ is arbitrary. \hfill \qed

\subsection{Proof of Lemma \ref{limitssame}}
Since the second marginal of $\nu^N$ is the uniform measure on $\{i/N\}_{i=1}^N$, it is clear that the 
 the second marginal of $\nu$ must be $\lambda$ (namely the Lebesgue measure on $S$).
To see that $\bar{L}=\nu,$ let
\begin{equation*}
\Delta_i^N \doteq f\left(\bar{X}_i^N(0), \frac{i}{N}\right) - \int_{\mathbb{R}\times[0,1]} f(x,\theta)\nu_i^N(dxd\theta).
\end{equation*} Note that 
\begin{equation*}
\bar{\mathbb{E}}_{\Pi^N}(\Delta_i^N|\clg_{i-1}) = \int_{\mathbb{R}\times[0,1]} f(x,\theta) \nu_i^N(dxd\theta) - \int_{\mathbb{R}\times[0,1]} f(x,\theta) \nu_i^N(dxd\theta) = 0.
\end{equation*}
Therefore, $\bar{\mathbb{E}}_{\Pi^N}(\Delta_i^N\Delta_j^N)=0$ for all $i\neq j,$ and so
\begin{equation*}
\bar{\mathbb{P}}_{\Pi^N}\left(\left|\int_\mathbb{R}\int_0^1 f(x,\theta)\nu^N(dxd\theta) - \int_\mathbb{R}\int_0^1 f(x,\theta)\bar L^N(dxd\theta)\right|>\epsilon\right) \leq \frac{1}{N^2\epsilon^2}\bar{\mathbb{E}}_{\Pi^N}\left(\sum_{i=1}^N|\Delta_i^N|^2\right)\leq \frac{4\|f\|_{\infty}^2}{N\epsilon^2}.
\end{equation*}
The result follows on sending $N\to \infty$. \hfill \qed

\section{Proof of Proposition \ref{union}}
\label{sec:prfpropun}
In this section we provide the proof of Proposition \ref{union}. Part (a) of the proposition will be proved in Section \ref{sec:prfpropunA} and part (b) will be 
completed in Section \ref{sec:prfpropunB}.

\subsection{Proof of Proposition \ref{union}(a)}
\label{sec:prfpropunA}

Let $\{Z^N\}$ be as in the statement of Proposition \ref{union} and $I$ be a $[0,\infty]$-valued function on $\Om$ such that for all continuous and bounded $g$ on $\Om$, \eqref{eq:lapopen} holds.  We begin with the following lemma.

\begin{lemma}\label{unifdec} Let $M<\infty$ and let $\{h_l\}_{l\in \NN}$ be a sequence of continuous functions  on $\Om$ such that $0\leq h_l(y) \leq M$ for all $y \in \Om$ and all $l$. Then 
\begin{equation*}
\lim_{l\rightarrow\infty} \left(\liminf_{N\rightarrow\infty} \frac{1}{N}\log \mathbb{E} [\exp(-Nh_l(Z^{N}))] - \liminf_{N\rightarrow\infty}\frac{1}{N}\log \mathbb{E} [\exp(-Nh_l(Z^{N}))\mathbb{I}_{\Om^l}(Z^{N})]\right)=0.
\end{equation*}
In particular, for every $\epsilon>0$ there exists $L\in \NN$ such that $l\geq L$ implies 
\begin{equation*}
\liminf_{N\rightarrow\infty}\frac{1}{N}\log \mathbb{E} [\exp(-Nh_l(Z^{N}))\mathbb{I}_{\Om^l}(Z^{N})]\geq
\liminf_{N\rightarrow\infty} \frac{1}{N}\log \mathbb{E} [\exp(-Nh_l(Z^{N}))] -\epsilon.
\end{equation*}
\end{lemma}
\begin{proof}
Since $0\leq h_l(y)\leq M$, 
\begin{align*}
&\frac{1}{N} \log \mathbb{E} [\exp(-Nh_l(Z^{N}))] -\frac{1}{N} \log \mathbb{E}[ \exp(-Nh_l(Z^{N}))\mathbb{I}_{\Om^l}(Z^{N})] \\
&\quad= \frac{1}{N} \log\left(1+\frac{\mathbb{E}[\exp\{-Nh_l(Z^{N})\}\mathbb{I}_{(\Om^{l})^c}(Z^{N})]}{\mathbb{E}[\exp(-Nh_l(Z^{N}))\mathbb{I}_{\Om^{l}}(Z^{N})]}\right)\\
&\quad\leq  \frac{1}{N}\log\left(1+ e^{NM}\frac{\mathbb{P}(Z^{N}\in (\Om^{l})^c)}{\mathbb{P}(Z^{N}\in \Om^{l})}\right)
.
\end{align*}
By (\ref{expdecay}) there exists $L$ such that $l\geq L$ implies 

\begin{equation*}
\limsup_{N\rightarrow\infty} \frac{1}{N}\log(\mathbb{P}(Z^{N}\in (\Om^{l})^c)) \leq -3M.
\end{equation*}
Therfore, for any $l\geq L$ there exists $N_0= N_0(l)$ such that $N\geq N_0$ implies
$
\mathbb{P}(Z^{N}\in (\Om^{l})^c) \leq \exp\{-2NM\},
$
which also implies $\mathbb{P}(Z^{N}\in \Om^{l}) \geq 1-e^{-2M} \doteq C_M>0$.
Thus, for all $l\geq L,$
\begin{align*}
  &\liminf_{N\rightarrow\infty} \frac{1}{N} \log \mathbb{E} [\exp(-Nh_l(Z^{N}))] -\liminf_{N\rightarrow\infty} \frac{1}{N} \log \mathbb{E}[ \exp(-Nh_l(Z^{N}))\mathbb{I}_{\Om^l}(Z^{N})] \\
&\quad \leq  \limsup_{N\rightarrow\infty}  \frac{1}{N} \log(1+ C_M^{-1} e^{-NM})=0.
\end{align*}
The result follows.
\end{proof}
\begin{proof}[Proof of Proposition \ref{union}(a)]
	The proof is adapted from  \cite[Theorem 1.2.3]{dupell4}.
Let $G\subset \Om$ be open. 
Assume $I(G) < \infty$ (otherwise, \eqref{lowergen} is trivially true). Fix  $\epsilon \in (0, \frac{1}{2})$. Let $x\in G$ be  such that $I(x) <I(G) + \epsilon$. Let $M \doteq I(G)+1$. Recall $G^l = G \cap \Om^l$. Since $G^l \nearrow G$, there exists $l_0$ such that $x\in G^l$ for all $l\geq l_0$.  For each such $l$ there exists $\delta_l>0$ such that $B^l\doteq \{y\in\Om^l:d_*(x,y)<\delta_l\}\subset G^l$. Define $h_l$ on $\Om$ by 
\begin{equation*}
h_l(y) \doteq M\min\left\{\frac{d_*(x,y)}{\delta_l},1\right\}, \ \ \ y \in \Om.
\end{equation*}
From Lemma \ref{lem:topprops}(c), $h_l$ is  continuous on $\Om$ and $0\leq h_l(y)\leq M$ for all $y \in \Om$ and all $l.$ Also observe that $h_l(x)=0$ and $h_l(y)=M$ if $d_*(x,y) \ge \delta_l.$ Therefore,

\begin{align*}
\mathbb{E}[\exp(-Nh_l(Z^{N}))\mathbb{I}_{\Om^l}(Z^{N})] &= \mathbb{E}[\exp(-Nh_l(Z^{N}))\mathbb{I}_{B^l}(Z^{N})] + \mathbb{E}[\exp(-Nh_l(Z^{N}))\mathbb{I}_{\Om^l \setminus B^l}(Z^{N})]\\
&\leq  \mathbb{P}(Z^{N}\in B^l) + e^{-NM}.
\end{align*}
Thus, by Lemma \ref{unifdec}, there exists $l_1 \ge l_0$ such that for all $l \ge l_1$, 
\begin{align*}
\max\left\{-M,\liminf_{N\rightarrow\infty} \frac{1}{N} \log \mathbb{P}(Z^{N}\in B^l)\right\} &\geq \liminf_{N \rightarrow \infty} \frac{1}{N}\log \mathbb{E}[\exp(-Nh_l(Z^{N}))\mathbb{I}_{\Om^l}(Z^{N})]\\
&\geq \liminf_{N\rightarrow\infty} \frac{1}{N}\log \mathbb{E} [\exp(-Nh_l(Z^{N}))] - \epsilon\\
&\ge-\inf_{y\in\Om}\{h_l(y) + I(y)\} - \epsilon\\
&\geq -h_l(x)-I(x) - \epsilon\\
&= -I(x) - \epsilon > -I(G)-2\epsilon,
\end{align*}
where the third inequality is from \eqref{eq:lapopen}.
But by assumption, $-M=-I(G)-1<-I(G)-2\epsilon,$ so
\begin{align*}
\liminf_{N\rightarrow\infty} \frac{1}{N} \log \mathbb{P}(Z^{N}\in G) \geq \liminf_{N\rightarrow\infty} \frac{1}{N} \log \mathbb{P}(Z^{N}\in B^l)
> -I(G)-2\epsilon.
\end{align*}
As the choice of $\epsilon \in (0, \frac{1}{2})$ was arbitrary, this proves the lemma.
\end{proof}

\subsection{Proof of Proposition \ref{union}(b)}
\label{sec:prfpropunB}

We begin by showing that  \eqref{upperl} implies  \eqref{uppergen}.

%

\begin{lemma}\label{suffcond}
	Suppose that $\{Z^N\}_{N\in \NN}$ is  a sequence of $\Om$-valued random variables such that \eqref{expdecay} is satisfied.
	Also let $I: \Om \to [0,\infty]$. 
 Suppose that $F$ is a closed set in $\Om$ and there is a $l_0\in \NN$ such that for all $l\ge l_0$
		\begin{equation}
		\label{upperln}\limsup_{N\rightarrow\infty} \frac{1}{N}\log \mathbb{P}(Z^N\in F^l) \le -I(F^l).
		\end{equation}
		Then
		\begin{equation}
		\label{upperlngen}\limsup_{N\rightarrow\infty} \frac{1}{N}\log \mathbb{P}(Z^N\in F) \le -I(F).
		\end{equation}
\end{lemma}
\begin{proof}
	Since for any $A\subset \Om$, $A^l \uparrow A$ as $l \to \infty$, 
 $\lim_{l\rightarrow\infty}I(A^l)=I(A)$. 
Let $F$ be a closed set in $\Om$ satisfying \eqref{upperln} for $l \ge l_0$ for some $l_0 \in \NN$.
Then
\begin{align*}
	\limsup_{N\rightarrow\infty} \frac{1}{N}\log \mathbb{P}(Z^N\in F) &\le 
\max\left\{\limsup_{N\rightarrow\infty}\frac{1}{N}\log \mathbb{P}(Z^N\in F^l), \limsup_{N\rightarrow\infty}\frac{1}{N}\log \mathbb{P}(Z^N\in \Om\setminus \Om^l)\right\}\\
&\le \max\left\{-I(F^l), \limsup_{N\rightarrow\infty}\frac{1}{N}\log \mathbb{P}(Z^N\in \Om\setminus \Om^l)\right\}.
\end{align*}
Sending $l\to \infty$,
$\limsup_{N\rightarrow\infty} \frac{1}{N}\log \mathbb{P}(Z^N\in F) \le \max\left\{-I(F), -\infty\right\} = -I(F).$
%
The result follows.
\end{proof}

We now complete the proof of part (b) of Proposition \ref{union}. Once more, the proof is adapted from  \cite[Theorem 1.2.3]{dupell4}. Let $F$ be closed set in $\Om$.
 By Lemma \ref{suffcond}, it suffices to show that (\ref{upperl}) holds for all $l$.
Fix $l \in \NN,$  and let $\varphi(\mu) \doteq \mathbb{I}_{(F^{l})^{c}}(\mu)\cdot \infty$ so that for all $N\in \NN$, $e^{-N\varphi(\mu)}=\mathbb{I}_{F^l}(\mu)$. 
For $j\in \NN$ let $h_j(\mu) \doteq j\min\{d_*(\mu,F^l),1\}$. From Lemma  \ref{lem:topprops}(c) $h_j$ is  a continuous function on $\Om$. Clearly $0\leq h_j(\mu) \leq j$ and  $h_j(\mu)\le \varphi(\mu)$ for all $\mu \in \Om$.
Therefore, for each fixed $j,$
\begin{align*}
\limsup_{N\rightarrow\infty} \frac{1}{N} \log \mathbb{P}(Z^N\in F^l) &= \limsup_{N\rightarrow\infty}\frac{1}{N} \log \mathbb{E} \exp(-N\varphi(Z^N))\\
&\leq \limsup_{N\rightarrow\infty} \frac{1}{N}\log\mathbb{E} \exp(-Nh_j(Z^N))\\
&= -\inf_{\mu\in\Om} \{h_j(\mu)+I(\mu)\},
\end{align*}
where the last equality holds since $h_j$ is a bounded and continuous function on $\Om$.
Thus it suffices to  show that 
\begin{equation}
\label{jlim}\liminf_{j\rightarrow\infty} \inf_{\mu\in \Om} \{h_j(\mu)+I(\mu)\} \geq   I(F^l).
\end{equation}
Suppose that (\ref{jlim}) does not hold. Then there exists $M<\infty$  such that 
\begin{equation*}
\liminf_{j\rightarrow\infty} \inf_{\mu\in\Om} \{h_j(\mu) + I(\mu)\} < M < I(F^l).
\end{equation*}  
Therefore there exists an infinite subsequence of $j$ such that
$
\inf_{\mu\in\Om} \{h_j(\mu) + I(\mu)\} < M, 
$ 
and for each $j$ along this subsequence there exists $\mu_j \in \Om$ such that 
$
 h_j(\mu_j) + I(\mu_j) < M.
$
 Note that $d_*(\mu_j,F^l)\rightarrow 0$  as $j\rightarrow \infty$ along the chosen subsequence since otherwise $h_j(\mu_j)$ would diverge to  $\infty$. Therefore, there exists a sequence $\nu_j$ of points in $F^l$ such that $d_*(\mu_j,\nu_j)\rightarrow0$ along the subsequence. By assumption the set $\{x\in \Om: I(x)\le M\}$ is compact. Thus
 we can extract a further subsequence of  $\mu_j$ along which $\mu_j$ converges in $\Om$ to some $\mu^*$ satisfying $I(\mu^*)\leq M$. From Lemma \ref{lem:topprops}(b) (part (iii)) we now have that  $d_*(\mu_j,\mu^*)\rightarrow 0$. Therefore, $d_*(\nu_j,\mu^*)\rightarrow 0$. Since $F^l$ is closed in $\Om^l$, this implies that $\mu^*\in F^l$. But, $I(\mu^*) \leq M < I(F^l)$, which is a contradiction.
Thus (\ref{upperl}) holds. \hfill \qed

\section{Existence, uniqueness and continuity of solutions to (\ref{ratepde})}\label{weakuniquesec}

In this section we provide the proofs of Lemmas \ref{existence} and \ref{continu}.

\subsection{Proof of Lemma \ref{existence}}
\begin{proof}
Let 
\begin{equation*}
\psi_i^N(t) \doteq \sum_{j=1}^N u\left(\frac{jT}{N},\frac{i+1}{N}\right)\mathbb{I}_{(jT/N,(j+1)T/N]}(t),
\end{equation*}
and 
\begin{equation*}
u_N(t,\theta) \doteq \sum_{i=1}^N \psi_i^N(t) \mathbb{I}_{(i/N,(i+1)/N]}(\theta),
\end{equation*}
so that 
\begin{equation*}
u_N(t,\theta) = \sum_{i,j=1}^N u\left(\frac{jT}{N},\frac{i}{N}\right) \mathbb{I}_{(jT/N,(j+1)T/N]}(t)\mathbb{I}_{(i/N,(i+1)/N]}(\theta).
\end{equation*}
Note that the sequence $\{\psi^N_i\}$ satisfies \eqref{eq:defnsimpcont} and \eqref{controlas}. It also satisfies the first inequality in \eqref{Hzeroeq} for some $C_0 \in (0,\infty)$.
Observe that $u_N$ converges to $u$ in $L^2$ since $u$ is uniformly continuous. By  \cite[Lemma 6.2.3.g]{dupell4}, we can obtain a probability measure $\pi\in \mathcal{P}(\mathbb{R} \times S)$ whose second marginal is the uniform measure on $S$ and which satisfies the following:
\begin{itemize}
\item[(i)]
$\int_{\mathbb{R}}x\pi_1(dx|\theta) = m_0(\theta)$,
where $\pi(dxd\theta)=\pi_1(dx|\theta)d\theta$ is the atomization of $\pi$,
\item[(ii)] $R(\pi_1(\cdot | \theta) || \Phi(\cdot)) = h(m_0(\theta))$ for each $\theta \in S$.
\end{itemize}
As in (\ref{muchoice}), let
\begin{equation*}
\bar\Phi_i^N(dx) = N\int_{(i-1)/N}^{i/N}\pi_1(dx|\theta)d\theta,
\end{equation*}
for $1\leq i \leq N$ and let $\Pi^N(dx) = \bar\Phi_1^N(dx_1)\ldots\bar\Phi_N^N(dx_N).$ By the calculations leading to \eqref{mutight}, $R(\Pi^N || \Phi^N) \le NR(\pi || \pi_0)$ where from (ii) above
$R(\pi || \pi_0) = \int_S h(m_0(\theta))d\theta < \infty$.
Therefore, by Lemma \ref{mubartight}, 
 $\bar{\mu}^N$ constructed as in Section \ref{subsec:varrep} using $\Pi^N$ and $\{\psi^N_i\}$ is tight and consequently we can find a subsequence along which $\bar\mu^N$ converges to some limit $\bar\mu.$ Thus, by Theorem  \ref{HL}, 
$\bar\mu(t,d\theta)$ has a density $m(t,\theta)$, namely $\mu(t,d\theta)=m(t,\theta)d\theta$ for a.e. $t$ which solves (\ref{ratepde}) with the given $u$ and satisfies integrability conditions (\ref{hl1}), (\ref{hl2}), and (\ref{TVbound}).
\end{proof}

\subsection{Proof of Lemma \ref{continu}}
\begin{proof}
For any $\epsilon>0$ and any bounded, Lipschitz function $J$ on the unit circle, there exists a smooth function $\tilde J$ such that $\|J-\tilde{J}\|_\infty\leq \epsilon$ and $\|\tilde J\|_L \le \operatorname \|J\|_L$. 
Therefore, since $\sup_{0\leq t\leq T} \|m_i(t,\cdot)\|_1<\infty$, it suffices to restrict attention to smooth test functions, $J$, in which case, $\|J\|_L = \|J'\|_{\infty}$.
 
 If $m(t,\theta)$ is any weak solution to (\ref{ratepde}) with $\sup_{0\leq t\leq T} \|m(t,\cdot)\|_1<\infty$,
then for all smooth functions $J$ on $S,$
\begin{equation}
\label{weak}\int_S J(\theta) m(t,\theta)d\theta - \int_S J(\theta)m(0,\theta)d\theta = \int_0^t\int_S J''(\theta)h'(m(s,\theta)) d\theta ds + \int_0^t\int_S J'(\theta)u(s,\theta)d\theta ds.
\end{equation}
Observe, by setting $J(\theta)\doteq1$, that $\int_S m(t,\theta) d\theta = \int_S m(0,\theta) d\theta \doteq a$ for all $0\leq t \leq T.$ Let $\tilde{m} \doteq m-a$  and let $\tilde{M}$ be defined as

\begin{equation}
\label{dv}  \tilde{M}(t,\theta) = \int_0^{\theta} \tilde{m}(t,y) dy, \; 0 \le \theta \le 1.
\end{equation}
Integrating by parts and using (\ref{dv}),
we have that 
\begin{equation}\label{l2eq}
\int_S J'(\theta) \tilde{M}(t,\theta)d\theta - \int_S J'(\theta)\tilde{M}(0,\theta)d\theta = \int_0^t\int_S J'(\theta)[h'(m(s,\theta))]_{\theta} d\theta ds + \int_0^t\int_S J'(\theta)u(s,\theta)d\theta ds.
\end{equation}
For any smooth function $\xi$ on $S$, we can use $J(\theta) = \int_0^{\theta}\xi(y)dy - \int_S\xi(y)dy$ in the above equation and conclude that \eqref{l2eq} holds for any smooth function $\xi$ on $S$ in place of $J'$. For any $g \in L^2(S : \mathbb{R})$, we can approximate $g$ by smooth functions in $L^2(S : \mathbb{R})$ and use the fact that $\tilde{M}(t,\cdot), [h'(m(t,\cdot))]_{\theta}$ and $u(t, \cdot)$ are in $L^2(S : \mathbb{R})$ for a.e. $t$ to conclude that \eqref{l2eq} holds for any function $g \in L^2(S : \mathbb{R})$ in place of $J'$. Thus, we have for each $t \in [0,T]$,
\begin{equation*}
\tilde{M}(t,\theta) = \tilde{M}(0,\theta) + \int_0^t \left([h'(m(s,\theta))]_{\theta} + u(s,\theta)\right)ds
\end{equation*}
for a.e. $\theta$. From this equality, we conclude that $t\rightarrow \tilde{M}(t,\cdot)$ is differentiable as a map into $L^2(S : \mathbb{R})$ in the weak sense (see definition in  \cite[Section 5.9.2]{evans}) and
\begin{equation}
\label{partialv}\partial_t \tilde{M} = \frac{1}{2}[h'(m)]_\theta - u.
\end{equation}
It then follows that the map $ t \rightarrow \|\tilde{M}(t,\cdot)\|_2^2$ is absolutely continuous and
\begin{equation}\label{partform}
\partial_t\|\tilde{M}(t,\cdot)\|_2^2 = \int_S 2 \tilde{M}(t,\theta) \partial_t \tilde{M}(t,\theta)d\theta, \ \ \text{a.e. } t \in [0,T].
\end{equation}

Now let $m_1$ and $m_2$ be as in the statement of the lemma and let $m_3 \doteq m_1-m_2$. Define $\tilde{m}_i$ and $\tilde{M}_i$ in a manner analogous to above and let $\tilde M_3 = \tilde M_1- \tilde M_2$. Then
$\tilde M_3$ solves
\begin{equation*}
\partial_t \tilde{M}_3 = \frac{1}{2}[h'(m_1)]_\theta - \frac{1}{2}[h'(m_2)]_\theta - (u_1-u_2).
\end{equation*}
Using \eqref{partform}, for a.e. $t$
\begin{align*}
\partial_t \|\tilde{M}_3(t,\cdot)\|_2^2 &= \int_S 2 \tilde{M}_3(t,\theta) \partial_t \tilde{M}_3(t,\theta)d\theta \\ 
&=\int_S \tilde{M}_3(t,\theta)([h'(m_1(t,\theta))- h'(m_2(t,\theta))]_\theta -2(u_1(t,\theta)-u_2(t,\theta)))d\theta\\
&= -\int_S \partial_\theta \tilde{M}_3(t,\theta)[h'(m_1(t,\theta))- h'(m_2(t,\theta))] d\theta -2\int_S \tilde{M_3}(t,\theta)(u_1(t,\theta)-u_2(t,\theta)) d\theta\\
&=-\int_S (h'(m_1(t,\theta))- h'(m_2(t,\theta)))(m_1(t,\theta)-m_2(t,\theta)) d\theta\\
&\qquad -2\int_S \tilde{M_3}(t,\theta)(u_1(t,\theta)-u_2(t,\theta)) d\theta\\
&\leq 2\int_S \tilde{M_3}(t,\theta)(u_1(t,\theta)-u_2(t,\theta))d\theta\\
&\leq2 \|\tilde{M}_3(t,\cdot)\|_2\|u_1(t,\cdot)-u_2(t,\cdot)\|_2.
\end{align*}
where the inequality in the next to last line is from the  convexity of $h$ and the inequality in the last line is by Cauchy-Schwarz inequality.
Thus for $t\in [0,T]$
\begin{align*}
\sup_{0\le s \le t} \|\tilde{M}_3(s,\cdot)\|_2^2 &\le 2 \int_0^t \sup_{0\le r \le s }\|\tilde{M}_3(r,\cdot)\|_2\|u_1(s,\cdot)-u_2(s,\cdot)\|_2 ds\\
&\le \int_0^t \sup_{0\le r \le s }\|\tilde{M}_3(r,\cdot)\|_2^2 ds + \|u_1-u_2\|_2^2.
\end{align*}
By Gronwall's inequality,
$$\sup_{0\le s \le T} \|\tilde{M}_3(s,\cdot)\|_2 \le e^{T/2}\|u_1-u_2\|_2.$$
%
%
Therefore, for all $J\in C^{\infty}(S)$ 
\begin{align*}
\left|\int_S J(\theta) (m_1(t,\theta) -m_2(t,\theta))d\theta \right| &= \left| \int_S J'(\theta) \tilde{M}_3(t,\theta)d\theta\right|\\
&\leq \|J'\|_\infty \|\tilde{M}_3(t,\cdot)\|_2\\
&\leq e^{T/2}\|J'\|_\infty \|u_1-u_2\|_2. 
\end{align*}
This completes the proof of part (i) of the lemma.

Suppose $u_n$ is a sequence of smooth functions that converges to $u$ in $L^2([0,T] \times S)$ and let $\mu_n, \mu$ be the signed measures associated to $u_n, u$ respectively as defined in statement of the lemma. 
Then from (i)  $d_{*}(\mu_n,\mu)\to 0$ as $n\to \infty$. To complete the proof of (ii) it suffices to show that $\mu_n$ is uniformly bounded in the total variation norm.  Suppose otherwise, 
then there is a subsequence (labeled again as $n$) and $t_n \in [0,T]$ such that $||\mu_n(t_n)||_{TV}$ is unbounded. By the uniform boundedness principle, there exists a continuous function $f$ on $S$ such that 
$|\int_S f(\theta) \mu_n(t_n, d\theta)|$ is unbounded. Without loss of generality, assume $|\int_S f(\theta) d\mu_n(t_n, d\theta)| \rightarrow \infty$ as $n \rightarrow \infty$.
As in the proof of Lemma \ref{existence}, associated with the function $m_0$ on $S$, we can find $\pi \in \clp(\RR\times S)$ such that $\pi(dxd\theta)=\pi_1(dx|\theta)d\theta$, and
$R(\pi_1(dx | \theta) || \Phi(dx)) = h(m_0(\theta))$, $\int_{\mathbb{R}}x\pi_1(dx|\theta) = m_0(\theta)$, for each $\theta \in S$. Define $\bar \Phi^N_i$ by \eqref{muchoice}
and $\{\psi^{N,n}_i\}$ by \eqref{psichoice} on replacing $u$ with $u_n$. Let $\Pi^N(dx) = \bar\Phi_1^N(dx_1)\ldots\bar\Phi_N^N(dx_N)$. Using these $\{\psi^{N,n}_i\}$ and $\Pi^N$ define $\bar \mu_{n}^N$ as
$\bar \mu^N$ was defined in Section \ref{subsec:varrep}.
%
%
%
%
%
For each fixed $n$, by the same argument used in \eqref{mutight}, the paragraph following it, and the uniqueness established in part (i) of the current lemma, $\mu_{n}^N$ converges weakly to $\mu_n$ as $\Omega$-valued random variables as $N \rightarrow \infty$. In particular, for each fixed $n$, $\int_S f d\mu_{n}^N(t_n, d\theta)$ converges in probability (since the limit is non-random) to $\int_S f(\theta) \mu_n(t_n, d\theta)$ as $N \rightarrow \infty$. Choose $N_n<\infty$ such that
$$
\bar{\mathbb{P}}_{\Pi_{N_n}}\left(\left|\int_S f d\mu_{N_n,n}(t_n) - \int_S f d\mu_n(t_n)\right| > 1\right) < \frac{1}{2}.
$$
Therefore, as $|\int_S f d\mu_n(t_n)| \rightarrow \infty$, for any $M>0$ we can find $n_M$ such that for all $n \ge n_M$,
$$
\bar{\mathbb{P}}_{\Pi_{N_n}}\left(\left|\int_S f d\mu^{N_n}_n(t_n)\right| > M\right) \ge \frac{1}{2}.
$$
But by the uniform $L^2$-boundedness of $u_n$, the fact that
$$
\frac{1}{N} R(\Pi^{N}  || \Phi^N) \le  R(\pi || \pi_0) = \int_S h(m_0(\theta)) d\theta <\infty,
$$
and Lemma $\ref{mubartight}$, we have that the collection $\{\mu^{N_n}_n\}_{n \ge 1}$ is tight (as a sequence of $\Omega$-valued random variables).
Thus we have a contradiction and therefore  $\{\mu_n\}_{n \ge 1}$ is uniformly bounded in total variation norm.
\end{proof}

\section{Proof of Lemma \ref{lem:topprops}}
\label{sec:pflemtop}
To see part (a), consider the new Polish space $(\tilde S, d)$ with $\tilde S = S \cup {P}$ where $P$ is an external point with $d(x,P)=1$ for all $x \in S$ and the restriction of $d$ to $S$ is the intrinsic metric on $S$. Suppose a sequence of $\{\mu_n\} \subset \mathcal{M}^l_S$ converges to $\mu$ weakly. The we must have $\|\mu\|_{TV} \le \sup_n \|\mu_n\|_{TV} \le l$ and so $\mu \in \clm_S^l$.
Consider the ``balanced" measures
\begin{align}\label{bal}
\tilde \mu_n^+ = \mu_n^+ + \left(l - \mu_n^+(S)\right)\mathbb{I}_{\{P\}},\; \; 
\tilde \mu_n^- = \mu_n^- + \left(l - \mu_n^-(S)\right)\mathbb{I}_{\{P\}}.
\end{align}
Note that $\tilde \mu_n^{\pm}$ are finite (nonnegative) measures with  total mass $l$ for each $n$. By the compactness of $\tilde S$, the collection $\tilde \mu_n^{\pm}$ is tight and thus any subsequence of $\tilde \mu_n^{\pm}$ has a further subsequence $\tilde \mu_{n_k}^{\pm}$ that converges weakly on $\tilde S$ to respective measures $\tilde \mu^{\pm}$ with total mass $l$. As the restriction of $\tilde \mu_{n_k}^{\pm}$ to $S$ is $\mu_{n_k}^{\pm}$, $\mu = \tilde \mu^+ \vert_{S} -  \tilde \mu^- \vert_{S}$. Furthermore, as any bounded Lipschitz $f$ on $S$ with $\|f\|_{BL} \le 1$ can be extended to a bounded Lipschitz $\tilde f$ on $\tilde S$ with $\|\tilde f\|_{BL} \le 1$ by assigning $\tilde f(P)=0$, we have
\begin{align*}
\sup_{f\in BL_1(S)} \left|\int_S f d\mu_{n_k} - \int_S f d\mu\right| \le \sup_{\tilde f\in BL_1(\tilde S)} \left|\int_{\tilde S}\tilde f d \tilde \mu_{n_k}^+ - \int_{\tilde S} \tilde f d\tilde \mu^+\right| + \sup_{\tilde f\in BL_1(\tilde S)} \left|\int_{\tilde S}\tilde f d \tilde \mu_{n_k}^- - \int_{\tilde S} \tilde f d\tilde \mu^-\right| \rightarrow 0
\end{align*}
as $k \rightarrow \infty$. Thus, $\sup_{f\in BL_1(S)} |\langle f, \mu_{n} - \mu\rangle| \rightarrow 0$ as $n \rightarrow \infty$. 

Conversely, suppose $\sup_{f\in BL_1(S)} |\langle f, \mu_{n} - \mu\rangle| \rightarrow 0$ as $n \rightarrow \infty$ and $\mu_n \in  \clm_S^l$ for all $n$. Define the measures $\tilde \mu_n^{\pm}$ on $\tilde S$ as before. For any subsequence of $\tilde \mu_n^{\pm}$, obtain a further subsequence $\tilde \mu_{n_k}^{\pm}$ converging weakly to $\tilde \mu^{\pm}$. Set $\tilde \mu = \tilde \mu^+ - \tilde \mu^-$. We will also denote by $\tilde \mu$ the restriction of this measure onto $S$. As for any continuous function $f$ on $S$, its extension $\tilde f$ onto $\tilde S$ obtained by defining $\tilde f(P)=0$ remains continuous on $\tilde S$, we conclude that $\mu_{n_k}$ converge weakly to $\tilde \mu$ as measures on $S$.  Since weak convergence is equivalent to bounded Lipschitz convergence for non-negative measures of total mass $l>0$, we have
\begin{align*}
\sup_{f\in BL_1(S)} \left|\int_S f d\mu_{n_k} - \int_S f d\tilde \mu\right| \le \sup_{\tilde f\in BL_1(\tilde S)} \left|\int_{\tilde S}\tilde f d \tilde \mu_{n_k}^+ - \int_{\tilde S} \tilde f d\tilde \mu^+\right| + \sup_{\tilde f\in BL_1(\tilde S)} \left|\int_{\tilde S}\tilde f d \tilde \mu_{n_k}^- - \int_{\tilde S} \tilde f d\tilde \mu^-\right| \rightarrow 0
\end{align*}
as $k \rightarrow \infty$. As $\sup_{f\in BL_1(S)}|\langle f, \mu_{n} - \mu\rangle| \rightarrow 0$, we have $\int_S f d\mu = \int_S f d\tilde \mu$ for all bounded Lipschitz functions $f$ on $S$
and hence, $\mu = \tilde \mu$. Hence, $\mu_{n_k}$ converges weakly to $\mu$. Since the choice of subsequence is arbitrary, the whole sequence $\mu_{n}$ converges weakly to $\mu$. 
Also, $\mu \in \clm_S^l$.
This establishes equivalence of weak convergence and bounded Lipschitz convergence for measures in $\mathcal{M}^l_S$. 

We now prove (b). In order to prove (i) it suffices to show that for every $f \in C(S)$ and $\eps>0$ 
$$F \doteq \{\tilde \mu \in \Om: \sup_{0\le t \le T} |\lan \tilde \mu(t), f\ran  - \lan \mu(t), f\ran| \ge \eps\}$$
is closed in $\Om$.
Suppose for some $l>0$, $\tilde \mu_n \in F^l = F\cap \Om_l$ and $\tilde \mu_n \to \tilde \mu$ in $\Om^l$.
Then we must have from part (a) that $\sup_{0\le t \le T} |\lan \tilde \mu_n(t), f\ran  - \lan \tilde \mu(t), f\ran| \to 0$. This shows that $\tilde \mu \in F \cap \Om^l$.

For part (ii) note that by part (i) and uniform boundedness principle, for some $l \in (0,\infty)$, $\mu_n \in \Om^l$ for all
$n$. Thus $\mu \in \Om^l$ as well. 
Finally, part (iii)  now follows from noting that from the definition of the direct limit topology, for every $\eps>0$, and $l'>0$
$$G^{l'}\doteq \{\tilde \mu \in \Om: d_*(\tilde \mu, \mu) <\eps\} \cap \Om^{l'} \mbox{ is open in } \Om^{l'}$$
and since $\mu_n\to \mu$ we must have $\mu_n \in \{\tilde \mu \in \Om: d_*(\tilde \mu, \mu) <\eps\}$ for large $n$.
Therefore, part (iii) is now a consequence of part (ii).

Finally consider (c). Suppose $\mu_n \to \mu$ in $\Om$. From (b) (ii) there exists $l'> l$ such that $\mu_n, \mu \in \Om^{l'}$ and $d_*(\mu_n,\mu)\to 0$.
Also note that $F^l$ is closed in $\Om^{l'}$. Thus since $h(\mu) = d_{*}(\mu, F^{l})$ for $\mu \in \Om^{l'}$ and the right side is a continuous function on $\Om^{l'}$, we have that $h(\mu_n)\to h(\mu)$ as $n\to \infty$. \hfill \qed

\bibliographystyle{plain}

\vspace{\baselineskip}

\textsc{\noindent S. Banerjee\newline Department of Statistics and
Operations Research\newline
University of North Carolina\newline  Chapel Hill,
NC 27599, USA\newline email:
sayan@email.unc.edu
 \vspace{\baselineskip}}

\textsc{\noindent A. Budhiraja \newline Department of Statistics and
Operations Research\newline University of North Carolina\newline Chapel Hill,
NC 27599, USA\newline email: budhiraj@email.unc.edu \vspace{\baselineskip} }

\textsc{\noindent M. Perlmutter\newline Department of Computational Mathematics, Science and Engineering\newline
Michigan State University\newline East Lansing, MI 48824, USA\newline email:
perlmut6@msu.edu \vspace{\baselineskip} }
\end{document}